\documentclass[preprint,12pt]{elsarticle}

\usepackage{amssymb,amsthm}
\usepackage{latexsym,amsmath,epsfig,amsfonts,graphicx,mathrsfs,color}%,mathdots}
\usepackage{bm}
\usepackage{pstricks,pst-node,pst-tree}
\usepackage{tikz}
\usetikzlibrary{arrows}

\newcommand{\vc}[1]{\boldsymbol{#1}}

\newcommand{\soph}[1]{{#1}}
\usepackage{dsfont}

\newcommand{\registered}
   {{\scriptsize \ooalign{\hfil\raise0.07ex\hbox{\scriptsize \sc r}\hfil%
              \crcr\mathhexbox20D}}}

\newtheorem{assumption}{Assumption}
\newtheorem{theorem}{Theorem}
\newtheorem{lemma}{Lemma}
\newtheorem{definition}{Definition}
\newtheorem{corollary}{Corollary}

\newtheorem{proposition}{Proposition}
\newcommand{\nc}{\newcommand}
\nc{\ds}{\displaystyle}
\nc{\mbZ}{\mathbb Z}
\nc{\mbQ}{\mathbb Q}
\nc{\mbR}{\mathbb R}
\nc{\mbC}{\mathbb C}
\nc{\mbN}{\mathbb N}
\nc{\mbE}{{\bf E}}
\nc{\mbP}{{\bf P}}

\nc{\PH}{\emph{PH} }
\nc{\ME}{\emph{ME} }
\nc{\LST}{\emph{LST} }
\nc{\rank}{\mbox{rank\hspace{1pt}}}

\numberwithin{equation}{section}

\numberwithin{figure}{section}

\usepackage[english]{babel}

\journal{Stochastic Processes and their Applications}

\begin{document}

\begin{frontmatter}

%% Title, authors and addresses

%% use the tnoteref command within \title for footnotes;
%% use the tnotetext command for theassociated footnote;
%% use the fnref command within \author or \address for footnotes;
%% use the fntext command for theassociated footnote;
%% use the corref command within \author for corresponding author footnotes;
%% use the cortext command for theassociated footnote;
%% use the ead command for the email address,
%% and the form \ead[url] for the home page:
%% \title{Title\tnoteref{label1}}
%% \tnotetext[label1]{}
%% \author{Name\corref{cor1}\fnref{label2}}
%% \ead{email address}
%% \ead[url]{home page}
%% \fntext[label2]{}
%% \cortext[cor1]{}
%% \address{Address\fnref{label3}}
%% \fntext[label3]{}

\title{A pathwise approach to the extinction of branching processes with countably many types}

%% use optional labels to link authors explicitly to addresses:
%% \author[label1,label2]{}
%% \address[label1]{}
%% \address[label2]{}

\author[peter]{Peter Braunsteins}
\ead{p.braunsteins@student.unimelb.edu.au}
\author[geoffrey]{Geoffrey Decrouez}
\ead{ggdecrouez@hse.ru}
\author[peter,sophie]{Sophie Hautphenne\corref{cor1}}
\ead{sophiemh@unimelb.edu.au}

\address[peter]{The University of Melbourne, Australia}
\address[geoffrey]{National Research University, Higher School of Economics, Russia}
\address[sophie]{Swiss Federal Institute of Technology Lausanne, Switzerland}
\cortext[cor1]	{Corresponding author}

\begin{abstract}
We consider the extinction events of Galton-Watson processes with countably infinitely many types. In particular, we construct truncated and augmented Galton-Watson processes with finite but increasing sets of types. A pathwise approach is then used to show that, under some sufficient conditions, the corresponding sequence of extinction probability vectors converges to the global extinction probability vector of the Galton-Watson process with countably infinitely many types. Besides giving rise to a {family of new} iterative methods for computing the global extinction probability vector, our approach paves the way to new global extinction criteria for branching processes with countably infinitely many types.

\end{abstract}

\begin{keyword}
multitype branching process \sep extinction probability  \sep pathwise approach \sep extinction criterion
\MSC[2010] 60J80 \sep 60J05 \sep 60J22 \sep 65H10 
%% keywords here, in the form: keyword \sep keyword

%% PACS codes here, in the form: \PACS code \sep code

%% MSC codes here, in the form: \MSC code \sep code
%% or \MSC[2008] code \sep code (2000 is the default)

\end{keyword}

\end{frontmatter}

%% \linenumbers

%% main text

\tikzset{
  treenode/.style = {align=center, inner sep=0pt, text centered,
    font=\sffamily},
  arn_n/.style = {treenode, circle, black, font=\sffamily\bfseries, draw=black,
    fill=white, text width=1.5em},% arbre rouge noir, noeud noir
  arn_r/.style = {treenode, circle, red, draw=red, 
    text width=1.5em, very thick},% arbre rouge noir, noeud rouge
  arn_x/.style = {treenode, circle, white, font=\sffamily\bfseries, draw=black,
    fill=black, text width=1.5em}% arbre rouge noir, nil
}

\section{Introduction}

Multitype Galton-Watson branching processes (MGWBPs) are stochastic models describing the evolution of a population of individuals who live one unit of time and give birth to a random number of offspring that may be of various types. Each type may have a different progeny distribution, and individuals behave independently of each other. These processes have been studied extensively during the last decades; classical reference books include Harris \cite{Har02}, Mode \cite{mode}, Athreya and Ney \cite{athreya_ney}, and Jagers \cite{jagers}. MGWBPs have numerous applications and have been used to model important problems arising in biology, ecology, physics and even computer science. Recent books with special emphasis on applications are Axelrod and Kimmel \cite{kimmel}, and Haccou, Jagers and Vatutin \cite{Hac05}.

One of the main topics of research on MGWBPs, and on branching processes in general, is the characterisation of the probability that the population eventually becomes empty. We denote by $q_i$ the conditional extinction probability of the branching process, given that it starts with a single individual of type $i$, and we let $\vc q:=(q_i)$. When the number of types is finite, it is well known that the vector $\vc q$ is the minimal non-negative solution of the fixed-point extinction equation, 
\begin{equation}\label{Fixed Point}
\vc s=\vc G(\vc s),
\end{equation}
where $\vc G(\vc s):=(G_i(\vc s))$ records the progeny generating function associated with each type. Most of the time this {finite} system of equations cannot be solved analytically, but the linear functional iteration algorithm or the quadratic Newton algorithm can be applied to compute $\vc q$ numerically. {In addition, there is a well-established extinction criterion, namely $\bm{q}=\bm{1}$ if and only if the Perron-Frobenius eigenvalue of the mean progeny matrix is less than or equal to one.}

{In contrast to the finite-type case, when there are infinitely many types} 
%Moreover,  in that case 
we need to distinguish between {different} extinction events. We say that there is \emph{global extinction} when the whole population becomes extinct, and $\vc q$ correspond to the probability vector for this event, and we refer to the event that every type becomes extinct as \emph{partial extinction}, and denote its probability vector by $\tilde{\vc q}$. 
{In the context of branching random walks (BRWs), in which individuals are assigned locations instead of types, partial extinction is analogous to \emph{local extinction} at every location.}
It is clear that global extinction implies partial extinction but the converse does not necessarily hold; indeed, it is possible for every type to eventually disappear while the total number of individuals approaches infinity {(see for instance \cite[Section 5]{haut12})}. 
{The vectors $\bm{q}$ and $\tilde{\vc q}$ both satisfy \eqref{Fixed Point}, where $\bm{q}$ is the minimal non-negative solution and $\bm{\tilde{q}}$ may or may not be equivalent to $\bm{{q}}$.}
%\soph{Note that branching random walks (BRWs) form a class of MGWBPs with countably many types since the location of an individual in the BRW can be interpreted as its type in the corresponding MGWBP; in the context of BRWs, partial extinction is analogous to \emph{local extinction} at every location.}

{Due to the challenges that arise when transitioning from a finite to a countably infinite type set, many of the questions that have been thoroughly explored in the finite setting remain open. Indeed,}
% When the number of types is countably infinite, 
%the extinction probability vector $\vc q$ is still the minimal non-negative solution of the fixed-point equation \eqref{Fixed Point}, 
%it is generally unclear how to compute $\vc q$ and $\bm{\tilde{q}}$ as they now has infinitely many entries.  
{apart from the recent works of \cite{Lin15} and \cite{Sag13} on a restricted class of MGWBPs with linear fractional progeny generating functions, and of \cite{haut12} on algorithmic techniques, scant attention has been paid to computational aspects of the infinite extinction probability vectors $\vc q$ and $\tilde{\vc q}$.} 
%{\red Moreover despite contributions from a number of authors (listed below) the literature sill lacks an easily applicable global extinction criterion that holds under mild conditions}. 
{In addition,} several authors have investigated conditions for $\vc q=\vc 1$ or $\tilde{\vc q}=\vc 1$, 
%in a MGWBP with countably many types or in a BRW
see for instance \cite{Zuc14,Har02,haut12,moy66,moy67,moyal62,Spa89,Tet04,Zuc11}; however, {while it has been well established that the convergence norm of the (infinite) mean progeny matrix provides a partial extinction criterion (see \cite{Mul08}), the literature still lacks an \soph{easily applicable global extinction criterion that holds under mild conditions}. 
%necessary and sufficient condition for partial extinction is that the convergence norm of the (infinite) mean progeny matrix is less than or equal to one (see \cite{Mul08}), the literature still lacks an \soph{easily applicable global extinction criterion that holds under mild conditions}. 
Here we address the two problems {in parallel by defining two new probabilistic tools: to each MGWBP with countably many types $\{\vc  Z_n\}$, we associate \textit{(i)} a sequence of truncated and augmented finite-type branching processes $\{\bar{\vc Z}^{(k)}_n\}_{k\geq 1, n\geq 0}$, which themselves naturally define \textit{(ii)} an embedded branching process $\{\vc S_k\}$ referred to as the \emph{seed process}. The next two paragraphs provide an intuitive description of these two tools and their benefits.

For each $k\geq 1$, the $k$th finite-type branching process $\{\bar{\vc Z}^{(k)}_n\}$ is constructed pathwise on the same probability space as the original process $\{\vc  Z_n\}$ by replacing all types larger than $k$ with a type randomly selected from the set $\{1,\ldots,k\}$ according to some distribution $\vc\alpha^{(k)}$. The corresponding (finite) extinction probability vector is denoted by $\bar{\vc q}^{(k)}$.
%, can be computed using established methods for MGWBPs with finitely many types. 
In our main theorem (Theorem \ref{general converge}), we prove that, under some sufficient conditions on $\{\vc  Z_n\}$ (closely related to the dichotomy property) and on the sequence of replacement distributions $\{\vc\alpha^{(k)}\}_{k\geq 1}$ (similar to a tightness condition), the sequence $\{\bar{\vc q}^{(k)}\}_{k\geq 1}$ converges to the global extinction probability $\vc q$. 
{This result establishes a link between the extinction of non-singular irreducible finite-type branching processes and global extinction in the infinite-type setting.}
{It} has several implications. First, 
Theorem \ref{general converge} extends the work in \cite{haut12}, in which two monotone sequences of extinction probability vectors, $\{\vc q^{(k)}\}_{k\geq 1}$ and $\{\tilde{\vc q}^{(k)}\}_{k\geq 1}$, are shown to converge respectively to $\vc q$ and $\tilde{\vc q}$. These sequences were obtained by replacing all types larger than $k$ either by an immortal type (yielding $\vc q^{(k)}$), or by a sterile type (yielding $\tilde{\vc q}^{(k)}$), and the monotone convergence theorem was the main argument in the proof. In contrast, the new sequence $\{\bar{\vc q}^{(k)}\}$ is not necessarily monotone, and a completely different approach is required. From a computational point of view, the flexibility in the choice of the replacement distributions $\{\vc\alpha^{(k)}\}$ motivates the search for an optimal choice maximising the convergence rate, but this is out of the scope of this paper.
%the new family of finite-type processes $\{\bar{\vc Z}^{(k)}_n\}$ 
%it builds a connection between non-singular finite- and infinite-type branching processes which can be exploited to derive new global extinction criteria; this was not the case for the processes defined in \cite{haut12}.
{Second, as a direct consequence of Theorem \ref{general converge}, we derive new sufficient conditions for $\vc q=\vc 1$ and $\vc q<\vc 1$. Such results could not be obtained using the sequences in \cite{haut12}. }
%Extending these conditions is a topic of future research.} 

 The seed process $\{\vc S_k\}$ is an MGWBP evolving in a varying environment, 
 {that arises naturally when exploring the asymptotic behaviour of $\{ \bm{\bar{q}}^{(k)} \}$.}
% embedded in the original process $\{\vc  Z_n\}$; 
 It is constructed pathwise from the family of finite-type processes $\{\bar{\vc Z}^{(k)}_n\}$ as follows: the individuals (or \emph{seeds}) in the $k$th generation of $\{\vc S_k\}$ correspond to the individuals in $\{\vc  Z_n\}$ which are replaced by a random type according to $\vc\alpha^{(k)}$ to form $\{\bar{\vc Z}^{(k)}_n\}$. 
 The seed process 
% plays a key role in the asymptotic behaviour of $\{\bar{\vc q}^{(k)}\}$ and 
 is {the} fundamental ingredient in the proof of {Theorem \ref{general converge}} but in addition, 
% we show that the seed process 
 enjoys {several} interesting properties on its own. For example, $\{\vc S_k\}$ almost surely becomes extinct if and only if global and partial extinction of the original process coincide. 
{While in the present paper our interest in the seed process remains its application to the sequence $\{ \bm{\bar{q}}^{(k)}\}$, we lay the foundations for a subsequent paper \cite{BrauHau2017}, in which properties of the seed process are exploited further, to yield, among other results, a global extinction criterion that applies to a class of branching processes referred to as \emph{lower Hessenberg}.} 

%Properties of the seed process can be further exploited to analyse MGWBP with countably many types: for example, in a subsequent paper \cite{BrauHau2017} dealing with a class of branching processes called \emph{lower Hessenberg}, the corresponding single-type seed process is used to fully characterise the set of fixed points and to derive a neat global extinction criterion.
}

Finally, we investigate the convergence properties of the sequence $\{\bar{\vc q}^{(k)}\}$ {when the conditions on $\{\bm{\alpha}^{(k)} \}$ in Theorem \ref{general converge} are not met. We consider
% for two specific replacement distributions which do not satisfy the sufficient condition on $\vc\alpha^{(k)}$
\emph{(a)} replacement by the last type, that is, $\vc\alpha^{(k)}=\vc e_k$, and \textit{(b)} replacement by a uniformly distributed type, that is, $\vc\alpha^{(k)}=\vc 1/k$ and, one particular example that focuses on each case, Examples 2 and 3, respectively. 
% and through these examples we demonstrate the contrasting asymptotic behaviour of $\{\bm{\bar{q}}^{(k)}\}$: 
 In Example 2, we prove that the limit of the sequence $\{\bm{\bar{q}}^{(k)}\}$ does not always exist, 
% however we show that in Example 2 the limit inferior does converge to the global extinction probability $\vc q$, 
and in Example 3, the limit does exist but may correspond to the partial extinction probability $\tilde{\vc q}$. Example 2 highlights the sensitivity of the limit of $\{\bm{\bar{q}}^{(k)}\}$ under \emph{(a)}, whereas Example 3 demonstrates how alternative choices of $\{\bm{\alpha}^{(k)} \}$ may lead to contrasting asymptotic behaviour in $\{\bm{\bar{q}}^{(k)}\}$. 
%Indeed, Example 3 illustrates a case where $\lim_{k} \bm{\bar{q}}^{(k)}=\bm{q}$ if $\bm{\alpha}^{(k)}=\bm{e}_1$, $\liminf_k \bm{\bar{q}}^{(k)}=\bm{q} \neq \limsup_k \bm{\bar{q}}^{(k)}=\bm{\tilde{q}}$ if $\bm{\alpha}^{(k)}=\bm{e}_k$, and $\lim_{k} \bm{\bar{q}}^{(k)}=\bm{\tilde{q}}$ if $\bm{\alpha}^{(k)}=\bm{1}/k$.}

The paper is organised as follows. The next section provides the background on MGWBP with countably many types. Section 3 focuses on the pathwise construction of the branching processes with corresponding extinction probabilities $\vc q^{(k)}$, $\tilde{\vc q}^{(k)}$, and $\bm{\bar{q}}^{(k)}$. In Section 4  we establish sufficient conditions for the convergence of $\{\bar{\vc q}^{(k)}\}$ to $\vc q$, we study properties of the related seed process, and we prove the main theorem on the convergence of $\{\bm{\bar{q}}^{(k)}\}$. {In Section 5} \soph{we derive sufficient conditions for $\vc q=\vc 1$ and $\vc q<\vc 1$}.
%\footnote{\blue Should we re-think about the organisation and title of the sections?}. 
Finally, in Section 6 we study the asymptotic behaviour of $\{\bm{\bar{q}}^{(k)}\}$ {for replacement distributions that do not satisfy the conditions of our main theorem} and provide some numerical {illustrations}. \soph{The pseudo-code for the computation of the global and partial extinction probabilities, and the proofs related to Examples 2 and 3 are provided in some appendices.
%using the corresponding converging sequences is provided in an Appendix, and the proofs related to Examples 2 and 3 are provided in some appendices.
}

\section{Preliminaries}

Consider a multitype Galton-Watson process with countably infinite type set $\mathcal{S}=\{1,2, 3,\dots \}$. Throughout the paper we assume that the process initially contains a single individual, whose type will be denoted as $\varphi_0$. The process then evolves according to the following rules:
\begin{itemize}
\item[{\it{(i)}}] each individual lives for a single generation, and
\item[ {\it{(ii)}}] at death it gives birth to $\bm{r}=(r_1, r_2,...)$ offspring, that is, $r_1$ individuals of type 1, $r_2$ individuals of type 2, etc., where the  vector $\bm{r}$ is chosen independently of all other individuals according to a probability distribution, $p_i( \cdot)$, specific to the parental type $i \in \mathcal{S}$.
\end{itemize}

{
{Following Mode \cite{mode} we now give an equivalent but more formal construction of this process.} 
This formulation differs slightly from the standard construction in Harris \cite{Har02} and Jagers \cite{Jag89} but is useful, in particular, in defining the sequence of truncated and augmented processes in Section \ref{ftbp}.
%This construction is used throughout the paper, in particular, in the next section where the related finite-type branching processes are formally defined.

Consider the set of all possible individuals of the form $\langle i_1 j_1\rangle$, $\langle i_1 j_1 i_2 j_2\rangle$, $\dots$, $\langle i_1 j_1 i_2 j_2 \dots i_n j_n \rangle$, $\dots$ where $\langle i_1 j_1 i_2 j_2 \dots i_n j_n \rangle$ is a member of the $n$th generation and is the $i_n$th child of type $j_n$ born to $\langle i_1 j_1 i_2 j_2 \dots i_{n-1} j_{n-1} \rangle$. In other words, each individual in the $n$th generation belongs to the set $\mathcal{J}_n=\langle \mbN \times \mathcal{S} \rangle^n$, and $\mathcal{J}=\langle 0 \rangle \cup \bigcup^{\infty}_{n=1} \mathcal{J}_n$ contains all individuals. To each individual $I \in \mathcal{J}$ we associate a sample space $\Omega_I$ made up of all infinite, non-negative, integer-valued vectors $\bm{r}_I$, with at most finitely many strictly positive entries, where $\bm{r}_I$ represents the number of offspring of the various types produced by the individual. Let $\mathcal{B}_I$ be the corresponding discrete $\sigma$-algebra, and let $\mbP_I$ be the probability measure such that $\mbP_I(B)=\sum_{\vc r_I\in B} p_i(\vc r_I)$ if $I$ is of type $i$, for all $B\in \mathcal{B}_I$. The sequence of 
probability spaces $\{(\Omega_I,\mathcal{B}_I,\mbP_I)\}_{I\in\mathcal{J}}$ induces the product probability space $(\Omega,\mathcal{F},\mbP)$ on which the Galton-Watson branching process with countably many types is defined.  The elements $\omega \in \Omega$ are of the form $\omega=( \bm{r}_I ; I \in \mathcal{J})$. Let $\bm{N}( \omega, I)=(N_1(\omega, I), N_2( \omega, I), \dots )$ contain the number of offspring of each type generated by individual $I$. The individual $I=\langle i_1 j_1 \dots i_n j_n \rangle$ then appears in the population if and only if,
\begin{equation}\label{appear}
i_1 \leq N_{j_1}(\omega, \langle 0 \rangle), \: \: i_2 \leq N_{j_2} ( \omega, \langle i_1 j_1 \rangle ), \: \: \dots , \: \: i_n \leq N_{j_n}( \omega, \langle i_1 j_1 \dots i_{n-1} j_{n-1} \rangle ).
\end{equation} 
For every individual $I=\langle i_1 j_1 \dots i_n j_n \rangle\in \mathcal{J}_n$, let $Z_{j}(\omega,I)=1$ if both $j_n=j$ and condition \eqref{appear} are satisfied, and equal $0$ otherwise. The population at generation $n$ is then given by the vector $\bm{Z}_{n}(\omega)$ which has entries
\begin{equation}\label{Zn}
{Z}_{n,j}(\omega) = \sum_{I \in \mathcal{J}_n} Z_j(\omega,I),\quad j\in\mathcal{S}.
\end{equation}
In the sequel, we will often drop the dependence in $\omega$ when it is not contextually important, and refer to the branching process as $\{ \bm{Z}_n \}_{n \geq 0}$. We let $| \bm{Z}_n | := \sum_{j \in \mathcal{S}} {Z}_{n,j}$ be the total population size at generation $n$. 
}

From the set of probability distributions $\{ p_i ( \cdot ) \}_{i \in \mathcal{S}}$ we define the \emph{progeny generating function} $\bm{G} : [0,1]^{\mathcal{S}} \to [0,1]^{\mathcal{S}}$, which has entries,
\begin{equation}
G_i ( \bm{s} ) = \sum_{ \bm{r} } p_{i}(\bm{r}) \bm{s}^{\bm{r}} = \sum_{ \bm{r} }  p_{i}(\bm{r}) \prod^\infty_{k=1} s_k^{r_k},\quad i\in\mathcal{S}.
%\footnote{\red We need to avoid summing over an uncountable set so I changed the domain to the sum. This isn't the best notation but considering we already defined the required set $\Omega_I$ it seems a pity not to use it.}
\end{equation}
%To the countable sample space $\Omega$ we associate the natural $\sigma$-algebra $\mathcal{F}$ and define a probability measure $\mbP$ as follows: for each $\omega=( \bm{r}_I ; I \in \mathcal{J})\in \Omega$, 
%$$p(\omega)=\prod_{i\in\mathcal{S}}\prod_{I\in\mathcal{J}:Z_i(\omega,I)=1} p_i(\vc r_I),$$ 
%and $\mbP(A)=\sum_{\omega\in A} p(\omega)$ for all $A\in\mathcal{F}$. The branching process $\{ \bm{Z}_n \}$ is therefore defined on the probability space $(\Omega,\mathcal{F},\mbP)$.
The mean progeny matrix $M$ is an infinite matrix whose entries are given by
\[
M_{ij} =  \left. \frac{\partial G_i ( \bm{s} )}{\partial s_j }  \right|_{\bm{s}= \bm{1}}, \quad \text{ for } i,j \in \mathcal{S}, 
\]
where $M_{ij}$ can be interpreted as the expected number of type $j$ children born to a parent of type $i$. We assume that the row sums of $M$ are finite, that is, the expected total number of direct offspring of an individual of any type is finite. It is sometimes convenient to associate a graph to the mean progeny matrix whose set of vertices corresponds to the set of types $\mathcal{S}$, and in which there is an oriented edge between nodes $i$ and $j$ with weight $M_{ij}$ if and only if $M_{ij}>0$. We shall refer later to this graph as the \textit{mean progeny representation graph}. We say that there is a path from type $i$ to type $j$ if such a (directed) path exists in the mean progeny representation graph.
The process $\{ \bm{Z}_n \}$ is \emph{irreducible} if there is a path between every pair of nodes.
%$i$ to$M$ { is} irreducible, that is, if for every pair of types $(i,j)$ there exists $n\in\mathbb{N}$ such that $(M^n)_{ij}>0$; it is said to be reducible otherwise. 

We distinguish between two types of extinction events: the \textit{global }extinction event, $\{\lim_{n \to \infty} | \bm{Z}_n | = 0\}$, corresponding to the event that the whole population eventually becomes extinct;  and the \textit{partial} extinction event, $\{ \forall l {\geq 1}: \lim_{n \to \infty}Z_{n,l}=0\}$, corresponding to the event that all types eventually become extinct. Note that in the finite-type case, both events are equivalent.
%\footnote{{\color{red} I would add here your classical example from 'Extinction probabilities of branching processes with countably infinitely many types' that also appears as Figure 1 in your conference paper. Unless you wanted to include this in the Introduction?}}. 
The conditional {global} extinction probability vector, given the initial type, is $\bm{q}=(q_1,q_2,\ldots)$, where
\[
q_i = \mbP \left( \lim_{n \to \infty} | \bm{Z}_n | = 0 \mid \varphi_0=i \right),
\] and
the conditional {partial} extinction probability vector, given the initial type, is $\bm{\tilde{q}}=(\tilde{q}_1,\tilde{q}_2,\ldots)$, where
\[
\tilde{q}_i= \mbP\left( \forall l \geq 1: \lim_{n \to \infty} Z_{n,l}=0 \mid \varphi_0=i\right).
\]
It is clear that global extinction implies partial extinction, that is,  $\bm{q} \leq \bm{\tilde{q}}$. The vectors $\bm{q}$ and $\bm{\tilde{q}}$ are both solutions to the fixed point equation \eqref{Fixed Point}.
%\begin{equation}\label{fixed point}
%\bm{s} = \bm{G} ( \bm{s} ). \footnote{Keep or remove as Peter suggests?}
%\end{equation}
Additionally, $\bm{q}$ is the minimal non-negative solution of \eqref{Fixed Point}, whereas $\bm{\tilde{q}}$ is not necessarily the minimal non-negative solution. 
%Moyal \cite[Lemma 3.3]{moyal62} proved that, when $\{ \bm{Z}_n \}$ is irreducible and $\inf_i q_i>0$, Equation \eqref{Fixed Point} has at most one solution satisfying $\sup_i s_i <1$, which corresponds to $\bm{q}$. Under more general conditions, Spataru \cite[Theorem 3]{Spa89} stated that Equation \eqref{Fixed Point} has at most two solutions, $\bm{q}$ and $\bm{1}$. However, the authors of \cite{Zuc14} showed the inaccuracy of the latter by providing an irreducible example such that $\bm{q} < \bm{\tilde{q}} < \bm{1}$; in addition, more recently in \cite{Zuc15}, the same authors constructed an example which satisfies Moyal's conditions and has infinitely many solutions $\bm{s}< \bm{1}$ such that $\sup_i s_i=1$.

{
%\footnote{\blue I think I prefer the introduction of this paragraph that we had before, i.e. saying that in our illustrative examples we shall make us of a process... Now it cuts a bit the flow and we don't see directly why we talk about this here}
In our illustrative examples we shall make use of a process defined for any MGWBP $\vc V:=\{ \bm{V}_n \}_{n \geq 0}$ with {type set} $\mathcal{S}_V \subseteq \mathcal{S}$, which was previously considered by
\cite{Zuc09,Com07,Gan05}, {among others}. We refer to this process as the \emph{type-$i$ branching process embedded with respect to {$\{\bm{V}_n \}_{n\geq 0}$}} and denote it by $\{E_n^{(i)}(\bm{V})\}_{n\geq 0}$. The sample paths of {$\{E_n^{(i)}(\bm{V})\}$} are constructed from those of $\{ \bm{V}_n : \varphi_0=i \}$ by taking all type-$i$ individuals that appear in $\{ \bm{V}_n  \}$ and defining the direct descendants of these individuals as their closest (in generation) type-$i$ descendants in $\{ \bm{V}_n  \}$. The process $\{E_n^{(i)}(\bm{V})\}$ evolves as a (single-type) Galton-Watson process whose extinction probability is equivalent to the probability that type $i$ becomes extinct in $\{ \bm{V}_n \}$. While we will use this fact directly in Example 1,
%\footnote{\blue do we need to mention that?}
 it also implies that type $i$ survives with positive probability in $\{\bm{V}_n\}$ if and only if the mean number of offspring in $\{E_n^{(i)}(\bm{V})\}$ is strictly greater than 1, that is,
\[
m_{E_n^{(i)}(\bm{V})} := M_{ii}+ \sum_{n=2}^\infty \left( \sum_{i_1, \dots , i_{n-1} \in \mathcal{S} \backslash \{i\}} M_{ii_1} M_{i_1 i_2} \dots M_{i_{n-1}i} \right)>1.
\]
Observe that $m_{E_n^{(i)}(\bm{V})}$ can be identified as the weighted sum of all first return paths to $i$ in the mean progeny representation graph associated with the branching process $\{ \bm{V}_n \}$. In the irreducible case, if $|\mathcal{S}_V|<\infty$, then 
%$m_{E_n^{(i)}(\bm{V})} \leq 1$ if and only if $\rho(M)\leq 1$ if and only if $\vc q=\tilde{\vc q}=\vc 1$ 
$$m_{E_n^{(i)}(\bm{V})} \leq 1\quad \Leftrightarrow \quad\rho(M)\leq 1\quad\Leftrightarrow\quad\vc q=\tilde{\vc q}=\vc 1,$$
where $\rho(M)$ denotes the Perron-Frobenius eigenvalue of the mean progeny matrix $M$; if $|\mathcal{S}_V|=\infty$, then 
%$m_{E_n^{(i)}(\bm{V})} \leq 1$ if and only if $\rho(M)\leq 1$ if and only if $\vc q=\tilde{\vc q}=\vc 1$ 
$$m_{E_n^{(i)}(\bm{V})} \leq 1\quad \Leftrightarrow \quad\nu(M)\leq 1\quad\Leftrightarrow\quad\tilde{\vc q}=\vc 1,$$
where $\nu(M)$ denotes the convergence norm of $M$; see for instance \cite{Zuc11}. }
}

In the sequel we {adopt} the shorthand notation $\mbP_i(\cdot):=\mbP(\cdot | \varphi_0=i)$ and $\mbE_i(\cdot):=\mbE(\cdot | \varphi_0=i)$. For any $k \geq 1$, we define the partition $T_k:= \{ 1,2,...,k\}$ and $T^c_k := \{ k +1, k+2,...\}$ of the set of types $\mathcal{S}=\{1,2,3,... \}$, and we let
\[
\tau_k(\omega) = \inf\left\{n \geq 0 : \sum_{i=k+1}^{\infty} {Z}_{n,i}(\omega)>0 \right\}
\]
be the first passage time to $T^c_k$. Note that for each $\omega\in\Omega$,  $\{\tau_k(\omega)\}_{{ k\geq 1}}$ forms a increasing sequence in $k$.

\section{Pathwise construction of finite-type branching processes on $(\Omega,\mathcal{F},\mbP)$}\label{ftbp}
%\footnote{\magenta I am not comfortable removing the ``finite-type'' in the title... I guess it wouldn't hurt too much if we leave it}

%\subsection{Truncations and augmentations of $\{{\textbf{Z}}_n\}$}
%\footnote{\blue Since there is no Section 3.2, I have removed the title of Subsection 3.1}

{In this section we construct the sequences of finite-type 
%\footnote{\blue maybe `finite-type' instead of `modified'?}
  processes $\{\bm{\tilde{Z}}_n^{(k)} \}_{n \geq 0}$, $\{ \bm{Z}_n^{(k)} \}_{n \geq 0}$ and $\{\bm{\bar{Z}}_n^{(k)} \}_{n \geq 0}$ on $(\Omega,\mathcal{F},\mbP)$. The first two sequences were studied in \cite{haut12}, while the latter has not been previously considered. This construction, {which was not detailed in \cite{haut12}}, plays a key role in the remainder of the paper.}

For each $k \geq 1$, the realisations of $\{\bm{\tilde{Z}}_n^{(k)} \}$ are constructed pathwise from those of $\{ \bm{Z}_n \}$ by removing the descendants of all individuals of type $i \in T^c_k$. More specifically, for each $\omega \in \Omega$ and individual $I= \langle i_1 j_1 \dots i_{n} j_{n} \rangle \in \mathcal{J}$ of type $j=j_n$, we let \soph{$\tilde{\bm{N}}^{(k)}(\omega,I)=\bm{N}(\omega, I)\,\mathds{1}\{j\leq k\}$. The condition of appearance of an individual in the truncated branching process is then the same as \eqref{appear}, replacing $N_{j\cdot}(\omega,I)$ by $\tilde{{N}}_{j\cdot}^{(k)}(\omega,I)$, and the definition of the population size vector $\bm{\tilde{Z}}^{(k)}_n(\omega)$ is analogous to \eqref{Zn}.} Consequently, in $\{\bm{\tilde{Z}}_n^{(k)} \}$
\begin{itemize}
\item[{{\it (i)}}]  all types  in $T_k$ have the same progeny as the corresponding types in $\{ \bm{ Z}_n \}$, and
\item[{ {\it (ii)}}]  all types in $T^c_k$ die with no offspring; these types are said to be \emph{sterile}.
\end{itemize}
We denote by $\bm{\tilde{q}}^{(k)}$ the global extinction probability vector of $\{\bm{\tilde{Z}}_n^{(k)} \}$. Since all types larger than $k$ are sterile in $\{\bm{\tilde{Z}}_n^{(k)} \}$, the truncated process behaves effectively like a finite-type branching process on the set of types $\{1,2,\ldots,k\}$.  It is clear that $\tilde{q}_i^{(k)}=1$ for all $i>k$, so the computation of $\bm{\tilde{q}}^{(k)}$ reduces to solving a finite system of $k$ equations.  {It was shown in \cite{haut12} that} the sequence $\{\bm{\tilde{q}}^{(k)}\}$ is monotone decreasing and converges pointwise to $\bm{\tilde{q}}$.

Similarly, for each $k \geq 1$ the process $\{ \bm{Z}^{(k)}_n \}$ is constructed pathwise from realisations of $\{ \bm{Z}_n \}$ by removing all individuals of type $i \in T^c_k$ and their descendants, and replacing each pruned branch with an infinite line of descent made up of type $\Delta$ individuals. More formally, for each $\omega \in \Omega$, the population size vector $\bm{Z}_n^{(k)}(\omega)$ has entries
\[
{Z}_{n,i}^{(k)} (\omega)= 
\begin{cases}
\sum^n_{j=0} \sum_{l=k+1}^\infty \tilde{Z}^{(k)}_{j,l}(\omega), & i=\Delta \\
\tilde{Z}_{n,i}^{(k)}( \omega), & 1 \leq i \leq k \\
0, & i >k.
\end{cases}
\]
As a consequence, in $\{ \bm{Z}^{(k)}_n \}$
\begin{itemize}
\item[{ {\it (i)}}]  all types in $T_k$ have the same progeny distribution as the corresponding types in $\{ \bm{Z}_n \}$, and
\item[{ {\it (ii)}}] all types in $T_k^c$ are instantaneously replaced by the absorbing type $\Delta$, which at each generation produces a single type $\Delta$ progeny {with probability one}.
\end{itemize} 
Once a type $\Delta$ individual is born, $\{ \bm{Z}^{(k)}_n \}$ does not become extinct. In this sense, individuals of type $\Delta$ can be thought of as \textit{immortal}. We denote by $\bm{q}^{(k)}$ the global extinction probability vector of $\{ \bm{Z}_n^{(k)} \}$; it contains only finitely many non-zero entries since $q_i^{(k)}=0$ for all $i>k$. It is clear that $\{\bm{Z}_n^{(k)}\}$ becomes extinct if and only if $\{\vc Z_n\}$ becomes extinct before the birth of the first individual with a type in $T_k^c$. It {was} proved in \cite{haut12} that the sequence $\{\bm{q}^{(k)}\}$ is monotone increasing and converges pointwise to $\bm{q}$.

\medskip

%We now use a similar pathwise approach to define a third sequence of branching processes from the original $\{\vc Z_n\}$, and in the next sections we study the convergence properties of the associated sequence of extinction probabilities.
For {each} $k\geq 1$, we construct recursively the truncated and augmented branching process {$\{ \bm{\bar{Z}}^{(k)}_n\}_{ n \geq 0}$ for} which
\begin{itemize}
\item[{ {\it (i)}}] all types in $T_k$ have the same progeny distribution as the corresponding types in $\{ \bm{Z}_n \}$, and
\item[{ {\it (ii)}}] all types in $T_k^c$ are instantaneously, and independently of each other, replaced by type $X\in\{1,...,k\}$ which is selected using the probability distribution $\bm{\alpha}^{(k)}$. The replaced individual then generates new individuals according to the progeny distribution of its type.
\end{itemize} 
To construct the sample paths of $\{ \bm{\bar{Z}}_n^{(k)}\}_{n \geq 0}$ we first augment the probability space $(\Omega,\mathcal{F},\mbP)$ to carry the sequence of independent random variables $\{X^{(k)}_{l}(I)\}_{l,k \in \mbN, I \in \mathcal{J}}$. For each $k\geq 1$, these random variables take values in $\{1, \dots, k\}$ and have probability distribution $\bm{\alpha}^{(k)}=(\alpha^{(k)}_1, \alpha^{(k)}_2, \dots, \alpha^{(k)}_k)$, where $\alpha^{(k)}_i=\mbP(X^{(k)}_{l}(I)=i)$. We interpret $X^{(k)}_{l}(\omega,I)$ as the replacement type of the $l$th offspring of type strictly larger than $k$ born to $I$ for the realisation $\omega$. Let $N_{(k,\infty)} (\omega,I)=\sum_{j=k+1}^\infty N_j(\omega,I)$ and define $\bm{\bar{N}}^{(k)}(\omega,I)$ with entries,
\[
\bar{N}_i^{(k)}(\omega,I) = 
\begin{cases}
N_i(\omega,I) + \sum^{N_{(k,\infty)} (\omega,I)}_{l=1} \bm{1} \{ X^{(k)}_{l}(\omega,I) = i \}, & 1 \leq i \leq k \\
0, & i>k. 
\end{cases}
\]
\soph{The condition of appearance of an individual in the truncated and augmented branching process is then the same as \eqref{appear}, replacing $N_{j\cdot}(\omega,I)$ by $\bar{{N}}_{j\cdot}^{(k)}(\omega,I)$, and the definition of the population size vector $\bm{\bar{Z}}^{(k)}_n(\omega)$ is analogous to \eqref{Zn}.} 
%Individual $I$ then appears in the truncated and augmented branching process if and only if 
%\begin{equation}\label{TA appear}
%i_1 \leq \bar{N}^{(k)}_{j_1}(\omega, \langle 0 \rangle), \: \: i_2 \leq  \bar{N}^{(k)}_{j_2} ( \omega, \langle i_1 j_1 \rangle ), \: \: \dots , \: \: i_n \leq  \bar{N}^{(k)}_{j_n}( \omega, \langle i_1 j_1 \dots i_{n-1} j_{n-1} \rangle ).
%\end{equation} 
%We define $\bar{Z}_j^{(k)}(\omega,I)$ which takes value 1 if $I$ is of type $j$ and condition \eqref{TA appear} is satisfied, and value 0 otherwise. Then $\bm{\bar{Z}}^{(k)}_n(\omega)$ has entries,
%\[
%{\bar{Z}}^{(k)}_{n,j}(\omega)= \sum_{i \in \mathcal{J}_n} \bar{Z}_j^{(k)}(\omega,I).
%\]

An illustration of $\{ \bm{Z}_n \}$, $\{ \bm{\tilde{Z}}^{(k)}_n\}$, $\{ \bm{Z}^{(k)}_n\}$ and $\{ \bm{\bar{Z}}^{(k)}_n\}$ is given in Figure \ref{allproc} for $k=2$ and a specific $\omega\in\Omega$. Note that $\{ \bm{\tilde{Z}}^{(2)}_n\}$ and $\{ \bm{Z}^{(2)}_n\}$ are both functions of $\{ \bm{Z}_n \}$ whereas in $\{ \bm{\bar{Z}}^{(2)}_n \}$ there are two subtrees with root 1, constructed using information redundant in $\{ \bm{Z}_n \}$, in place of the two sterile individuals in $\{ \bm{\tilde{Z}}^{(2)}_n \}$.

\begin{figure}\label{allproc}
\hspace{-1.6cm}\begin{tikzpicture}[->,>=stealth',level/.style={sibling distance = 2.5cm/#1,
  level distance = 0.8cm}] 
 $\{\bm{Z}_n \}\phantom{\sum_{j}}$
\node [arn_n] {1}
    child{ node [arn_n] {2} 
            child{ node [arn_n] {4}            								
            }                          
    }
    child{ node [arn_n] {3} 
            child{ node [arn_n] {1} 
							child{ node [arn_n] {4}}							
            }
            child{ node [arn_n] {4}
							child{ node [arn_n] {2}}
							child{ node [arn_n] {3}}
            } 
		}
;  
\end{tikzpicture} 
%\hspace{4mm}
\begin{tikzpicture}[->,>=stealth',level/.style={sibling distance = 1.7cm/#1,
  level distance = 0.8cm}] 
  $\{\bm{\tilde{Z}}_n^{(2)} \}\phantom{\sum_{j}}$
\node [arn_n] {1}
    child{ node [arn_n] {2} 
            child{ node [arn_n] {4}            								
            }                          
    }
    child{ node [arn_n] {3} 
            child{ node [white] {1} 
						[white]  	child{ node [white] {4}}							
            }
          child{ node [white] {4} [white]  
							child{ node [white] {2}}
							child{ node [white] {3}}
            } 
		}
		;
\end{tikzpicture} 
%\bigskip
%\hspace{-13mm}
\begin{tikzpicture}[->,>=stealth',level/.style={sibling distance = 1.7cm/#1,
  level distance = 0.8cm}] 
  $\{\bm{Z}_n^{(2)} \}\phantom{\sum_{j}}$
  \node [arn_n] {1}
    child{ node [arn_n] {2} 
            child{ node [arn_x] {$\Delta$} 
            	child{ node [arn_x] {\scalebox{.55}{$\vdots$}}  }           								
            }                          
    }
    child{ node [arn_x] {$\Delta$}
    	child{ node [arn_x] {$\Delta$} 
		child{ node [arn_x] {\scalebox{.55}{$\vdots$}}  }
	}            
		}
;  
\end{tikzpicture} 
\hspace{2mm}
\begin{tikzpicture}[->,>=stealth',level/.style={sibling distance = 2.5cm/#1,
  level distance = 0.8cm}] 
  $\{\bm{\bar{Z}}_n^{(2)} \}\phantom{\sum_{j}}$
\node [arn_n] {1}
    child{ node [arn_n] {2} 
            child{ node [arn_n] {1}   
		child{ node [arn_n] {1}}
		child{ node [arn_n] {2}}      								
            }                          
    }
    child{ node [arn_n] {1}          
		}
;  
\end{tikzpicture}
\caption{A visualisation of $\{ \bm{Z}_n \}$, $\{ \bm{\tilde{Z}}^{(2)}_n\}$, $\{ \bm{Z}^{(2)}_n\}$ and $\{ \bm{\bar{Z}}^{(2)}_n\}$ corresponding to a specific realisation $\omega \in \Omega$.  It is such that $\bm{r}_{\langle 0 \rangle}=\bm{e}_2+ \bm{e}_3$, $\bm{r}_{\langle 12 \rangle}=\bm{e}_4$, $\bm{r}_{\langle 13 \rangle}=\bm{e}_1+\bm{e}_4$, $\bm{r}_{\langle 1311 \rangle}=\bm{e}_4$, $\bm{r}_{\langle 1314 \rangle}=\bm{e}_2+\bm{e}_3$ and $\bm{r}_{\langle 1214 \rangle}=\bm{r}_{\langle 131114 \rangle}=\bm{r}_{\langle 131412 \rangle}=\bm{r}_{\langle 131413 \rangle}=\bm{0}$ are used to construct $\{ \bm{Z}_n \}$, $\{ \bm{\tilde{Z}}^{(2)}_n\}$ and $\{ \bm{Z}^{(2)}_n\}$, and additionally $X^{(2)}_1(\omega, \langle 0 \rangle)=1$, $\bm{r}_{\langle 11 \rangle}=\bm{0}$, $X^{(2)}_1(\omega, \langle 12 \rangle)=1$, $\bm{r}_{\langle 1211 \rangle}=\bm{e}_1+\bm{e}_2$ and $\bm{r}_{\langle 121111 \rangle}=\bm{r}_{\langle 121112 \rangle}=\bm{0}$ are used to construct $\{ \bm{\bar{Z}}^{(k)}_n \}$. All other information contained in $\omega$ is not required. }
\end{figure}

We denote by $\bm{\bar{q}}^{(k)}$ the global extinction probability vector of $\{ \bm{\bar{Z}}^{(k)}_n \}$. The vector $\bm{\bar{q}}^{(k)}$ contains infinitely many entries, and is such that for all $i>k$, $\bar{q}_i^{(k)}=\sum^k_{j=1} \alpha^{(k)}_j \bar{q}_j^{(k)}$; this represents the probability that the daughter process of a replaced individual becomes extinct, and will be denoted by $\bm{\alpha}^{(k)} \bm{\bar{q}}^{(k)} := \sum^k_{j=1} \alpha^{(k)}_j \bar{q}_j^{(k)}$ in the sequel. The computation of $\bm{\bar{q}}^{(k)}$ again reduces to finding the extinction probability vector of a MGWBP with a finite type set. 
\soph{The pseudo-code for the computation of the three sequences $\{\bm{{q}}^{(k)}\}$, $\{\bm{\tilde{q}}^{(k)}\}$, and $\{\bm{\bar{q}}^{(k)}\}$ is provided in Appendix A.}

The goal of the next section is to determine sufficient conditions for the convergence of the sequence $\{\bm{\bar{q}}^{(k)}\}_{k\geq 1}$ to  $\vc q$. { Unlike} the sequences $\{\vc q^{(k)}\}$ and $\{\tilde{\vc q}^{(k)}\}$, the convergence of the sequence $\{\bm{\bar{q}}^{(k)}\}$ may not be monotone. We show in the next lemma that $\{\bm{\bar{q}}^{(k)}\}$ is however always caught between $\{\vc q^{(k)}\}$ and $\{\tilde{\vc q}^{(k)}\}$.

\begin{lemma}\label{lem1}
For any $k \geq 1$ and replacement distribution $\bm{\alpha}^{(k)}$,
\[
\bm{q}^{(k)} \leq \bm{\bar{q}}^{(k)} \leq \bm{\tilde{q}}^{(k)}.
\]
\end{lemma}
\begin{proof} From the pathwise construction of the branching processes, it is clear that
$$\{\omega: \lim_{n\to\infty} |{\vc Z}_n^{(k)}(\omega)|=0\}\subseteq \{\omega: \lim_{n\to\infty} |\bar{\vc Z}_n^{(k)}(\omega)|=0\}\subseteq \{\omega: \lim_{n\to\infty} |\tilde{\vc Z}_n^{(k)}(\omega)|=0\},$$ and the result follows.
%{ Fix $k\geq 1$ and suppose that the process $\{ \bm{Z}_n^{(k)}\}_{n\geq 1}$ dies. Then with probability one $\{\bm{Z}_n\}$ dies without producing a type in $T^c_k$, which in turn implies that $\{ \bm{\bar{Z}}_n^{(k)} \}$ dies.} Thus,
%\[
%{\left\{  \lim_{n\to\infty} |\bm{Z}_n^{(k)} |=0 \right\} \subseteq \left\{ \lim_{n\to\infty} | \bm{\bar{Z}}_n^{(k)} |=0 \right\}\,.}
%\]
%Next, by construction, we have $| \bm{\bar{Z}}^{(k)}_n| \geq |\bm{\tilde{Z}}^{(k)}_n|$ a.s. for all $n \in \mbN$. Thus,
%\[
%{ \left\{ \lim_{n\to\infty} | \bm{\bar{Z}}_n^{(k)} |=0 \right\} \subseteq \left\{ \lim_{n\to\infty} | \bm{\tilde{Z}}_n^{(k)} |=0 \right\}\,,}
%\]
%{ and the result follows.}
\end{proof}

\begin{corollary}\label{qint} 
For any sequence $\{ \bm{\alpha}^{(k)} \}$ of replacement distributions,
$$ \vc q\leq \liminf_{k \to \infty} \bm{\bar{q}}^{(k)}\leq \limsup_{k \to \infty} \bm{\bar{q}}^{(k)}\leq \bm{\tilde{q}}.$$
%\[
%\limsup_{k \to \infty} \bm{\bar{q}}^{(k)} \in [ \bm{q}\,, \bm{\tilde{q}} ] \qquad \text{and}\qquad
%\liminf_{k \to \infty} \bm{\bar{q}}^{(k)} \in [ \bm{q}\,, \bm{\tilde{q}} ]\,.
%\]
\end{corollary}
\begin{proof} The result is immediate by Lemma \ref{lem1} since $\bm{q}^{(k)} \to \bm{q}$ and $\bm{\tilde{q}}^{(k)} \to \bm{\tilde{q}}$.\end{proof}

A consequence of Corollary \ref{qint} is that, when it exists, the limit of the sequence $\{\bm{\bar{q}}^{(k)}\}$ can only overestimate the probability of global extinction. In Section \ref{motivating}, we illustrate a situation where $\vc q<\lim_{k\to\infty} \bm{\bar{q}}^{(k)}<\bm{\tilde{q}}$.

%We shall make use of these facts in Examples 1, 2 and 3. 
%Additionally, in Example 1 we will use this process in a slightly different way, by noting that if $G_{E_n^{(i)}(\bm{V})}(\cdot)$ is the progeny generating function of $\{E_n^{(i)}(\bm{V})\}$ then the probability that type $i$ becomes extinct in $\{ \bm{V}_n \}$ is the minimal non-negative solution of the extinction equation
%\begin{equation}\label{ext_emb}x=G_{E_n^{(i)}(\bm{V})}(x)\,,\quad {\text{for } 0\leq x\leq 1.}
%\end{equation} In general, an explicit expression for $G_{E_n^{(i)}(\bm{V})}(\cdot)$ may be difficult to derive, however there are particular cases where it can be obtained, as in Example~1.

%%%%%%%%%%%%%%%%%%%%%%%
\tikzset{
  treenode/.style = {align=center, inner sep=0pt, text centered,
    font=\sffamily},
  arn_n/.style = {treenode, circle, black, font=\sffamily\bfseries, draw=black,
    fill=white, text width=1.5em},% arbre rouge noir, noeud noir
  arn_r/.style = {treenode, circle, red, draw=red, 
    text width=1.5em, very thick},% arbre rouge noir, noeud rouge
  arn_x/.style = {treenode, circle, white, font=\sffamily\bfseries, draw=black,
    fill=black, text width=1.5em}% arbre rouge noir, nil
}

%%%%%%%%%%%%%%%%%%%%%%%%%END  NEW VERSION%%%%%%%%%%%%%%%%%%%%%%%%%%%%%%%
\section{Sufficient conditions for the convergence of $\{{\bar{q}}^{(k)}\}$ to $q$}\label{suff_cond}

In this section, we assume that the sequence of replacement distributions $\{\vc\alpha^{(k)}\}$ satisfies a property slightly more general than tightness,
% in {\blue\cite[Theorem 4.1]{Gib87}}\footnote{tightness in \cite{Gib87b} is on the stationary distribution $\bm{\pi}$. Also should we remove this or is tightness of a distribution not a well known thing? Seneta first considered this problem in ``Computing the stationary distribution for infinite Markov Chains'' (1980) should we reference this instead of \cite{Gib87b}?{\magenta I think we should indeed remove the reference here, but keep the fact that our condition is weaker than tightness}}, 
 that is, 
\begin{assumption}\label{tight}There exist {constants} $N_1,N_2 \geq 1$ and $a>0$, all independent of $k$, such that
\begin{equation*}\label{alpha condition}
\sum^{\min\{ N_1, k\} }_{i =1} \alpha_i^{(k)} \geq a\qquad \mbox{for all $k \geq N_2$.}
\end{equation*}
\end{assumption} 
{Situations where Assumption \ref{tight} fails to hold include $\bm{\alpha}^{(k)}=\vc e_k$ and $\bm{\alpha}^{(k)}=\vc 1/k$. These special cases will be considered in Section \ref{other_replace}. Replacement with a fixed type, however, satisfies Assumption \ref{tight}; for example, when $\bm{\alpha}^{(k)}=\vc e_1$, it holds with $N_1=1$, $N_2=1$ and $a=1$.} An example of sequence of replacement distributions satisfying Assumption \ref{tight} but which is not tight is $\vc\alpha^{(k)}=(a,0,\ldots,0,1-a)$ for some $0<a<1$.
%Situations where Assumption \ref{tight} fails to hold include the case $\bm{\alpha}^{(k)}=\vc e_k$, since for any $N_1,$ we have $\sum^{N_1 }_{i =1} \alpha_i^{(k)}=0$ for all $k\geq N_1+1$. The case $\bm{\alpha}^{(k)}=\vc 1/k$ also violates Assumption \ref{tight}, since for any $N_1,$ we have $\sum^{N_1 }_{i =1} \alpha_i^{(k)} \leq N_1/k \to 0$ as $k \to \infty$. Special cases where Assumption \ref{tight} does not hold will be considered in Section \ref{other_replace}. Replacement with a fixed type, however, satisfies Assumption \ref{tight}; for example, when $\bm{\alpha}^{(k)}=\vc e_1$, with $N_1=1$, $N_2=1$ and $a=1$.

\subsection{A motivating example}\label{motivating}
Assumption \ref{tight} alone is not a sufficient condition for the convergence of $\{\vc{\bar{q}}^{(k)}\}$ to the global extinction probability $\vc q$, as we illustrate in the next example in which individuals are {replaced by type 1 with probability one}.

\medskip
\noindent \textbf{Example 1.} Consider a two-parameter {irreducible} branching process $\{\bm{Z}_n \}$ with countably many types where, at death,  type-1 individuals produce a single type-2 individual with probability $a>0$ and no offspring with probability $1-a$, and each type-$i \in \{ 2, 3, ...\}$ individual produces a single type-$(i+1)$ offspring with probability one and a further Poisson$\left( b^{i-1} \right)$ type-1 individuals, where $0<b<1$. The progeny generating function of this process is thus given by 
$$G_1(\vc s)=a\,s_2+1-a,$$
%{
%\begin{equation*}
%p_{1 \bm{j} } = 
%\begin{cases}
%a & \text{ if }\bm{j}= \bm{e}_2, \\
%1-a & \text{ if } \bm{j}= \bm{0}, \\
%0 &\text{ otherwise},
%\end{cases}
%\end{equation*}
%}
and for $i \geq 2$,
$$G_i(\vc s)=\sum_{k\geq 0} \dfrac{(b^{i-1})^k}{k!}e^{-b^{i-1}} s_1^k\,s_{i+1}=\exp\{b^{i-1}(s_1-1)\}\,s_{i+1}.$$The corresponding mean progeny representation graph is shown in Figure~\ref{Example1}.
%{
%\begin{equation*}
% p_{i\bm{j}}= 
% \begin{cases}
% \dfrac{(b^{i-1})^{j_1}}{j_1!} e^{ b^{i-1}}& \text{ if }\bm{j}=j_1\bm{e}_1+\bm{e}_{i+1}, \\
%0& \text{ otherwise}.
% \end{cases}
%\end{equation*}
%}
%In this example it may be helpful to think of types as representing age (years). That is, $\{ \bm{Z}_n \}$ represents a population of `beings', where at each generation, 1 year olds
%\begin{itemize}
%\item die with probability $1/3$ and
%\item reach 2 years old with probability $2/3$,
%\end{itemize} 
%and $i \geq 2$ year olds
%\begin{itemize}
%\item reach $i+1$ years old with probability 1 (these beings cannot die) and
%\item give birth to a further Poisson$\left( (2/3)^{i-1} \right)$ 1 year olds.
%\end{itemize}
%Suppose the population initially contains a single 1 year old, $\varphi_0=1$. Then it is clear that the probability of global extinction is $q_1=1/3$.
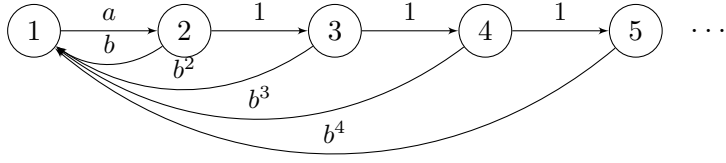
\begin{figure}
\centering
\begin{tikzpicture}
\tikzset{vertex/.style = {shape=circle,draw,minimum size=1.3em}}
\tikzset{edge/.style = {->,> = latex'}}
% vertices
\node[vertex] (1) at  (0,0) {\small 1};
\node[vertex] (2) at  (2,0) {\small 2};
\node[vertex] (3) at  (4,0) {\small 3};
\node[vertex] (4) at  (6,0) {\small 4};
\node[vertex] (5) at  (8,0) {\small 5};

%edges

\draw[edge,above] (1) to node {\footnotesize $a$ } (2);
\draw[edge,above] (2) to node {\footnotesize $1$ } (3);
\draw[edge,above] (3) to node {\footnotesize $1$ } (4);
\draw[edge,above] (4) to node {\footnotesize $1$ } (5);

\draw[edge,above] (2) to[bend left=35] node {\footnotesize $b$ } (1);
\draw[edge,above] (3) to[bend left=35] node {\footnotesize $b^2$ } (1);
\draw[edge,above] (4) to[bend left=37] node {\footnotesize $b^3$ } (1);
\draw[edge,above] (5) to[bend left=40] node {\footnotesize $b^4$ } (1);

\node at (9,0) {$\dots$};

%\draw[line width=1.0pt, edge,below] (8) to[bend right=20]  (7);
%\draw[line width=1.0pt,edge,below] (8) to[bend right=10]  (6);
%\draw[line width=1.0pt,edge,below] (8) to[bend right=0]  (5);
%\draw[line width=1.0pt,edge,below] (8) to[bend left=10]  (4);
%\draw[line width=1.0pt,edge,below] (8) to[out=-80, in=-60, looseness=1.5]  (3);
%\draw[line width=1.0pt,edge,below] (8) to[out=-230, in=50, looseness=.5]  (2);

\end{tikzpicture}
\caption{\label{Example1}The mean progeny representation graph corresponding to Example 1.}
\end{figure}

{We} assume that the population initially {contains a single} individual of type $1$.  
{ Note }that the probability of global extinction is $q_1=1-a$ and $q_i=0$ for all $i \geq 2$. 
%{ Indeed,} when $i \geq 2$, each type-$i$ individual has a single type-$(i+1)$ child with probability one. Thus, $q_i=0$ for all $i \geq 2$, and if $\{ \bm{Z}_n : \varphi_0=1 \}$ survives the first generation, which happens with probability $a$, then it almost surely lives forever\footnote{\color{red}I would remove this explanation, which is rather obvious}.
%{\blue In this example, if type 1 becomes extinct then every type becomes extinct.}\footnote{Remove the entire sentence?We no longer need to justify this as it follows from section 3.2.}
%{ Zucca} observes more generally in \cite{Zuc11} that for any irreducible branching process with countably many types, the partial extinction event $\{\forall l{\geq 1} : \lim_{n \to \infty} Z_{n,l}=0\}$ is equivalent to the local extinction event $\{ \lim_{n \to \infty} Z_{n,l}=0\mbox{ for at least one type }l\}$. Thus, 
{By irreducibility,
${\tilde{q}}_1$ is equal to the extinction probability of the  embedded type-$1$ process, $\{ E_n^{(1)}(\bm{Z}) \}$, that is, ${\tilde{q}}_1$ is the minimal nonnegative solution to 
%$$
%G_1(x)=1-a+a F(x)\,,\quad {\text{for } 0\leq x\leq 1\,,}
%$$ 
%$$G_{ E_n^{(1)}(\bm{Z})}(x)=1-a+a F(x)\,,\quad {\text{for } 0\leq x\leq 1\,,}
%$$ 
$$x=1-a+a F(x)\,,\quad {\text{for } 0\leq x\leq 1\,,}
$$
where} $F(\cdot)$ is the probability generating function (p.g.f) of a sum of countably infinitely many independent Poisson random variables with respective parameters $b^{i-1}$, for $i\geq 2$, {and is given by}
\begin{equation}\label{F}F(x)=\prod_{i\geq 2}^\infty \exp\{b^{i-1}(x-1)\}=\exp\{b(1-b)^{-1}(x-1)\}.\end{equation} Hence $F(\cdot)$ is the p.g.f. of a Poisson random variable with parameter $b/(1-b)$. The corresponding mean progeny $m_{ E_n^{(1)}(\bm{Z})}=ab/(1-b)$ indicates that ${\tilde{q}}_1=1$ if and only if $a\leq (1-b)/b,$ in which particular cases $1-a=q_1<{\tilde{q}}_1$. {The left panel in} Figure \ref{f1} shows the difference $\tilde{q}_1-q_1$ as a function of the parameter values.
%, { which} highlights the fact that the partial and global extinction probabilities may { differ} even in the irreducible case.

\begin{figure}
\centering
%\begin{tabular}{ccc}
%    \includegraphics[scale=.13]{ex1_paper_0}&
%    \includegraphics[scale=.13]{ex1_paper_2}&
%    \includegraphics[scale=.13]{ex1_paper_3}
%\end{tabular}
  \includegraphics[scale=.135]{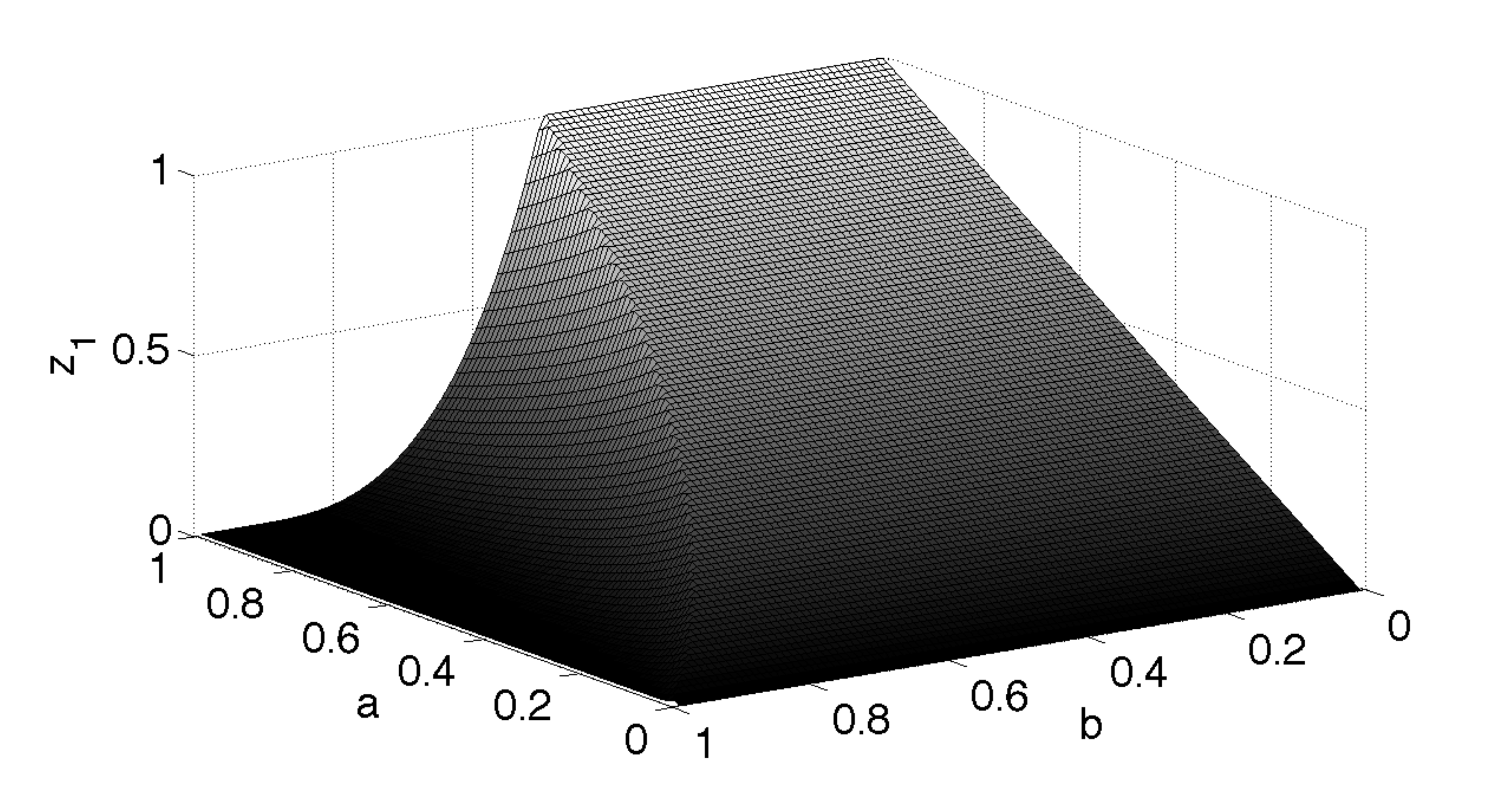}\hspace{-6mm}
    \includegraphics[scale=.135]{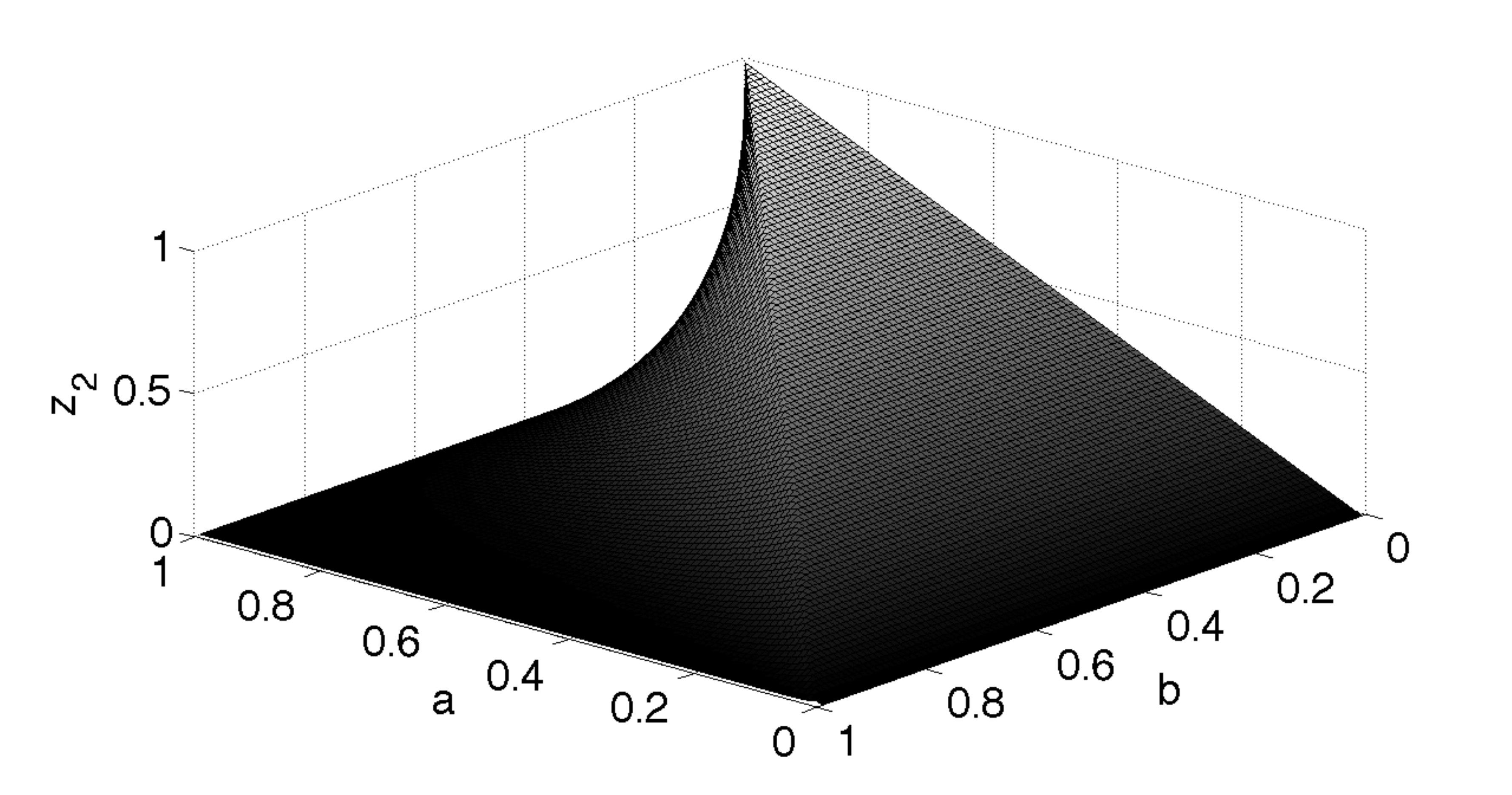}\hspace{-6mm}
    \includegraphics[scale=.135]{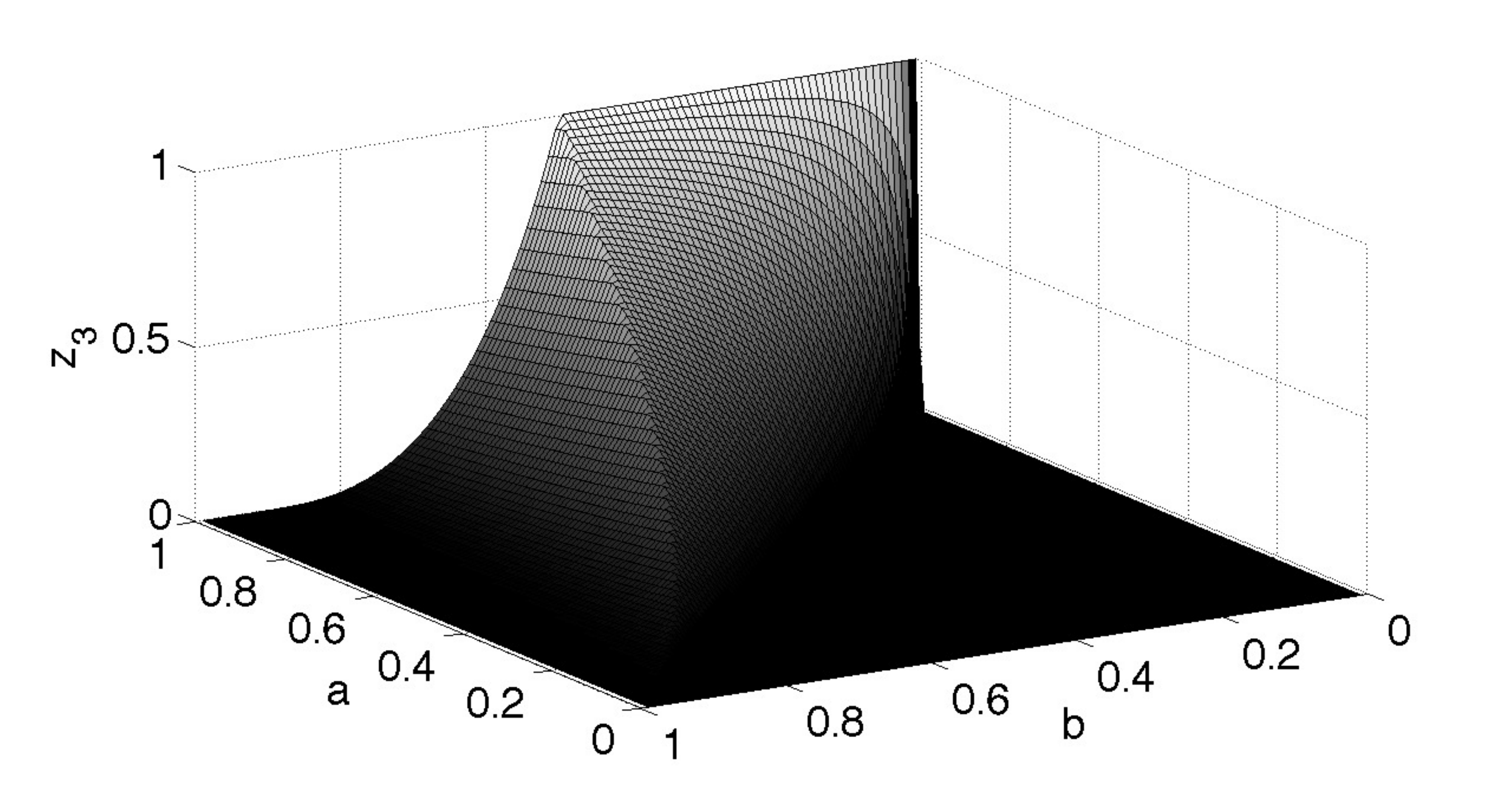}
 \caption{Differences in the extinction probabilities related to Example 1, plotted as a function of the parameters $a$ and $b$. Left panel: $z_1=\tilde{q}_1-q_1$. Middle panel: $z_2=\bar{q}_1-q_1$. Right panel: $z_3=\tilde{q}_1-\bar{q}_1$. }
    \label{f1}
\end{figure}

Now consider the process $\{ \bm{\bar{Z}}^{(k)}_n \}$ for $\vc\alpha^{(k)}=\vc e_1$, and its conditional extinction probability $\bar{q}^{(k)}_1$, given that $\varphi_0=1$. This irreducible branching process has finitely many types, hence it becomes extinct if and only if type 1 becomes extinct. {Thus,}  $\bar{q}^{(k)}_1$ corresponds to the extinction probability of the type-$1$ process embedded with respect to $\{\bm{\bar{Z}}^{(k)}_n \}$. 
The progeny generating function of $\{ E_n^{(1)}(\bar{\vc Z}^{(k)}) \}$, that we denote by $G_{1,k}(\cdot)$, is given by 
$$
G_{1,k}(x)=1-a+aF_k(x)x\,,\quad {\text{for } 0\leq x\leq 1\,,}
$$ 
where 
$$
F_k(x)=\prod_{i\geq 2}^{{k}} \exp\{b^{i-1}(x-1)\}=\exp\{b(1-b^{{k-1}})(1-b)^{-1}(x-1)\}
$$ is the p.g.f of a Poisson random variable with parameter $b(1-b^{{k-1}})(1-b)^{-1}$. 
Note that here we multiply $aF_k(x)$ by $x$ to account for the type $k+1$ descendant (instantaneously replaced by type 1) of each type 1 individual that has a type $2$ offspring.
By continuity of $G_{1,k}(\cdot)$, the limit $\bar{q}_1:=\lim_{k\to\infty}\bar{q}^{(k)}_1$ is the minimal nonnegative solution of  $${x=1-a+aF(x)x,}$$ where $F(x)$ is given in \eqref{F}. 
The corresponding mean progeny $m_{1,\infty}=G'_{1,\infty}(1)=a(1+b/(1-b))=a+m_{ E_n^{(1)}(\bm{Z})}$ indicates that $\bar{q}_1=1$ if and only if $a\leq 1-b,$ in which particular cases $\bar{q}_1>q_1=1-a$. { The middle panel in }Figure \ref{f1} shows the difference $\bar{q}_1-q_1$ as a function of the parameter values. This highlights the fact that the sequence $\{\bar{\vc q}^{(k)}\}$ does not always converge to the global extinction probability $\vc q$.

For completeness, we also provide in {the right panel of} Figure \ref{f1} the difference $\tilde{q}_1-\bar{q}_1$ as a function of the parameter values. From above, we have that if $a\leq 1-b$ then $q_1<\tilde{q}_1=\bar{q}_1=1$, so the sequence $\{\bar{\vc q}^{(k)}\}$ can potentially converge to the partial extinction probability.

\subsection{The seed process}

Example 1 illustrates the need to further explore the conditions under which $\{\bm{\bar{q}}^{(k)}\}$ converges to $\bm{q}$ as $k\to\infty$. 
%The conditions which are established are closely related to \soph{sufficient conditions for the original process $\{ \bm{Z}_n \}$ to satisfy the dichotomy property}.
Observe that, for any $k\geq 1$, we have 
\begin{align}
\bar{q}_i^{(k)}&=\mbP_i\left(\lim_{n\to\infty}|\bm{\bar{Z}}_n^{(k)} |=0\right)\nonumber\\
&= \mbP_i\left(\lim_{n\to\infty}|\bm{\bar{Z}}_n^{(k)} |=0\,,  \lim_{n\to\infty}|\bm{Z}_n^{(k)} |=0\right)\nonumber\\
&\qquad+\mbP_i\left( \lim_{n\to\infty}|\bm{\bar{Z}}_n^{(k)} |=0\,, \lim_{n\to\infty}|\bm{Z}_n^{(k)} |>0\right)\nonumber\\
%&= q_i^{(k)}+\mbP_i\left( \lim_{n\to\infty}|\bm{\bar{Z}}_n^{(k)} |=0\,, \lim_{n\to\infty}|\bm{Z}_n^{(k)} |>0\right)\nonumber\\
\label{eq1}&= q_i^{(k)}+\sum_{x\geq 1} \left(\vc\alpha^{(k)}\bar{\vc q}^{(k)}\right)^x \mbP_i\left(|\bm S_k|=x\right),
\end{align}where $|\bm S_k|$ denotes the (finite) number of sterile types produced over the lifetime of $\{ \bm{\tilde{Z}}^{(k)}_n \}$, which are replaced by some random types in $\{ \bm{\bar{Z}}^{(k)}_n \}$ and immortal types in $\{ \bm{Z}_n^{(k)}\}$. To understand Equation \eqref{eq1} one may think of simulating the branching processes with $\varphi_0=i$ in two stages: by first constructing the path of $\{ \bm{\tilde{Z}}_n^{(k)} \}$, and then constructing those of $\{ \bm{{Z}}_n^{(k)} \}$ and $\{ \bm{\bar{Z}}_n^{(k)} \}$ by taking the outcome of $\{ \bm{\tilde{Z}}^{(k)}_n \}$, replacing the sterile individuals, and simulating their daughter processes according to the respective replacement and updating rules. Conditional on the first stage of simulation, there are two ways in which $\{ \bm{\bar{Z}}^{(k)}_n \}$ can die, either: $(i)$ $\{ \bm{\tilde{Z}}^{(k)}_n \}$ dies before producing a sterile type, in which case $\{ \bm{{Z}}^{(k)}_n \}$ also dies (this occurs with probability $q^{(k)}_{i}$), or $(ii)$ $\{ \bm{\tilde{Z}}^{(k)}_n \}$ dies after producing $1\leq x <\infty$ sterile individuals, in which case $\{\bm{\bar{Z}}^{(k)}_n \}$ dies with probability $( \bm{\alpha}^{(k)} \bm{\bar{q}}^{(k)} )^x$ in the second stage of simulation. 

Because $\bm{q}^{(k)} \to \bm{q}$, this generally indicates that in order for $\{\bm{\bar{q}}^{(k)}\}$ to converge to $\bm{q}$, we need to avoid cases where there is a positive probability that the number of sterile individuals produced over the lifetime of $\{ \bm{\tilde{Z}}_n^{(k)} \}$ remains positive and uniformly bounded for all $k$. {This is not satisfied} in Example 1 {as,} for any $a<1$, $0<b<1$ and for all $k\geq 1$,
\begin{equation*}
\mbP_1\left(|\bm S_k|=1\right)  > aF(0) >0.
\end{equation*}
We defer {a formal} statement of this idea until Lemma \ref{nconv} and now formally introduce the \emph{seed process} $\{ \bm{S}_k \}_{k\geq 1}$, defined from the paths of $\{ \bm{\tilde{Z}}^{(k)}_n \}$.

\begin{definition}\label{Seed def}
The seed process $\{\bm{S}_k=(S_{k,1},S_{k,2},...)\}_{k \geq 1}$ defined on {$( \Omega, \mathcal{F}, \mbP)$} is such that for any $\omega\in\Omega$, if $\lim_{n\to\infty}\bm{\tilde{Z}}^{(k)}_n(\omega)=\vc 0$, then
\begin{equation}
S_{k,i}(\omega) = 
\begin{cases}
0, & \text{if } i\leq k, \\
\sum^{\infty}_{n=1} {\tilde{Z}}^{(k)}_{n,i}(\omega), & \text{if } i>k,
\end{cases}
\end{equation}
whereas, if $\lim_{n\to\infty}\bm{\tilde{Z}}^{(k)}_n(\omega)>\vc 0$, then
\begin{equation}\label{catastrophe}
S_{k,i}(\omega) =0\quad \mbox{for all {$i \geq 1$}}.
\end{equation}
\end{definition}

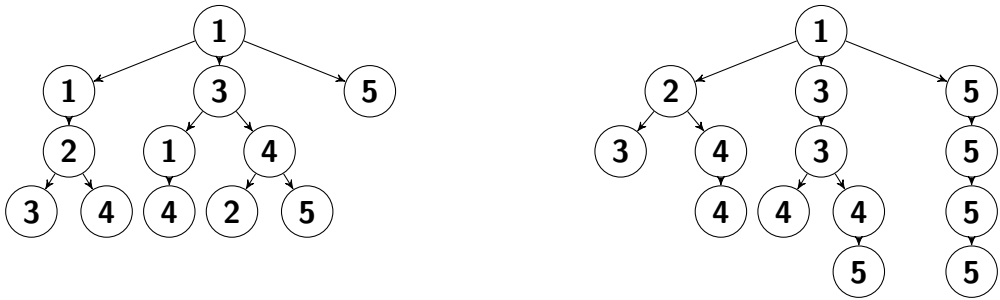
\begin{figure}
\vspace{-0.6cm}
\begin{tikzpicture}[->,>=stealth',level/.style={sibling distance = 4cm/#1,
  level distance = 0.8cm}] 
 \node [white] {1}   % artificial to have the two trees at the same level
[white]  child{ node [arn_n] {1}
   child{ node [arn_n] {1} 
           [black]  child{ node [arn_n] {2}
            	     child{ node [arn_n] {3}}   
	     	     child{ node [arn_n] {4}}       								
            }                          
    }
  child{ node [arn_n] {3}
       [black]        child{ node [arn_n] {1} 
							child{ node [arn_n] {4}}							
            }
            child{ node [arn_n] {4}
							child{ node [arn_n] {2}}
							child{ node [arn_n] {5}}
            }     
		}
		child{ node [arn_n] {5}
       [black]  }
%child{node [arn_n] {5} [white]
%  child{ node [white] {2}}
%  child{ node [white] {2}} }
  }
  child{ node [white] {1}} % artificial to put space between the two trees
child{ node [arn_n] {1}
    child{ node [arn_n] {2}            
            [black]   child{ node [arn_n] {3}}   
	     child{ node [arn_n] {4}
	     	child{ node [arn_n] {4}}
	     }       								                                    
    }
    child{ node [arn_n] {3}
        [black]     child{ node [arn_n] {3} 
		child{ node [arn_n] {4}}
		child{ node [arn_n] {4}
			child{ node [arn_n] {5}}	
		}							
            }
		}
   child{ node [arn_n] {5}
 [black]   child{ node [arn_n] {5}
   child{ node [arn_n] {5}
   child{ node [arn_n] {5}
   }
   }
   }
   }
   }
 ;  
\end{tikzpicture} 
%\\[+10pt]
%\begin{tikzpicture}[->,>=stealth',level/.style={sibling distance = 3cm/#1,
%  level distance = 1.25cm}] 
%\node [arn_n] {1}
%    child{ node [arn_n] {2}            
%              child{ node [arn_n] {3}}   
%	     child{ node [arn_n] {4}
%	     	child{ node [arn_n] {4}}
%	     }       								                                    
%    }
%    child{ node [arn_n] {3}
%            child{ node [arn_n] {3} 
%		child{ node [arn_n] {4}}
%		child{ node [arn_n] {4}
%			child{ node [arn_n] {5}}	
%		}							
%            }
%		}
%   child{ node [arn_n] {5}
%   child{ node [arn_n] {5}
%   child{ node [arn_n] {5}
%   child{ node [arn_n] {5}
%   }
%   }
%   }
%   }
%;  
%\end{tikzpicture}
\caption{\label{seedevo}A realisation of $\{ \bm{Z}_n \}$ (left). { The corresponding seed process, as introduced in Definition \ref{Seed def}, is $\bm{S}_0=(1,0,0,\ldots)$, $\bm{S}_1=(0,1,1,0,1,\hdots)$, $\bm{S}_2=(0,0,2,1,1,\hdots)$, $\bm{S}_3=(0,0,0,3,1,\hdots)$, $\bm{S}_4=(0,0,0,0,2,\hdots)$ and $\bm{S}_5=(0,0,0,0,0,\hdots)$. The right graph illustrates the tree representation of $\{ \bm{S}_k \}$.}}
\end{figure}

We take the convention that $|\vc S_{k}|=0$ when $\{ \bm{\tilde{Z}}^{(k)}_n \}$ does not become extinct in order to ensure that $ \mbP_{i}(| \bm{S}_k | < \infty)=1$ for all fixed $i$ and $k$. It is not hard to show that $\{\bm{S}_k\}$ forms a Markov chain on $\mathbb{N}^\infty$. More precisely, $\{\bm{S}_k\}$ is a branching process with countably many types, in which the progeny distribution depends on the generation (branching process in varying environment), and which can undergo {total catastrophe}. \soph{Such a total catastrophe happens at generation $k+1$ in the seed process for some $\omega\in \Omega$ if $\bm{\tilde{Z}}^{(k)}_n(\omega)$ becomes extinct while $ \bm{\tilde{Z}}^{(k+1)}_n(\omega) $ survives.}
Like $\{ \bm{Z}_n \}$, each outcome of $\{ \bm{S}_k \}$ can be represented as a tree. An illustration of this is given in Figure \ref{seedevo} where an outcome of $\{ \bm{Z}_n \}$ is given along with the corresponding outcome of $\{ \bm{S}_k \}$. 
{Observe that in the nearest neighbour branching random walk, when $\bm{\tilde{q}}=\bm{1}$, the seed process reduces to the first modified process used in the proof of  \cite[Theorem 2.9]{Gan10}}.
{In our generalised construction}, when $\{ \bm{Z}_n \}$ dies, individuals in the $k$th generation of the seed process with type strictly greater than $k+1$ produce only one exact copy of themselves, and the number of generations $\{ \bm{S}_k \}$ lives is equivalent to the largest type produced in $\{ \bm{Z}_n \}$.

%only individuals with type equal to the current generation produce a random number of progeny and the number of generations $\{ \bm{S}_k \}$ lives is equivalent to the largest type produced in $\{ \bm{Z}_n \}$. 

%Using the probabilistic interpretation of the sequence $\{\vc q^{(k)}\}$, we can show that the conditional distribution of the largest type $L$ in $\{ \bm{Z}_n \}$ is given by\footnote{\color{red} Is this used later on? I do not see how this formula is derived.}
%$$
%\mbP\left(L=\ell\, |\,\varphi_0=i,{\lim_{n\to\infty} |\bm{Z}_n|=0}\right)=(q_i^{(\ell)}-q_i^{(\ell-1)})/q_i.
%$$

{The seed process enjoys several other properties which will be exploited to prove {Theorem \ref{general converge}} on the convergence of $\{\bar{\vc q}^{(k)}\}$ to $\vc q$ stated in the next subsection.}
%We {\magenta first} {show} that $\bm{0}$ is an absorbing state in $\{ \bm{S}_k\}$ as it will be used to prove further results.\footnote{\magenta Maybe we can remove this sentence?}

\begin{lemma}\label{absorb}
{ The state $\bm{0}$ is absorbing for the seed process $\{ \bm{S}_k \}$.}
\end{lemma} 

\begin{proof} Suppose $\vc S_k(\omega)=\vc 0$ for some $\omega\in\Omega$. Then either 
$\lim_{n\to\infty} |\tilde{\vc Z}_n^{(k)}(\omega)|>0$ or $\lim_{n\to\infty} |{\vc Z}_n^{(k)}(\omega)|=0$. In addition, by construction, $$\{\omega: \lim_{n\to\infty} |\tilde{\vc Z}_n^{(k)}(\omega)|>0\}\subseteq \{\omega: \lim_{n\to\infty} |\tilde{\vc Z}_n^{(k+1)}(\omega)|>0\}$$ and 
$$\{\omega: \lim_{n\to\infty} |{\vc Z}_n^{(k)}(\omega)|=0\}\subseteq \{\omega: \lim_{n\to\infty} |{\vc Z}_n^{(k+1)}(\omega)|=0\},$$which implies $\vc S_{k+1}(\omega)=\vc 0$.
%
% Suppose $\bm{S}_k=\bm{0}$ and $\{ \bm{\tilde{Z}}^{(k)}_n \}$ does not die. By construction $\bm{\tilde{Z}}^{(k+1)}_n \geq \bm{\tilde{Z}}^{(k)}_n$ a.s. for all $n \in \mbN$. This implies $\{ \bm{\tilde{Z}}^{(k+1)}_n \}$ does not die, and $\bm{S}_{k+1}= \bm{0}$ with probability one. 
%Now suppose $\bm{S}_k=\bm{0}$ and $\{ \bm{\tilde{Z}}^{(k)}_n \}$ dies. Then in $\{ \bm{\tilde{Z}}^{(k)}_n \}$ there are no sterile type $j \in T^c_k$ progeny, which is equivalent to $\tau_k=\infty$ a.s. Since $\tau_{k+1} \geq \tau_k$ we have $\tau_{k+1}=\infty$ a.s. Hence there are no sterile type $j \in T^c_k$ progeny in $\{ \bm{\tilde{Z}}^{(k+1)}_n \}$. This gives $\bm{S}_{k+1}=\vc 0$ with probability one. 
\end{proof}

Additionally, we obtain an expression for the probability that $\{ \bm{S}_k \}$ has reached the absorbing state by generation $k$ in terms of the extinction probabilities of $\{ \bm{\tilde{Z}}^{(k)}_n \}$ and $\{ \bm{{Z}}^{(k)}_n \}$:

\begin{lemma}\label{absorb2}
$\mbP\left(|\bm S_k|=0\,|\,\varphi_0=i\right)=1-\tilde{q}_i^{(k)}+q_i^{(k)}$.
\end{lemma}  
\begin{proof}{By the same argument as in the proof of Lemma \ref{absorb},
% By definition, $|\bm S_k|=0$ if and only if either $\{\bm{\tilde{Z}}_n^{(k)}\}$ does not become extinct, or dies out and has no leaves in $T_k^c$. The latter condition is equivalent to the process $\{\bm{Z}_n^{(k)}\}$ dying out. Thus,
\[
\mbP_i\left(|\bm S_k|=0\right)=\mbP_i\left(  \left\{\lim_{n\to\infty}|\bm{\tilde{Z}}_n^{(k)}|>0\right\} \cup \left\{
\lim_{n\to\infty} |\bm{{Z}}_n^{(k)}|=0  \right\}\right)\,,
\]
where the two events are mutually exclusive.}\end{proof}

This provides us with a condition for the global and partial extinction probabilities to coincide, 
\begin{corollary}\label{extcrit}
{For all $i\geq 1$, the following two statements are equivalent
\begin{itemize}
\item[{\it (i)}] $q_i=\tilde{q}_i$
\item[{\it (ii)}] $\mbP_i\left(\lim\limits_{k\to\infty}|  \bm{S}_k | = 0\right)=1.$
\end{itemize}
}
\end{corollary}
%Corollary \ref{extcrit} provides us with a new insight into a global extinction criterion when partial extinction occurs almost surely. Indeed, in some cases the extinction probability of the seed process can be easier to study than that of the original branching process.
%
%%\subsection{Dichotomy of the seed process}
%\medskip
%
%We now return to {the primary focus of this section} which is to establish conditions under which $\{\bar{\bm{q}}^{(k)}\}$ converges to $\vc q$.
%Since $\bm{q}^{(k)} \to \bm{q}$, { we see that }$\bm{\bar{q}}^{(k)} \to \bm{q}$ is equivalent to $\bm{\bar{q}}^{(k)}- \bm{q}^{(k)} \to \bm{0}$, { and we focus on deriving conditions under which the latter statement holds true. 
We rewrite equation \eqref{eq1} as
\begin{equation}\label{CouplingEq}
\bar{q}_{i}^{(k)}-q_{i}^{(k)} = \mbE_i  \left( \left( \vc\alpha^{(k)}\bar{\vc q}^{(k)}\right)^{|\bm{S}_k|}-\bm{1}\left\{|\bm{S}_k|=0\right\} \right).
\end{equation}
%It is clear from equation \eqref{CouplingEq} that if $|\bm{S}_k|$ has a positive probability to remain in the set $\{1,2,\dots,B\}$ for some $B \in \mbN$ as $k \to \infty$ and if $\liminf _{k\to\infty}\vc\alpha^{(k)}\bar{\vc q}^{(k)}>0$, then $\bar{q}^{(k)}_{i}-q^{(k)}_{i} \nrightarrow 0$. 
The next lemma formalises the discussion preceding Definition \ref{Seed def}.
%}

\begin{lemma}\label{nconv}Assume that there exists $B<\infty$ such that 
\[
\liminf_{k \to \infty} \mbP_i( 0 < |\bm{S}_k| < B) =\alpha>0.
\]
If, in addition, $\{\vc \alpha^{(k)}\}$ satisfies Assumption \ref{tight} for some $N_1$ such that $q_j>0$ for all $j\in\{1,\ldots,N_1\},$
%
%
%Assume that there exist {constants} $N_1,N_2 \in \mbN$ { independent of $k$}, and $a > 0$, such that
%\begin{equation*}\label{alpha condition}
%\sum^{\min\{ N_1, k\} }_{i =1} \alpha_i^{(k)} \geq a\quad \mbox{for all $k \geq N_2$,}\quad \mbox{and}\quad  q_j>0\quad \mbox{for all $j\in\{1,\ldots,N_1\}.$}
%\end{equation*}In addition, assume that there exists $B<\infty$ such that 
%\[
%\liminf_{k \to \infty} \mbP( 0 < |\bm{S}_k| < B\,|\,\varphi_0=i) =\alpha>0.
%\]
then $\liminf_{k\to\infty}\bar{q}^{(k)}_i > q_i$.
\end{lemma}
\begin{proof}
Since $q_j^{(k)} \to q_j>0$ for all $j\in\{1,\ldots,N_1\}$, there exists $\beta > 0$ and $K \in \mbN$ such that, for all $k>K$ and $j\in\{1,\ldots,N_1\}$, $q_j^{(k)} \geq \beta$. By Lemma \ref{lem1}, we also have $\bar{\vc q}^{(k)} \geq \vc q^{(k)}$ for all $k \in \mbN$. Hence
\[
\bar{q}_j^{(k)} \geq \beta  \text{ for all $k > K$ and all $j\in\{1,\ldots,N_1\}$}.
\]It follows from Assumption \ref{tight} that for any $k>\max\{K,N_1,N_2\}$,
$$\vc\alpha^{(k)}\bar{\vc q}^{(k)}\geq\sum_{j=1}^{N_1} \alpha^{(k)}_j \bar{q}_j^{(k)}\geq \beta a>0.$$
Then, by \eqref{CouplingEq},
\begin{eqnarray*}
\lefteqn{\liminf_{k \to \infty} \left( \bar{q}^{(k)}_i - q^{(k)}_i \right)}\\ &=&\liminf_{k \to \infty}  \mbE_i  \left( \left(\vc\alpha^{(k)}\bar{\vc q}^{(k)}\right)^{|\bm{S}_k|}-\bm{1}\left\{|\bm{S}_k|=0\right\} \right)\\
& \geq& \liminf_{k \to \infty} \mbE_i \left( \left(\vc\alpha^{(k)}\bar{\vc q}^{(k)}\right)^{| \bm{S}_k|} \,\Big|\, 0 < |\bm{S}_k|< {B} \right)\mbP_i \left( 0 < |\bm{S}_k| < B \right) \\
&\geq &  (\beta a)^{B} \alpha >0,
\end{eqnarray*}
which completes the proof.\end{proof}

%\color{red}
Lemma \ref{nconv} suggests that the conditions we impose for $\bm{\bar{q}}^{(k)} \to \bm{q}$  should also be sufficient for
% highlights the importance of providing sufficient conditions for the process 
 $\{ |\bm{S}_k | \}$ to satisfy the dichotomy property, that is, with probability one, either $ | \bm{S}_k | \to \infty$ as $k \to \infty$, or a value $n$ exists for which $|\bm{S}_k|=0$ for all $k \geq n$. 
%We impose a well known sufficient condition for dichotomy of $| \bm{Z}_n |$ \soph{(see \cite{Jag92,Hac05}):}
{We impose a condition similar to, but more general than, the well known sufficient condition \soph{`$\inf_i {q}_i >0$'} for $\{|\bm{Z}_n |\}$ to satisfy the dichotomy property} \soph{(see \cite{Jag92,Hac05}):} 
\begin{assumption} \label{extbd} 
{$\liminf_i q_i >0$.}
%\footnote{\red I think this makes it clearer that our conditions are more general than $\inf_i q_i>0$. Are you guys are happy with this change? If so, I'll update the parts of the document where we refer to $N_3$. \soph{Yes, I agree, however you can keep $N_3$ in the proof of the lemma.}}
%{There exists $N_3 \in \mbN$ such that,}
%\begin{equation*}
%\inf_{i \geq N_3} q_i > 0.
%\end{equation*}
\end{assumption}
Observe that Assumption \ref{extbd} is satisfied when $\liminf_{i} p_i( \bm{0} )>0$. 
%For more on dichotomy we refer the reader to \cite{Jag92} and \cite{Ste15}.
\begin{lemma}\label{SeedDich}
Suppose Assumption \ref{extbd} holds, then for all $i\in \mathcal{S}$
\begin{equation*}
\mbP_i \left( |\bm{S}_k | \to 0 \text{ or } \infty \right) =1.
\end{equation*}
\end{lemma}
\begin{proof}
%{\red By Assumption \ref{extbd} there exists $N_3 \in \mbN$ and $\beta >0$ such that $q_i>\beta$ for all $i \geq N_3$, and by Lemma \ref{absorb2} we have $\mbP_i(\lim_{k \to \infty} |\bm{S}_k| =0) \geq q_i$ for all $i$. Thus
%\[
%\mbP_i(\lim_{k \to \infty} |\bm{S}_k| =0) \geq \mbE_i(\bm{q}^{\bm{S}_{N_3}}) \geq \mbE_i\left((1-\beta)^{|\bm{S}_{N_3}|}\right)=\delta_i>0, \quad 1 \leq i \leq N_3,
%\]
%and $\mbP_i(\lim_{k \to \infty} |\bm{S}_k| =0) \geq \beta$ for all $i >N_3$. Up to the possibility of total catastrophe the individuals in $\{\bm{S}_k\}$ behave independently, which means for any $| \bm{\varphi_0} |\leq x $, we have
%\begin{equation}\label{SJDic}
%\mbP(\lim_{k \to \infty} |\bm{S}_k| =0 \, | \, \bm{S}_0 = \bm{\varphi}_0) \geq \left( \min \left\{ \inf_{1\leq i \leq N_3} \delta_i, \beta \right\} \right)^x >0.
%\end{equation}
%Combining Lemma \ref{absorb} and Equation \eqref{SJDic} with the fact that $\{\bm{S}_k\}$ is a Markov chain, we observe that the conditions of \cite[Theorem 2]{Jag92} are satisfied and the result then follows from this theorem.}
\soph{By Assumption \ref{extbd} there exist $N_3 \in \mbN$ and $\beta >0$ such that $q_i>\beta$ for all $i > N_3$, and by Lemma \ref{absorb2}, $\mbP_i(\lim_{k \to \infty} |\bm{S}_k| =0) \geq q_i$ for all $i$. Thus for all $1 \leq i \leq N_3$,
\[
\mbP_i(\lim_{k \to \infty} |\bm{S}_k| =0) \geq \mbE_i(\bm{q}^{\bm{S}_{N_3}}) \geq \mbE_i\left(\beta^{|\bm{S}_{N_3}|}\right):=\delta_i>0, 
\]
and for all $i >N_3$, $\mbP_i(\lim_{k \to \infty} |\bm{S}_k| =0) \geq \beta$. Up to the possibility of a total catastrophe, the individuals in $\{\bm{S}_k\}$ behave independently, hence for any $| \bm{s_0} |\leq x $, we have
\begin{equation}\label{SJDic}
\mbP(\lim_{k \to \infty} |\bm{S}_k| =0 \, | \, \bm{S}_0 = \bm{s}_0) \geq \left( \min \left\{ \inf_{1\leq i \leq N_3} \delta_i, \beta \right\} \right)^x >0.
\end{equation}
Combining Lemma \ref{absorb} and Equation \eqref{SJDic} with the fact that $\{\bm{S}_k\}$ is a Markov chain, the result then follows from \cite[Theorem 2]{Jag92}.}
\end{proof}

\soph{
In specific cases, the extinction probability of the seed process can be easier to analyse than that of the original branching process. In \cite{BrauHau2017} we consider one such subclass of branching processes called \textit{lower Hessenberg} where, by building upon the results of the present section, we are able to analyse the set of fixed points of the original process and derive necessary and sufficient conditions for its almost sure global extinction. 
%
%For specific classes of branching processes with countably many types, 
%%such as the one considered in \cite{BrauHau2017}, 
%the seed process enjoys additional properties which can be directly exploited to analyse the set of fixed-points of the original process, and to derive necessary and sufficient conditions for almost sure global extinction. 
%The applicability of the seed process extends beyond the scope of the present paper, and we refer the reader to \cite{BrauHau2017} for further information.
}

%\subsection{Convergence of $\{\bar{\vc q}^{(k)}\}$ to $\textbf{q}$}

%%%%%%%%%%%%%%%%%%%%%%%%END THE SEED PROCESS V2%%%%%%%%%%%%%%%%%%%%%%%%%%%%%

\subsection{Convergence to global extinction}

{In this section, we state {our result} on the pointwise convergence of the sequence $\{\bar{\vc q}^{(k)}\}$ to the global extinction probability $\vc q$. To obtain convergence, Equation \eqref{CouplingEq} suggests that, conditionally on $\mbP(| \bm{S}_k |\to\infty)>0$, one must show that $\vc\alpha^{(k)}\bar{\vc q}^{(k)}$ is bounded away from 1 for all sufficiently large $k$. To prove this, we use a regenerative argument, which may break down for some replacement distributions $\vc\alpha^{(k)}$}, such as the ones presented in the next section.

{For} a fixed $k$, each seed in $\bm{S}_k$ corresponds to a sterile individual produced over the lifetime of $\{ \bm{\tilde{Z}}^{(k)}_n \}$. To obtain $\{ \bm{\bar{Z}}^{(k)}_n \}$%\footnote{How do we construct $\{ \bm{\bar{Z}}^{(k)}_n \}$ and the embedded replacement process when $\{ \bm{\tilde{Z}}^{(k)}_n \}$ does not become extinct?}
, these seeds are replaced, independently of each other, by new individuals whose types follow the distribution $\vc\alpha^{(k)}$, and whose daughter processes themselves may be thought of as producing an i.i.d. number of new seeds, and so on. Thus, the process formed by taking all `replaced' individuals from $\{ \bm{ \bar{Z}}^{(k)}_n \}$ which correspond to seeds, and connecting each of these individuals to its nearest replaced seed ancestor in $\{ \bm{\bar{Z}}_n^{(k)} \}$, is a multitype Galton-Watson process on {$(\Omega,\mathcal{F},\mbP)$} with type space $T_k=\{1,\ldots,k\}$. We refer to this process as the \emph{embedded replacement
%\footnote{{\red I still prefer embedded seed process given we are not taking all replaced individuals} \blue I agree with you but embedded seed process is not ideal either because we don't look at the seeds but rather at the replaced seeds. So why not ``the embedded replaced seed process''? (don't laught :)) Maybe in the superscripts we could also simplify and have $(e,k)$ instead of $(er,k)$... }
 process}, and denote it as {$\{ {{\vc Z}}_n^{(e,k)}\}_{n\geq 0}$}. In {$\{ {{\vc Z}}_n^{(e,k)}\}$}, each child's type is chosen independently of the type of its parent and other siblings, and therefore the corresponding progeny generating function $\bm{G}^{(e,k)}: [0,1]^{T_k} \to [0,1]^{T_k}$
% \footnote{\blue Shouldn't we have $\mathcal{S}^{(k)}$ (defined above) as the power of $[0,1]$? (first because there may be less than $k$ types, and also because we should have the set of types and not just the number of types in the notation). In that case, $\bm{\alpha}^{(k)} \bm{s}$ would be a notation for $\sum_{i\in\mathcal{S}^{(k)}}\alpha_i^{(k)} s_i$.}
  is such that 
\begin{equation}\label{ERgen}
G^{(e,k)}_i(\bm{s})=\sum_{x \geq 0} \left(\bm{\alpha}^{(k)} \bm{s} \right)^x \mbP_i(|\bm{S}_k|=x),\quad  i\in T_k.
\end{equation}
We use the convention that $\bm{Z}_0^{(e,k)} \equiv \bm{\bar{Z}}^{(k)}_0$, that is, we include the initial individual in $\bm{Z}_0^{(e,k)}$ regardless of whether it has been replaced. The embedded replacement process can be constructed pathwise for each $\omega\in\Omega$, but we omit the details here. Conditional on the initial type $\varphi_0 \in T_{k}$, for each $\omega \in \Omega$ we have $|\bm{Z}^{(e,k)}_1(\omega)|=|\bm{S}_k(\omega)|$. {Figure \ref{embseed} gives an illustration of the construction of $\{\bm{Z}^{(e,k)}_n\}$ when compared to the corresponding realisation of $\{ \bm{ \bar{Z}}^{(k)}_n \}$ when $k=4$ and $\vc\alpha^{(4)}=\vc e_1$. In Figure \ref{embseed} the type 2 root is common to both processes and the black type-1 nodes represent individuals that have been replaced. 

It is clear that if the embedded replacement process does not become extinct then neither does $\{ \bm{\bar{Z}}^{(k)}_n \}$. Thus, 
\begin{equation}\label{ERineq}
{\bar{q}}_j^{(k)} \leq  {q}_j^{(e,k)} \quad \text{ for all } j \in T_k.
\end{equation} 
where $\bm{q}^{(e,k)}$ {is} the extinction probability vector of $\{\bm{Z}^{(e,k)}_n \}$. We are now {in a position to prove {our} result} on the pointwise convergence of the sequence $\{\bar{\vc q}^{(k)}\}$ to the global extinction probability $\vc q$.

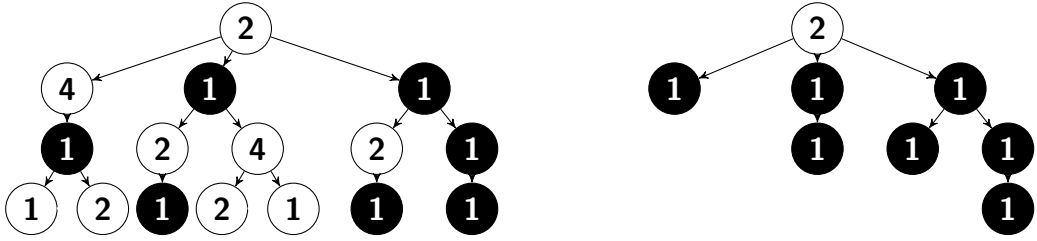
\begin{figure}
\vspace{-0.6cm}
\begin{tikzpicture}[->,>=stealth',level/.style={sibling distance = 1.9cm/#1,
  level distance = 0.8cm}] 
  \node [white] {...}
[white] child{ node [arn_n] {2}
    child{ node [arn_n] {4} 
         [black]   child{ node [arn_x] {1}
            	     child{ node [arn_n] {1}}  
            	     child{ node [white] {...} [white]
							 } 
	     	     child{ node [arn_n] {2}}       								
            }                          
    }
  child{ node [white] {...}}
    child{ node [arn_x] {1}
          [black]   child{ node [arn_n] {2} 
							child{ node [arn_x] {1}} }
            child{ node [white] {...} [white]
							 }
            child{ node [arn_n] {4}
							child{ node [arn_n] {2}}
							child{ node [white] {...} [white]
							 }
							child{ node [arn_n] {1}}
            }     
		}
	child{ node [white] {...}}
	child{ node [white] {...}}
   child{ node [arn_x] {1}
   [black] 	child{ node [arn_n] {2}
		child{ node [arn_x] {1}}	
	}
	child{ node [white] {...} [white]
							 }
	child{ node [arn_x] {1}
		child{ node [arn_x] {1}}	
	}
   }}
child{ node [white] {...}} 
child{ node [white] {...}} 
child{ node [white] {...}} %% 2 fake children
child{ node [arn_n] {2} 
child{ node [arn_x] {1} [black]}
child{ node [white] {...}} 
child{ node [black] [arn_x] {1} 
	child{ node [arn_x] {1}  }
	[black]}
child{ node [white] {...}}  
child{ node [white] [arn_x] {1} 
	child{ node [arn_x] {1} [black]}
	child{ node [white] {...} [white]} 
	child{ node [arn_x] {1} 
		child{ node [arn_x] {1}}
	[black]}
	[black]} 
	[white]}
;  
\end{tikzpicture} 
%\\[+20pt]
%\begin{tikzpicture}[->,>=stealth',level/.style={sibling distance = 3cm/#1,
%  level distance = 1.25cm}] 
%\node [arn_x] {1}
%    child{ node [arn_x] {1}}
%    child{ node [arn_x] {1}
%	child{ node [arn_x] {1}}
%    }
%   child{ node [arn_x] {1}
%	child{ node [arn_x] {1}}
%	child{ node [arn_x] {1}}
%   }
%;  
%\end{tikzpicture}
\caption{\label{embseed}A realisation of $\{ \bm{\bar{Z}}^{(4)}_n \}$ (left) with the corresponding realisation of the $\{\bm{Z}^{(e,4)}_n\}$ (right) when $\vc\alpha^{(k)}=\vc e_1$. The black nodes represent individuals with type greater than 4 which are immediately replaced with type 1.}
\end{figure}

%\begin{theorem}\label{general converge}Suppose Assumptions \ref{tight} and \ref{extbd} hold. In addition, assume that there exists $N_1$ such that either
%\begin{itemize}
%\item[(i)] $\tilde{q}_j<1$ for all $j\in\{1,\ldots, N_1\}$, or
%\item[(ii)] $\tilde{q}_j=1$ for all $j\in\{1,\ldots, N_1\}$, and there is a path from any $j\in\{1,\ldots, N_1\}$ to the initial type $i$.
%\end{itemize}
%Then 
%\[
%\lim_{k\to\infty}\bar{q}^{(k)}_{i} \to q_{i},
%\]for any initial type $i$.
%\end{theorem}

%\color{red}
\begin{theorem}\label{general converge}Suppose Assumption \ref{extbd} holds. In addition, suppose that the replacement distributions $\{\vc\alpha^{(k)}\}$ satisfy Assumption \ref{tight} with $N_1$ such that either
\begin{itemize}
\item[(i)] $\tilde{q}_j<1$ for all $j\in\{1,\ldots, N_1\}$, or
\item[(ii)] $\tilde{q}_j=1$ for all $j\in\{1,\ldots, N_1\}$, and there is a path from any $j\in\{1,\ldots, N_1\}$ to the initial type $i$.
\end{itemize}
Then 
\[
\lim_{k\to\infty}\bar{q}^{(k)}_{i}= q_{i}.
\]
In particular, if $\{ \bm{Z}_n \}$ is irreducible, then under Assumptions \ref{tight} and \ref{extbd},
\[
\lim_{k\to\infty}\bm{\bar{q}}^{(k)}= \bm{q}.
\]
\end{theorem}

\color{black}

If $\{\vc Z_n\}$ is irreducible, then \emph{\it(i)} or \emph{\it(ii)} immediately follows, but the converse is not true. Theorem \ref{general converge} therefore holds in many reducible cases too. 
The conditions on  $\tilde{q}_j$ are easy to verify since a simple criterion exists for partial extinction, see \cite{haut12}.

\begin{proof}
By \eqref{CouplingEq}, we have for any fixed $i\geq 1$, $k\geq 1$, and for any arbitrary integer $K\geq 1$,
\begin{eqnarray}\nonumber
\lefteqn{\bar{q}_i^{(k)}-q^{(k)}_i =}\\
& &\mbE_i \left( \left. \left( \vc\alpha^{(k)}\bar{\vc q}^{(k)}\right)^{|\bm{S}_k|} \right| \,0<|\bm{S}_k|<K \right) \mbP_i(0<|\bm{S}_k|<K) \label{seed1ab} \\
&&\quad +\mbE_i \left( \left. \left(\vc\alpha^{(k)}\bar{\vc q}^{(k)}\right)^{|\bm{S}_k|} \right| \,|\bm{S}_k| \geq K \right) \mbP_i(|\bm{S}_k| \geq K). \label{seed2ab}
\end{eqnarray}
{Under Assumption \ref{extbd}, by Lemma \ref{SeedDich} we get that for any $K\geq 1$}, $\mbP_i(0<|\bm{S}_k|<K) \to 0$ as $k \to \infty$, so \eqref{seed1ab} vanishes as $k\to\infty$. It {remains} to show that \eqref{seed2ab} vanishes {as well}. 

By Lemmas \ref{absorb}, \ref{absorb2} and \ref{SeedDich}, 
$\lim_{k \to \infty}\mbP_i(|\bm{S}_k| \geq K)=c_i$ independently of the choice of $K$, where $c_i=\tilde{q}_i-q_i$.
Thus we obtain 
\begin{align} \label{bla}
{ \limsup_{k \to \infty} \left( \bar{q}^{(k)}_i-q^{(k)}_i \right) \leq c_i\,\limsup_{k \to \infty} \left(  \vc\alpha^{(k)}\bar{\vc q}^{(k)} \right)^K\,.}
\end{align} We now prove that $\vc\alpha^{(k)}\bar{\vc q}^{(k)}$ is bounded away from 1 for $k$ sufficiently large whenever $c_i>0$, assuming \textit{(i)} and \textit{(ii)} separately.

 \textbf{Assume that \textit{(i)} holds.} Then there exists $\varepsilon>0$ and {$L \geq 1$} such that for all $j\in\{1,\ldots,N_1\}$, and for all $k \geq L$, $\tilde{q}_j^{(k)}<1-\varepsilon$. Therefore, by Lemma \ref{qint}, {for all $j\in\{1,\ldots,N_1\}$ and $k \geq L$, we have $\bar{q}^{(k)}_j < 1-\varepsilon$. It follows that for any $k>\max\{L,N_1\}$, 
\begin{eqnarray*} \vc\alpha^{(k)}\bar{\vc q}^{(k)} &=&\sum^{N_1}_{j=1} \alpha_j^{(k)} \bar{q}^{(k)}_j +\sum^k_{j=N_1+1} \alpha_j^{(k)} \bar{q}^{(k)}_j \\&<&(1-\varepsilon)\sum^{N_1}_{j=1} \alpha_j^{(k)} +\sum^k_{j=N_1+1} \alpha_j^{(k)} \\&=&\sum^k_{j=1} \alpha_j^{(k)}  -\varepsilon \sum^{N_1}_{j=1} \alpha_j^{(k)} \leq 1-\varepsilon\,a,
\end{eqnarray*}where the last inequality follows from Assumption \ref{tight}. With this, \eqref{bla} becomes
$$\limsup_{k \to \infty} \left( \bar{q}^{(k)}_i-q^{(k)}_i \right) < c_i\,(1-\varepsilon\,a)^K,$$ and the result follows from Corollary \ref{qint} and by choosing $K$ large enough.
%\footnote{Given that we swapped to $\limsup$ should we mention Lemma 1/Corollary 1 here?}.
\medskip

\textbf{Assume that \emph{(ii)} holds.} First observe that if $c_i=0$ for all $i\geq 1$ in \eqref{bla}, then the result immediately follows. In the remainder of the proof we assume that there exists $i\geq 1$ such that $c_i>0$, and we first show that this implies that $c_j>0$ for all $j\in\{1,\ldots,N_1\}$.
Indeed, let $\theta_i$ be the first time an individual of type $i$ is born in $\{\vc Z_n\}$. 
%\footnote{This part always confuses me, how about replacing the paragraph below by: Suppose the process starts with a single individual of type $j$ and then generates a type $i$ descendant before a seed ($\theta_i< \tau_k$). In this case $\bm{Z}_{\theta_i} \geq \bm{e}_i$. By assumption $\tilde{q}_j=1$ for all $j \in \{1, \dots, N_1\}$ which implies killings cannot in the seed process. Consequently, the number of seed produced by the daughter processes of the individuals represented by $\bm{Z}_{\theta_i}$ is equal to or greater than those of $\bm{e}_i$, so that 
%\begin{align}\label{Seed stuff}
%\mbP_j(|\bm{S}_k| \geq K) \geq \mbP_i(|\bm{S}_k| \geq K)\, \mbP_j(\theta_i<\tau_k).
%\end{align}
%By assumption $\mbP_j(\theta_i < \infty)>0$, therefore as $k \to \infty$ Equation \eqref{Seed stuff} becomes {\red $c_j\geq c_i \,\mbP_j(\theta_i<\infty)>0$}, as required.}

Then, by assumption, 
{$\mbP_j(\theta_i<\infty)>0$} for all $j\in\{1,\ldots,N_1\}$. 
Next, we have {
\begin{equation*} \mbP_j(|\bm{S}_k| \geq K) \geq \mbP_j(|\bm{S}_k| \geq K\,|\,\theta_i<\tau_k)\, \mbP_j(\theta_i<\tau_k).\end{equation*} }
In addition, if the process starts with one type-$j$ individual and generates a type-$i$ individual before a seed, then the total number of seeds would be larger than $K$ if the type-$i$ individual itself generates more than $K$ seeds, that is,
\begin{align*}
\mbP_i(|\bm{S}_k| \geq K)\leq \mbP_j(|\bm{S}_k| \geq K\,|\, \theta_i<\tau_k),
\end{align*}
so that 
\begin{align*}
\mbP_j(|\bm{S}_k| \geq K) \geq \mbP_i(|\bm{S}_k| \geq K)\, \mbP_j(\theta_i<\tau_k).
\end{align*}Note that here we use the assumption $\tilde{q}_j=1$ to avoid the possibility of {total} catastrophe in the seed process.
%\footnote{{\red Should we add the sentence `note that here we use $\tilde{q}_j=1$ to avoid the possibility of catastrophe in the seed process'.}}
As $k\to\infty$, the last inequality becomes {$c_j\geq c_i \,\mbP_j(\theta_i<\infty)>0$}, as required.}

Now that we have shown $c_i>0$ for some $i\geq 1$ implies $c_j>0$ for all $j\in\{1,\ldots,N_1\}$, it remains to show that $c_i>0$ implies
$ \vc\alpha^{(k)}\bar{\vc q}^{(k)} $ is bounded away from 1 for $k$ sufficiently large.
{Since $c_j>0$, it follows that} for any $\varepsilon>0$ there exists an integer {$W_j$ depending on $K$ such that for all $k>W_j$, 
\begin{equation}\label{blabla}\mbP_j( | \bm{S}_k | \geq K)>c_j-\varepsilon.\end{equation}
Let {$W =\max_{1\leq j\leq N_1} \{W_j\},$} and $c=\min_{1 \leq j \leq N_1 } \{ c_j \}>0$. With reference to \eqref{ERgen} and \eqref{blabla}, we observe that for any $k \geq \{W, N_1, N_2\}$, the process $\{\bm{Z}_n^{(e,k)}: \varphi_0 \in \{1,\dots,N_1\} \}$ is then stochastically larger than the branching process $\{ \bm{Z}_{n}^{(e,k,2)} \}_{{n\geq 0}}$
%\footnote{\blue I noticed that we are not always consistent when we define our BPs: sometimes it is defined for $n\geq 0$, sometimes for $n\geq 1$.} 
with type set $T_{N_1}:=\{1,\dots,N_1\}$ and progeny generating function $\bm{G}^{(e,k,2)}:[0,1]^{T_{N_1}} \to [0,1]^{T_{N_1}}$ such that, for any $i\in T_{N_1}$,
$$
G^{(e,k,2)}_i(\bm{s}) = \left( \sum^{N_1}_{j=1} \alpha_j^{{(k)}} s_j + 1-\sum^{N_1}_{j=1} \alpha_j^{{(k)}} \right)^K (c-\varepsilon) +1-(c-\varepsilon).
%\footnote{\blue Maybe a short sentence explaining the idea would help the reader here (a bit like the sentence we had in footnote 5 of version 9)? Also, small detail: I did not like the index $l$ so I replaced it by $j$. Another option that looks better is this $\ell$. }
$$
This corresponds to the branching process in which each individual has $K$ offspring with probability $c-\varepsilon$ and 0 offspring otherwise, then the types of the offspring are assigned independently according to the possibly defective distribution $(\alpha^{(k)}_1, \dots, \alpha^{(k)}_{N_1})$ and individuals not assigned a type are immediately killed. Since $G^{(e,k,2)}_i(\cdot)$ is independent of $i$, $\{ |\bm{Z}_{n}^{(e,k,2)} |\}_{{n\geq 0}}$ behaves like a single-type Galton-Watson process, that is, it is locally isomorphic to a single-type branching process (see \cite[Definition 4.2]{Zuc11}). Combining this with the fact that by Assumption \ref{tight}, $\sum_{i=1}^{N_1} \alpha_i^{(k)} \geq a$ for all $k \geq N_2$, we see that, for all $k \geq \max\{W,N_1,N_2\}$, $\{ |\bm{Z}_{n}^{(e,k,2)} |\}$ is stochastically larger than the single-type {branching} process $\{ Z_{n}^{(e,3)} \}_{{n\geq 0}}$ with progeny generating function
$$
G^{(e,3)}(s) = (as+1-a)^K(c-\varepsilon) +1-(c-\varepsilon).
$$
By taking $K>2/(a(c-\varepsilon))$ in order to bound the mean progeny {of $\{ Z_n^{(e,3)} \}$} away from 1, we obtain {$q^{(e,k)}_j\leq q^{(e,k,2)}_j\leq q^{(e,3)}<1-\gamma$} for any $k \geq \max\{ W, N_1, N_2\}$, $j \in \{ 1, \dots , N_1\}$, and for some $\gamma>0$. Using the same argument as the one used when assuming \emph{(i)} holds, we obtain $\vc\alpha^{(k)} {\vc q}^{(e,k)} <1- \gamma a$, and therefore by \eqref{ERineq}, $$  \vc\alpha^{(k)}\bar{\vc q}^{(k)} \leq  \vc\alpha^{(k)} {\vc q}^{(e,k)} <1-\gamma a$$} for $k$ sufficiently large, which proves the result.

\end{proof}

\section{Conditions for $\bm{q}<\bm{1}$ and $\bm{q}=\bm{1}$}

\soph{Theorem \ref{general converge} establishes a relationship between extinction of finite-type branching processes and global extinction of infinite-type branching processes}.
%\footnote{\blue should we highlight somewhere the fact that the truncated and augmented BP are non singular while the truncated-immortal ones are which is the reason why they cannot be exploited to derive global extinction criteria?}
\soph{We now} directly exploit this link and \soph{well-known results} on finite-type branching processes in a first attempt to derive sufficient conditions for $\bm{q}=\bm{1}$ and $\bm{q}<\bm{1}$. 
%Currently these conditions are in terms of the replacement distributions $\{\bm{\alpha}^{(k)}\}$. Establishing a relation between these conditions and intrinsic properties of the original process $\{\vc Z_n\}$ is an ongoing topic of research.
 \soph{Throughout this section we assume that $\{\vc Z_n\}$ and $\{\vc\alpha^{(k)}\}$ satisfy the conditions of Theorem \ref{general converge}.}
%Theorem \ref{general converge} establishes a relationship between extinction in finite and infinite type branching processes and thus enables the vast literature on extinction in finite type BPs to yield new results in the infinite type case. We now exploit this link to provide sufficient conditions for $\bm{q}=\bm{1}$ and $\bm{q}<\bm{1}$. Our results should be compared to [REFs] and [REFs], respectively.

For $k\geq 1,$ let $\tilde{M}^{(k)}$ denote the $k$th north-west truncation of the mean progeny matrix $M$, and let $ \vc x^{(k)}$ be the \soph{$k\times 1$} vector such that ${x}_i^{(k)}=\sum_{j>k} M_{ij}$. Then the matrix $\bar{M}^{(k)}:=\tilde{M}^{(k)}+{\vc  x}^{(k)}\vc\alpha^{(k)}$ is the mean progeny matrix of $\{ \bm{\bar{Z}}^{(k)}_n \}$, and $\bar{\vc q}^{(k)}=\vc 1$ if and only if $\rho(\bar{M}^{(k)})\leq 1$. This leads to a neat sufficient condition for almost sure global extinction.

\begin{corollary}\label{ext_cri}If $\rho(\bar{M}^{(k)})\leq 1$ for infinitely many $k$ then $\vc q=\bm{1}$. 
%\footnote{\blue Could we find an example where Harris' sufficient condition (there must exist a positive integer $n_0$ such that $0< a\leq (M^{n_0})_{ij}\leq b<\infty$) and Tetzlaff condition ($\liminf_{n\to\infty} M^n\vc 1<\infty$) for $\vc q=\vc 1$ are not satisfied, but where the condition of our corollary is satisfied? {\red Both examples in the next section do not satisfy Harris' condition. Tetzlaff's condition is very general but is more difficult to implement than Corollary \ref{ext_cri} as it uses products of infinite matrices.}}
\end{corollary} 

\soph{Corollary \ref{ext_cri} implies} that if $\liminf_k \rho(\bar{M}^{(k)}) < 1$ then $\bm{q}=\bm{1}$. Conversely, one may expect that \soph{$\liminf_k \rho(\bar{M}^{(k)}) > 1$ implies $\bm{q}<\bm{1},$}
however, this is not necessarily the case. Indeed, 
%in demonstrating that the same irreducible mean progeny matrix $M$ may lead to $\bm{q}=\bm{1}$ or $\bm{q}<\bm{1}$,  
\cite[Example 4.4]{Zuc11} corresponds to a case where $\liminf_{k}\rho(\bar{M}^{(k)}) \geq 2$ and $\bm{q}=\bm{1}$. Additional higher moment conditions are therefore required. We impose the following condition.

\begin{assumption}\label{BoundedSM}
There exists $B_1 < \infty$ such that $\mbE_i (| \bm{Z}_1|^2)< B_1$ for all $i\geq 0$.
\end{assumption}

Let $\bm{\bar{v}}^{(k)}$ denote the right Perron-Frobenius eigenvector of $\bar{M}^{(k)}$ and $\bar{A}^{(k)}_{i,j\ell}:=\soph{\partial \bar{G}^{(k)}_i(\vc s)/(\partial s_j\partial s_\ell)|_{\bm{s}=1}}$, \soph{for $1\leq i,j,\ell\leq k$, where $\bar{\vc G}^{(k)}(\vc s)$ is the progeny generating function corresponding to $\{ \bm{\bar{Z}}^{(k)}_n \}$.} 
%Since the sequences $\{\rho(\bar{M}^{(k)})\}$ and $\{\bar{\vc q}^{(k)}\}$ are generally non-monotone, the reciprocal of Corollary \ref{ext_cri} does not necessarily hold; exploiting further properties of {the seed process and of the truncated and augmented branching process} $\{ \bm{\bar{Z}}^{(k)}_n \}$ to derive global extinction criteria is a topic of current research. {It would for example be interesting to investigate under which conditions there exists a sequence $\{\vc\alpha^{(k)}\}$ and an integer $K\geq 1$ such that $\rho(\bar{M}^{(k)})$ increases monotonically and $\{\bar{\vc q}^{(k)}\}$ decreases monotonically for all $k\geq K$. Indeed, in that case, $\vc q=\bm{1}$ if and only if $\lim_{k\rightarrow\infty} \rho(\bar{M}^{(k)}) \leq 1$.}
%Using arguments based on those of \cite{Jof67} and \cite[Section 4]{Spa89}\footnote{\blue Do we need to mention that? Aren't they just the usual argument that everyone uses?} 
We now provide sufficient conditions for $\bm{q}<\bm{1}$. 

\begin{proposition}\label{SuSu}
\soph{Under Assumption \ref{BoundedSM}, if $\{\bm{\alpha}^{(k)}\}$ is such that}
\begin{itemize}
\item[(i)] there exists $B_2<\infty$ independent of $i,j,k$ such that $\bar{v}^{(k)}_j/\bar{v}^{(k)}_i < B_2$ whenever $\bar{M}^{(k)}_{ij}>0$, and
\item[(ii)] \soph{there exists $i\geq 1$} such that $\liminf_{k \to \infty} (\bar{v}^{(k)}_i/\sup_j\{\bar{v}^{(k)}_j \})=b>0$,
\end{itemize}
%\begin{equation}\label{LemSScond}
%\bar{M}^{(k)}_{ij}>0 \,\,\Rightarrow \,\, \frac{\bar{v}^{(k)}_j}{\bar{v}^{(k)}_i }< B_2<\infty \quad  \mbox{ and } \quad \exists \, i \geq 1 \, \, \,\mbox{s.t.} \,\, \, \liminf_{k \to \infty} \frac{\bar{v}^{(k)}_i}{\sup_i\{\bar{v}^{(k)}_i \}}>0,
%\end{equation}
then $\liminf_k \rho(\bar{M}^{(k)}) > 1$ implies $q_i<1$.
%then \eqref{SufSur} holds.\footnote{\soph{in the irreducible case, otherwise we just have $q_i<1$ for some $i$. }}
\end{proposition}
\begin{proof}
Observe that if there exists $\bm{s}^{(k)}$ such that $\bm{\bar{G}}^{(k)}(\bm{s}^{(k)}) \leq \bm{s}^{(k)}$ then $\bm{\bar{q}}^{(k)} \leq \bm{s}^{(k)}$. 
%Thus, if we can construct a sequence $\{ \bm{u}^{(k)} \}$, which in addition to the former, is such that $\limsup_i u^{(k)}_i <1$ for some $i \geq 0$ then the result follows from Theorem ... . For $\theta>0$ we let $\bm{u}^{(k)}=\bm{1}-\frac{\theta \bm{\bar{v}}^{(k)}}{\sup_i \{ \bar{v}^{(k)}_i \}}$.
\soph{Let $c^{(k)}:=1/\sup_j \{ \bar{v}^{(k)}_j \}$.}
By the Taylor expansion formula in \cite[Corollary 3]{Ath93}, we have for any $1\leq i\leq k$ and $0<\theta<1$,
\begin{eqnarray*}
\bar{G}^{(k)}_i (\bm{1}-\theta \bm{\bar{v}}^{(k)} c^{(k)}) &\leq& 1 - \theta c^{(k)} \sum_j \bar{v}^{(k)}_j \bar{M}^{(k)}_{ij} + (\theta c^{(k)})^2 \sum_{j,\ell} \bar{v}_j^{(k)} \bar{v}_\ell^{(k)} \bar{A}^{(k)}_{i,j\ell}\\
&\leq & {1} - \theta \rho(\bar{M}^{(k)}) \bar{v}^{(k)}_i c^{(k)} + ( \theta B_2 \bar{v}_i^{(k)} c^{(k)} )^2 \sum_{j,\ell} \bar{A}^{(k)}_{i,j\ell},
\end{eqnarray*}
where $\sum_{j,\ell} \bar{A}^{(k)}_{i,j\ell}=\mbE_i (|\bm{\bar{Z}}_1^{(k)} |^2 ) =\mbE_i (|\bm{{Z}}_1|^2 ) \leq B_1.$
%\begin{align*}
% \sum_{j,\ell} \bar{A}^{(k)}_{i,j\ell} &= 
% \mbE_i \left( 2\sum_{1 \leq j < \ell \leq k} \bar{Z}^{(k)}_{1,j}\bar{Z}^{(k)}_{1,\ell} + \sum_{j=1}^k \bar{Z}^{(k)}_{1,j} (\bar{Z}^{(k)}_{1,j} -1) \right) \\
%% &\leq  \mbE_i \left( 2\sum_{1 \leq j < \ell \leq k} \bar{Z}^{(k)}_{1,j}\bar{Z}^{(k)}_{1,\ell} + \sum_{j=1}^k \left(\bar{Z}^{(k)}_{1,j}\right)^2\right) \\
% &=\mbE_i \left(\left|\bm{\bar{Z}}_1^{(k)} \right|^2 \right) =\mbE_i \left(\left|\bm{{Z}}_1\right|^2 \right) \leq B_1.
% \end{align*}
 Thus, for any $1<a<\liminf \rho(\bar{M}^{(k)})$ there exists $K<\infty$ such that
\[
\bm{\bar{G}}^{(k)} \left(\bm{1}-\theta \bm{\bar{v}}^{(k)}c^{(k)}\right) \leq \bm{1} - \theta \bm{\bar{v}}^{(k)}c^{(k)} (a - \theta B^2_2 B_1),
\]
for all $k>K$.
If $\theta < \frac{a-1}{B_1B^2_2}$ then $\bm{\bar{q}}^{(k)} <  \bm{1} - \theta \bm{\bar{v}}^{(k)}c^{(k)}$ for all $k \geq K$. By Theorem \ref{general converge} and \soph{(ii)}
%the fact that $\liminf_{k \to \infty} \bar{v}^{(k)}_i/ \sup_j\{\bar{v}^{(k)}_j \}=b>0$ for some $i \geq 1$, 
we then obtain $q_i < 1 - \theta b$. 
%Using the fact that for some $i \geq 1$, $\liminf_{k \to \infty} \bar{v}^{(k)}_i/ \sup_i\{\bar{v}^{(k)}_i \}$ and applying Theorem ... then gives the result. \qed
\end{proof}

{Observe that 
%$\bar{v}^{(k)}_j/\bar{v}_i^{(k)}=\lim_{n \to \infty} (\mbE_j | \bm{\bar{Z}}^{(k)}_n|/\mbE_i | \bm{\bar{Z}}_n^{(k)}|).$ 
%Thus, 
if there exists $\varepsilon>0$ such that $\bar{M}^{(k)}_{ij}>0$ implies $\bar{M}^{(k)}_{ij}>\varepsilon$ for all $k\geq 0$, then 
%$$
%\bar{v}^{(k)}_j/\bar{v}_i^{(k)} \leq \lim_{n \to \infty} (\mbE_j | \bm{\bar{Z}}_n^{(k)}|/\varepsilon \mbE_j | \bm{\bar{Z}}_{n-1}^{(k)}|)=\rho(\bar{M}^{(k)})/\varepsilon.
%$$
$
\bar{v}^{(k)}_j/\bar{v}_i^{(k)} <\rho(\bar{M}^{(k)})/\varepsilon.
$
This means that if, in addition, $\limsup_k \rho(\bar{M}^{(k)})<\infty$ then $(i)$ holds.}
%If $\limsup_k \rho(\bar{M}^{(k)})<\infty$ and there exists $\varepsilon>0$ such that $\bar{M}^{(k)}_{ij}>0$ implies $\bar{M}^{(k)}_{ij}>\varepsilon$ for all $k\geq 0$ then \soph{(i)} holds. \footnote{\soph{I am not sure to see that: I agree that if \textit{for some} $i,j$, $\bar{M}^{(k)}_{ij}>0$ implies $\bar{M}^{(k)}_{ij}>\varepsilon$ for all $k\geq 0$, then $\limsup_k \frac{\bar{v}^{(k)}_j}{\bar{v}^{(k)}_i }<B_2<\infty$ but I don't see why this bound is uniform in $i,j$. Maybe I am misinterpreting (i)}}

Proposition \ref{SuSu} leads naturally to sufficient conditions for the entries of $\bm{q}$ to be uniformly bounded away from 1.

\begin{corollary}\label{UnifBound}
\soph{Under Assumption \ref{BoundedSM},  if $\{\bm{\alpha}^{(k)}\}$ is such that} $0<b \leq \bar{v}_i^{(k)} \leq c<\infty$ for all $k \geq 0$ and $1 \leq i \leq k$, then $\liminf_k \rho(\bar{M}^{(k)}) >1$ implies $\sup_i q_i <1.$
%Suppose Assumption \ref{BoundedSM} holds and there exists a sequence $\{\bm{\alpha}^{(k)}\}$ satisfying Assumption \ref{tight} such that $0<b \leq \bar{v}_i^{(k)} \leq c<\infty$ for all $k \geq 0$ and $1 \leq i \leq k$. Then $$\liminf_k \rho(\bar{M}^{(k)}) >1 \quad \Rightarrow \quad \sup_i q_i <1.$$
%\begin{itemize}
%\item $B>0$ such that $\bm{E}_i(| \bm{Z}_1 |^2)<B$ for all $i \geq 0$ and 
%\item some sequence $\{\alpha^{(k)}\}$ satisfying Assumption ... such that $0<b \leq \bar{v}_i^{(k)} \leq c<\infty$ for all $k \geq 0$ and $1 \leq i \leq k$.
%\end{itemize}
%Then 
%\[
%\liminf_k \rho(\bar{M}^{(k)}) >1 \quad \Rightarrow \quad \sup_i q_i <1.
%\]
\end{corollary}
\begin{proof} 
Following the arguments in the proof of Lemma \ref{SuSu}, there exists $\theta>0$ such that $q_i <1-\theta b/c$ for all $i\geq 1$.
%We have
%\begin{align*}
%\bar{G}^{(k)}_i(\vc1-\theta \bm{\bar{v}}^{(k)}) 
%%&\leq 1- \theta \sum_j \bar{v}^{(k)}_j \bar{M}^{(k)}_{ij} + \theta^2 \sum_{j,\ell} \bar{v}_j^{(k)} \bar{v}_\ell^{(k)} \bar{A}^{(k)}_{i,j\ell} \\
%&\leq 1- \theta \rho(\bar{M}^{(k)}) \bar{v}^{(k)}_i + \theta^2 c^2 \sum_{j,\ell} \bar{A}^{(k)}_{i,j\ell},
%\end{align*}
%which means for any $1<a<\liminf \rho(\bar{M}^{(k)})$ there exists $K<\infty$ such that, 
%$$
% \bm{\bar{G}}^{(k)} ( \vc1 - \theta \bm{\bar{v}}^{(k)} ) \leq \bm{1} - \theta a \bm{\bar{v}}^{(k)} + \theta^2 c^2 B_1 \bm{1}\leq \bm{1} - \theta \bm{\bar{v}}^{(k)}(a - \theta c^2 B_1/ b), \quad \forall k>K.
%$$
%We may then choose $\theta$ small enough to ensure $\left( a - \theta c^2 B_1/ b\right)>1$, in which case $\bar{q}^{(k)}_i<1-\theta b$ for all $k>K$ and $1\leq i \leq k$. Theorem \ref{general converge} then gives the result.
\end{proof}

Theorem 4 of \cite{Spa89} is similar to Corollary \ref{UnifBound}, however it requires $\nu(M)>1$ which is known to be sufficient for $\bm{\tilde{q}}<\bm{1}$.
{Note that repeating the same arguments with the sequence $\{ \bm{q}^{(k)} \}$ instead of $\{ \bar{\bm{q}}^{(k)} \}$ leads to a result similar to \cite[Theorem 4]{Spa89} since the sequence of spectral radii of the mean progeny matrices corresponding to $\{ \bm{Z}^{(k)}_n \}$ converge to $\max \{ 1, \nu(M) \}$ as $k\to\infty$.} %Consequently, following the arguments above, a result similar to \cite{Spa89} is obtained. 
%%%%%%%%%%%%%%%%%%%%PETER RESULT%%%%%%%%%%%%%%%%%%%%%%%%%%%
%{\red If we instead use the sequence $\{ \bm{\tilde{q}}^{(k)} \}$ we obtain a neat sufficient condition for \emph{strong local survival}, that is, $\bm{q}= \bm{\tilde{q}} < \bm{1}$. Let $\bm{v}$ be such that $M \bm{v} = \nu(M) \bm{v}$.
%
%\begin{proposition}\label{SLS}
%Under Assumption \ref{BoundedSM} if $\nu(M)>1$ and $0<b \leq v_i \leq c<\infty$ for all $i \geq 1$ then $\sup_i \tilde{q}_i<1$. If, in addition, $\inf q_i>0$ then $\bm{q}=\bm{\tilde{q}}<\bm{1}$. 
%\end{proposition}
%\begin{proof}
%We obtain $\sup_i \tilde{q}_i<1$ by combining the arguments leading to Corollary \ref{UnifBound}, in this case applied to the equivalent quantities of $\{ \bm{\tilde{Z}}^{(k)}_n \}$, and \cite[Theorem's 6.8 and 6.9]{Sen81}. The fact that $\bm{q} = \bm{\tilde{q}}$ then follows from \cite[Lemma 3.3]{moyal62}.
%\end{proof}
%}\footnote{\red I think we should delete the above passage highlighted in red. }
%%%%%%%%%%%%%%%%%END PETER RESULT%%%%%%%%%%%%%%%%%%%%%%%%%%%%
The primary difference {between Corollary \ref{UnifBound} and \cite[Theorem 4]{Spa89}} therefore {lies in} the conditions `$\liminf_{k} \rho( \bar{M}^{(k)} ) >1$' and `$\nu(M)>1$'. 
{In both Examples 2 and 3, when $\bm{q}<\bm{\tilde{q}}=\bm{1}$, the former is satisfied but the latter is not.}

\section{Examples and relaxations of Assumption \ref{tight}}
%\footnote{\blue What about `Relaxing Assumption 1' or something like that for the title of the section? {\red How about this? I wouldn't mind retaining the word `examples'} Agree!}
\label{other_replace}

Theorem \ref{general converge} proves that $\bm{\bar{q}}^{(k)} \to \bm{q}$ for a large class of replacement distributions $\{ \bm{\alpha}^{(k)} \}$. 
{In this section we demonstrate that when $\{ \bm{\alpha}^{(k)} \}$ is chosen so that Assumption \ref{tight} does not hold, the sequence $\{ \bm{\bar{q}}^{(k)} \}$ exhibits a range of asymptotic behaviours.} Indeed, we show that its limit does not necessarily exist (Example 2), or does not necessarily converge to $\vc q$ (Example 3). {The proofs of the results pertaining to these examples are gathered in \soph{Appendix B}.}
{These results are related to those in \cite{Gib87,Gib87b,Hey91}, where the algorithmic computation of the stationary distribution of a recurrent infinite state Markov chain was considered.}

%However, when $\{ \bm{\alpha}^{(k)} \}$ is chosen so that Assumption \ref{tight} does not hold one may not obtain $\lim_{k \to \infty} \bm{\bar{q}}^{(k)}= \bm{q}$. In fact, as we demonstrate in the next two examples, this limit does not necessarily exist (Example 2), or does not necessarily converge to $\vc q$ (Example 3). {The proofs of the results pertaining to these examples are gathered in \soph{Appendix B}.}
%{These results are related to those in \cite{Gib87,Gib87b,Hey91}, where the algorithmic computation of the stationary distribution of a recurrent infinite state Markov chain was considered.}
%\footnote{\blue We need to state this sentence a bit differently but I think this is a better place for it than in the introduction, where it may give the impression that our results are not new}

\medskip
\noindent
\textbf{Example 2 (Replacement with type $k$).} Let $\bm{\alpha}^{(k)} = \bm{e}_k$ and consider a modified version of the example of \cite[Section 5.1]{haut12}, in which the odd types are `stronger' than the even types. That is, we assume $a,c>0$, $d>1$ and define
$$G_1(\vc s)= \dfrac{cd}{t}\, s_2^t+1- \dfrac{cd}{t},$$ and for $i\geq 2$,
\[
G_i(\vc s)= 
\begin{cases}
\dfrac{cd}{u}\, s_{i+1}^u+\dfrac{ad}{u}\, s_{i-1}^u+1-\dfrac{d(a+c)}{u} & \text{when $i$ is odd, }\\\dfrac{c}{dv}
\,s_{i+1}^v+\dfrac{a}{dv} \,s_{i-1}^v+1-\dfrac{(a+c)}{dv} & \text{when $i$ is even,}
\end{cases}
\]where $t=\lceil dc \rceil +1$, $u = \lceil d(c+a) \rceil +1$ and $v=\lceil (c+a)/d \rceil +1$.

%\[
%p_{1 \bm{j} } = 
%\begin{cases}
%dc/t & \text{ if }\bm{j}= t {\bm{e}_2}, \\
%1-dc/t & \text{ if }\bm{j}=\bm{0},\\
%0 &\text{ otherwise},
%\end{cases}
%\quad \mbox{and} \quad
%p_{i \bm{j} } = 
%\begin{cases}
%dc/u, & \text{ if } \bm{j}= u \bm{e}_{i+1}, \\
%da/u, & \text{ if } \bm{j} = u \bm{e}_{i-1}, \\
%1-d(a+c)/u, & \text{ if } \bm{j}=\bm{0},\\
%0 &\text{ otherwise},
%\end{cases}
%\]
%when $i \neq 1$ is odd, 
%{where $a,c>0$, }and 
%\[
%p_{i \bm{j} } = 
%\begin{cases}
%c/(dv), & \text{ if }\bm{j}= v \bm{e}_{i+1}, \\
%a/(dv), & \text{ if }\bm{j} = v \bm{e}_{i-1}, \\
%1-d(a+c)/v, &\text{ if } \bm{j}=\bm{0},\\
%0 &\text{ otherwise},
%\end{cases}
%\]
%when $i$ is even, 
%where $t=\lceil dc \rceil +1$, $u = \lceil d(c+a) \rceil +1$ and $v=\lceil (c+a)/d \rceil +1$. 
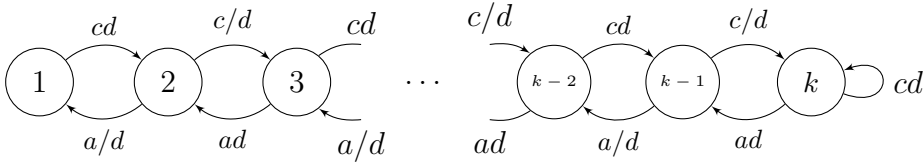
\begin{figure}
\begin{tikzpicture}

\tikzset{vertex/.style = {shape=circle,draw,minimum size=1.3em}}
\tikzset{edge/.style = {->,> = latex'}}
% vertices
\node[vertex, minimum size=.9cm] (1) at  (0,0) {1};
\node[vertex, minimum size=.9cm] (2) at  (1.71,0) {2};
\node[vertex, minimum size=.9cm] (3) at  (3.42,0) {3};

\node at (5.13,0){$\dots$};

\node[vertex, minimum size=.9cm] (k-2) at  (6.84,0) {\tiny $k-2$};
\node[vertex, minimum size=.9cm] (k-1) at  (8.55,0) {\tiny $k-1$};
\node[vertex, minimum size=.9cm] (k) at  (10.26,0) {\small $k$};

%edges

\draw[edge,above] (1) to[bend left=40] node {\footnotesize $cd$ } (2);
\draw[edge,above] (2) to[bend left=40] node {\footnotesize $c/d$ } (3);

\draw[edge,below] (2) to[bend left=40] node {\footnotesize $a/d$ } (1);
\draw[edge,below] (3) to[bend left=40] node {\footnotesize $ad$ } (2);

\draw[edge,above] (k-2) to[bend left=40] node {\footnotesize $cd$ } (k-1);
\draw[edge,above] (k-1) to[bend left=40] node {\footnotesize $c/d$ } (k);

\draw[edge,below] (k) to[bend left=40] node {\footnotesize $ad$ } (k-1);
\draw[edge,below] (k-1) to[bend left=40] node {\footnotesize $a/d$ } (k-2);

\draw[edge,right] (k) to [out=-20,in=20,looseness=6] node {$cd$} (k);

\draw[above] (3) to[bend left=20, pos=1] node {$cd$} (4.275,0.5);
\draw[edge,below] (4.275,-0.5) to[bend left=20,pos=0] node {$a/d$}  (3);

\draw[edge,above] (5.985,0.5) to[bend left=20,pos=0] node {$c/d$}  (k-2);
\draw[below] (k-2) to[bend left=20,pos=1] node {$ad$}  (5.985,-0.5);

%\draw[edge,above] (2) to[bend left=40] node {\footnotesize $b$ } (1);
%\draw[edge,above] (3) to[bend left=45] node {\footnotesize $b^2$ } (1);
%\draw[edge,above] (4) to[bend left=50] node {\footnotesize $b^3$ } (1);
%\draw[edge,above] (5) to[bend left=55] node {\footnotesize $b^4$ } (1);
%
%\node at (9,0) {$\dots$};

%\draw[line width=1.0pt, edge,below] (8) to[bend right=20]  (7);
%\draw[line width=1.0pt,edge,below] (8) to[bend right=10]  (6);
%\draw[line width=1.0pt,edge,below] (8) to[bend right=0]  (5);
%\draw[line width=1.0pt,edge,below] (8) to[bend left=10]  (4);
%\draw[line width=1.0pt,edge,below] (8) to[out=-80, in=-60, looseness=1.5]  (3);
%\draw[line width=1.0pt,edge,below] (8) to[out=-230, in=50, looseness=.5]  (2);
\end{tikzpicture}
\caption{\label{Example2}The mean progeny representation graph corresponding to $\{ \bm{\bar{Z}}^{(k)}_n \}$ when $k$ is odd in Example 2. }
\end{figure}
When $k$ is odd $\{ \bm{\bar{Z}}^{(k)}_n \}$ has the mean progeny representation graph given in Figure \ref{Example2};
%\[
%\bar{M}^{(k)}=\begin{bmatrix}
%0 & dc  &   &    &    &   &	&\\
%a/d & 0 & c/d &    &    &   &	&\\
%   & da & 0 & dc &    &   &	&\\
%   &    & a/d & 0 & c/d &    &	&\\
%   &    &    & \ddots & \ddots & \ddots   &	&\\
%   & 	 &	&	&		a/d &0	& c/d 	\\
%   &	 &	&	&	 0	& da	& dc	 
%\end{bmatrix},
%\]
 when $k$ is even, there is an equivalent graph.
 % is the same except that $\bar{M}_{k,k-1}=a/d$ and $\bar{M}_{k,k}=c/d$. 
  We consider the type-$k$ process {$\{ E_n^{(k)} ( \bm{\bar{Z}}^{(k)}) \}$} embedded with respect to $\{ \bm{\bar{Z}}^{(k)}_n : \varphi_0=k \}$, with mean progeny $m_{ E_n^{(k)} ( \bm{\bar{Z}}^{(k)}) }$ that we denote by $\bar{m}^{(k)}$ for short. The limit of the sequence $\{\bar{m}^{(k)}\}$ does not generally exist, however its limit superior and inferior are finite when $ac\leq 1/4$, as we show in the next lemma.

\begin{lemma}\label{mean_prog_emb}The mean progeny of $\{ E_n^{(k)} ( \bm{\bar{Z}}^{(k)}) \}$ described in Example 2 satisfies
\begin{equation}\label{m bar k} 
\lim_{k \to \infty}  \bar{m}^{(2k+1)}=cd+\frac{1}{2} \left(1 - \sqrt{1-4ac} \right) \; \mbox{and} \; \lim_{k \to \infty}  \bar{m}^{(2k)}=c/d+\frac{1}{2} \left(1 - \sqrt{1-4ac} \right)
\end{equation} when $ac\leq 1/4$, and
$\lim_{k \to \infty}  \bar{m}^{(k)}=+\infty$ when $ac>1/4$.
\end{lemma}

As a consequence of Lemma \ref{mean_prog_emb}, $\lim_{k \to \infty} \bar{m}^{(2k+1)}-\bar{m}^{(2k)}=c(d-d^{-1})$, which indicates it is possible to choose $a$, $c$ and $d$ so that as $k \to \infty$, $\bar{m}^{(k)}$ oscillates between values less than 1 and greater than 1. This observation leads us to the following result.

\begin{proposition}\label{Tridiagonal}Consider the branching process described in Example 2. { Assume that $d>1$ and that $\bm{\alpha}^{(k)} = \bm{e}_k$. Then $\lim_{k \to \infty} \bm{\bar{q}}^{(k)} = \bm{q}$ when $ac>1/4$. Additionally, $\bm{\tilde{q}} = \bm{1}$ if and only if $ac \leq 1/4$, and when this is satisfied,}
\begin{itemize}
\item[{ {\it{(i)}}}] if $d^{-1}>\left( 1 + \sqrt{1-4ac} \right){/2c}$ then $\lim_{k \to \infty} \bm{\bar{q}}^{(k)} = \bm{q}$,
\item[{ {\it{(ii)}}}] if $d^{-1} \leq \left( 1 + \sqrt{1-4ac} \right){/2c} < d$  then
\[
\lim_{k \to \infty} \bm{\bar{q}}^{(2k+1)} = \bm{q} \quad \mbox{and} \quad \lim_{k \to \infty} \bm{\bar{q}}^{(2k)} = \bm{\tilde{q}}= \bm{1},
\]
\item[{ {\it{(iii)}}}] if $d \leq \left( 1 + \sqrt{1-4ac} \right){/2c}$ then $\lim_{k \to \infty} \bm{\bar{q}}^{(k)}=\bm{\tilde{q}}= \bm{q}=\bm{1}$.
\end{itemize}
\end{proposition}

In Figure \ref{Replacek1} we plot $\tilde{q}_1^{(k)}$ (black dashed), $q_1^{(k)}$ (grey dashed) and $\bar{q}^{(k)}_1$ for $\bm{\alpha}^{(k)}=\bm{e}_1$ (solid grey bold), $\bm{\alpha}^{(k)}=\bm{1}/k$ (solid black bold) and $\bm{\alpha}^{(k)}=\bm{e}_k$ (solid fine).
%\footnote{We might either keep colours for the final version, and refer to the colours in the main body of the text, of need to redo the plots with  difference levels of thickness and levels of grey. Also, the legend in Figure 4.1 should appear only once?}.
 In the top two plots we let $a=1/6$ and $c=7/8$, in which case {$ac  < 1/4$} and $\left( 1 + \sqrt{1-4ac} \right){/2c \approx 0.94}$. With this in mind, we choose $d^{-1}=0.95$ (panel (a)) and $d^{-1}=0.93$ (panel (b)). {In agreement with Proposition \ref{Tridiagonal}, for $\bm{\alpha}^{(k)}=\vc e_k$ we observe that $\bar{q}^{(k)}_1 \to q_1$} when $d^{-1}=0.95$, whereas {$\bar{q}^{(2k+1)}_1 \to q_1$ and $\bar{q}^{(2k)}_1 \to \tilde{q}_1=1$ when }$d^{-1}=0.93$. For these values of $a$, $c$ and $d$ it appears that $q_1<1$, which lead us to conclude that {$\liminf_{k \to \infty} \bar{q}_1^{(k)}\neq \limsup_{k \to \infty} \bar{q}_1^{(k)}$ }when $d^{-1}=0.93$ and $\bm{\alpha}^{(k)}=\bm{e}_k$. In panel (c) of Figure \ref{Replacek1} we let $a=1/3$ and $c=13/16$, in which case {$ac > 1/4$} so that $\tilde{q}_1<1$ for any value of $d$. We choose $d=2$ and observe that all sequences converge to $q_1=\tilde{q}_1$. In panel (d) of Figure \ref{Replacek1} we let $a=1/6$, $c=13/16$ and $d=2$, which means {$ac < 1/4$} and $d^{-1} < \left( 1 + \sqrt{1-4ac} \right){/2c \approx 1.03} < d$, {which entails }that for $\bm{\alpha}^{(k)}=\vc e_k$, $\bar{q}^{(2k+1)}_1 \to q_1$ and $\bar{q}^{(2k)}_1 \to \tilde{q}_1$. However, in this case $\tilde{q}_1=q_1=1$ and thus the limit of $\bar{q}^{(k)}$ exists.

\begin{figure}
\begin{tabular}{c}
(a)\\

\includegraphics[scale=.13]{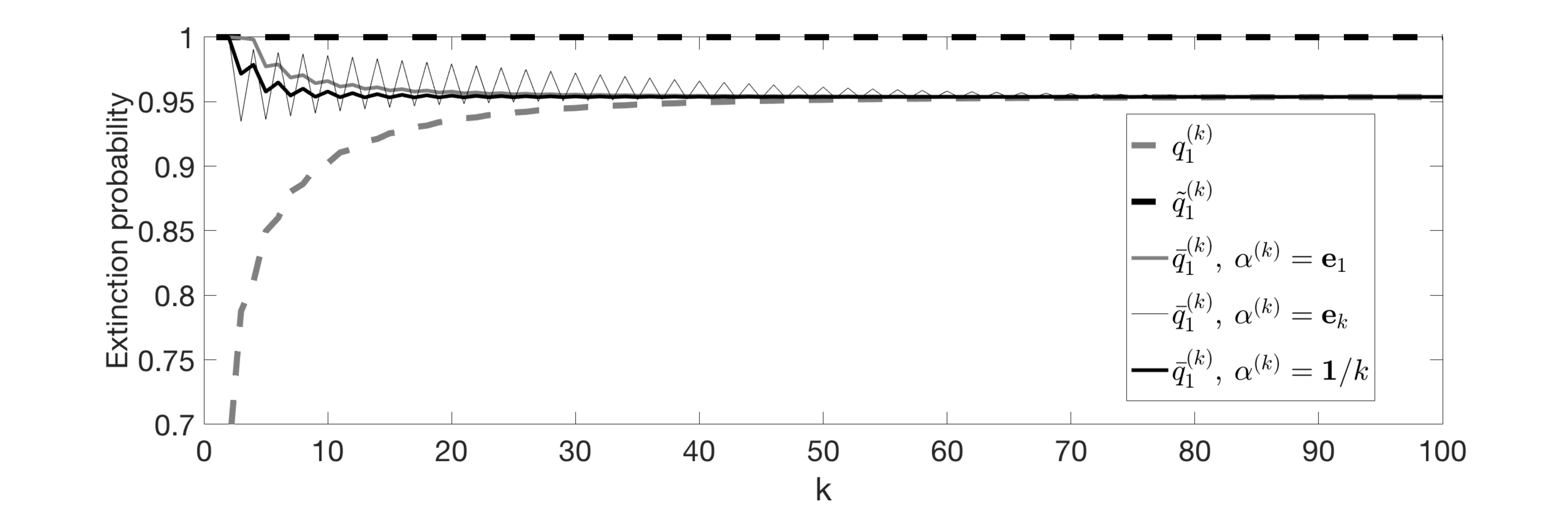}\\
(b)\\

\includegraphics[scale=.13]{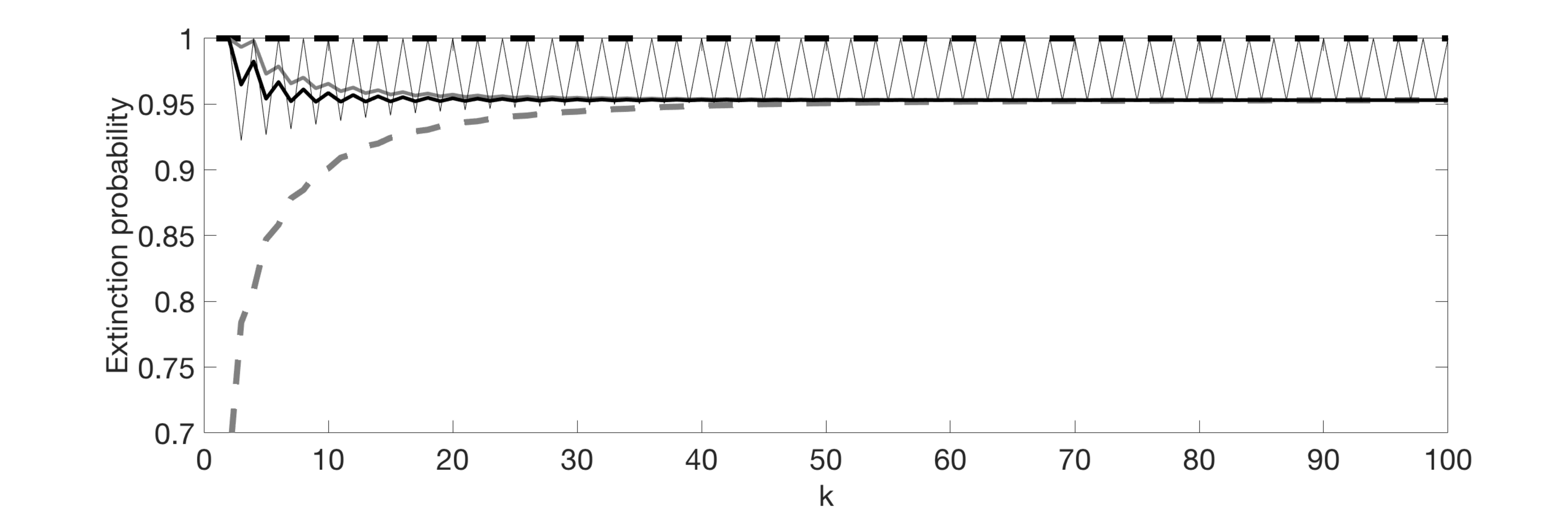}\\
(c)\\

\includegraphics[scale=.13]{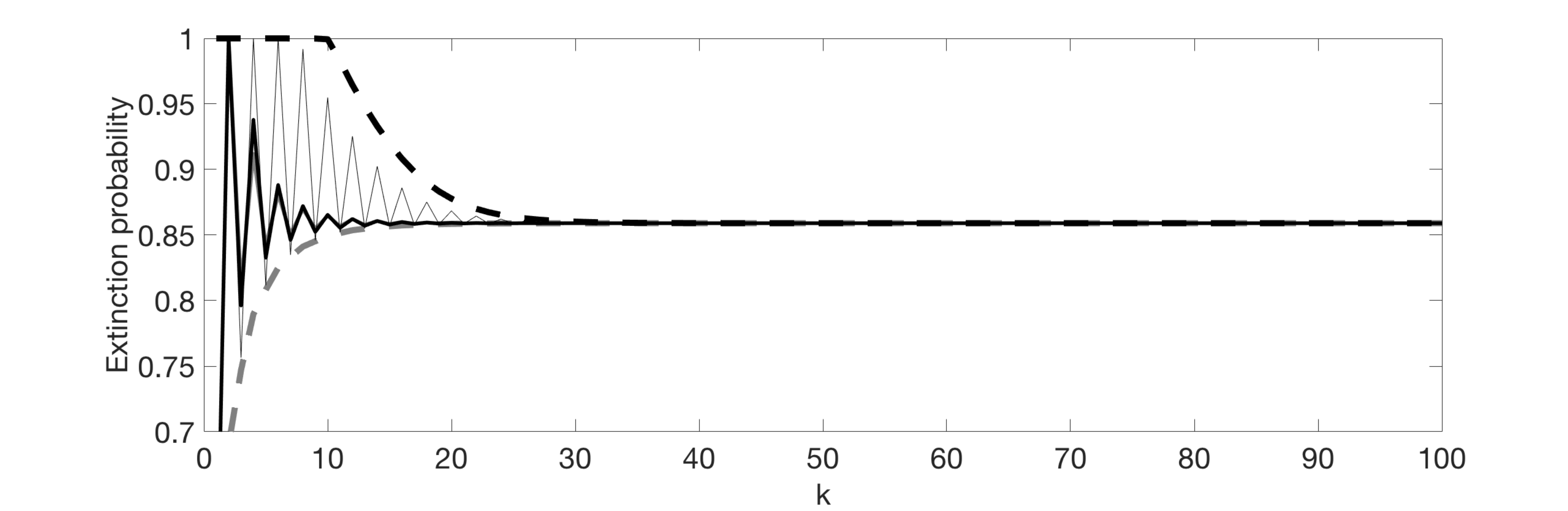}\\
(d)\\

\includegraphics[scale=.13]{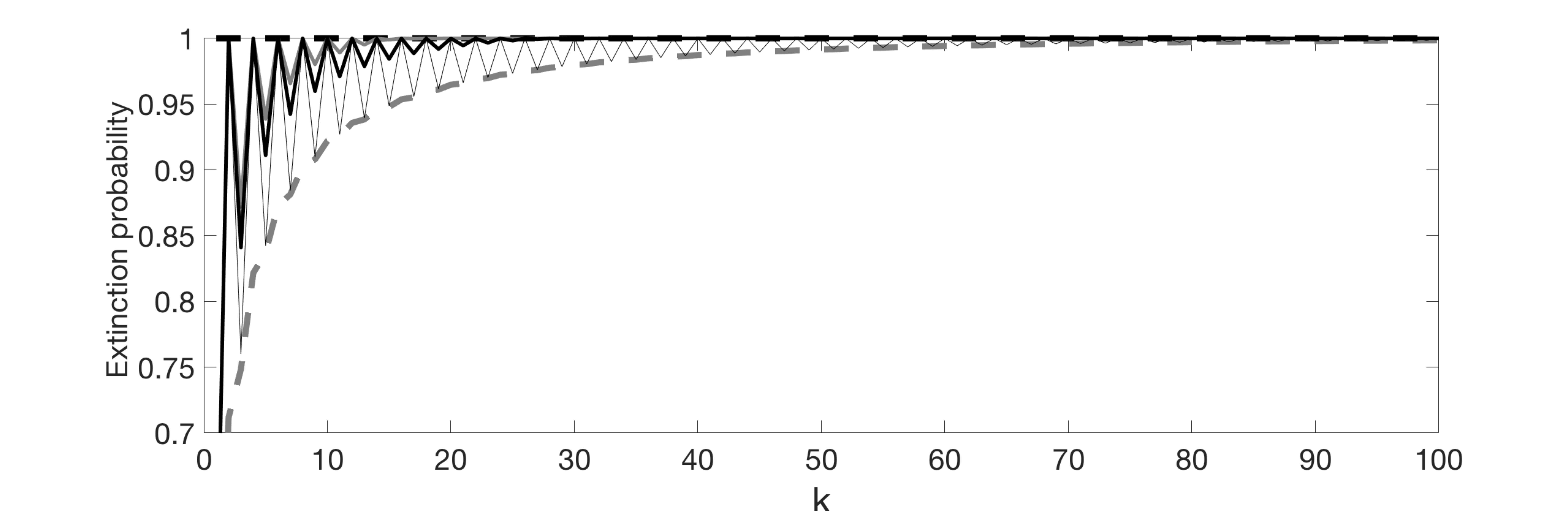}
\end{tabular}
\caption{\label{Replacek1}{Sequences of extinction probabilities $\tilde{q}_1^{(k)}$, $q_1^{(k)}$ and $\bar{q}^{(k)}_1$ for different replacement distributions and different parameters values, corresponding to Example 2. Details are given in the text.}}
\end{figure}

Observe that for the branching process described in {Example 2, Proposition \ref{Tridiagonal} implies} that when $\bm{\alpha}^{(k)}=\bm{e}_k$, then $\liminf_{k \to \infty} \bm{\bar{q}}^{(k)} = \bm{q}$. {The next example indicates that when $\bm{\alpha}^{(k)} = \bm{1}/k$ and $\inf_i q_i >0$, there is not always a subsequence of $\bm{\bar{q}}^{(k)}$ that converges to $\bm{q}$.}

\medskip
\noindent
\textbf{Example 3 (Replacement with a uniform type).} For ease of notation, {in this example} we use the type set $\mathcal{S}=\{2,3,4,\dots \}$. Suppose $p,\varepsilon \in (0,1)$ and $3p \varepsilon^2<1$, and consider the following progeny distribution:
$$G_2(\vc s)=p\,s_4^3+(1-p),$$ and for $i\geq 3$,
\[
G_{i}(\vc s) = 
\begin{cases}
\varepsilon p(1-3p \varepsilon^{i/2})\, s_{i-1}\,s_{2i}^3\\\quad+p(1-3p \varepsilon^{i/2})(1- \varepsilon)\, s_{2i}^3\\\quad\quad+
\varepsilon(1-p(1-3p \varepsilon^{i/2}) )\,s_{i-1}\\\quad\quad\quad+
(1-\varepsilon-p(1-3p \varepsilon^{i/2})(1- \varepsilon)), & \mbox{if } i \in \{ 2^l \}_{l\geq 2},\\\varepsilon\,s_{i-1}+(1-\varepsilon)&  \mbox{otherwise.}
\end{cases}
\]

%\[
%p_{2 \bm{j} } = 
%\begin{cases}
% p, & \mbox{if } \bm{j}=3 \bm{e}_{4} \\
%1-p, &  \mbox{if } \bm{j} = \bm{0},\\0&  \mbox{otherwise,}
%\end{cases}\] and for $i \in \{ 2^l \}_{l\geq 2}$, let
%\[
%p_{i \bm{j} } = 
%\begin{cases}
%\varepsilon p(1-3p \varepsilon^{i/2}), & \mbox{if } \bm{j}= \bm{e}_{i-1}+3 \bm{e}_{2i} \\
%p(1-3p \varepsilon^{i/2})(1- \varepsilon), & \mbox{if } \bm{j} = 3 \bm{e}_{2i} \\
%\varepsilon(1-p(1-3p \varepsilon^{i/2}) ), & \mbox{if } \bm{j}=\bm{e}_{i-1} \\
%1-\varepsilon-p(1-3p \varepsilon^{i/2})(1- \varepsilon), & \mbox{if } \bm{j} = \bm{0},\\0&  \mbox{otherwise,}
%\end{cases}
%\]
%and finally, if $i \notin \{ 2^l \}_{l\geq 2}$, let
%\[ p_{i \bm{j} } = 
%\begin{cases}
%\varepsilon, &\mbox{if } \bm{j}= \bm{e}_{i-1} \\
%1-\varepsilon,& \mbox{if } \bm{j} = \bm{0},\\0&  \mbox{otherwise.}
%\end{cases}
%\]
%In this branching process, the types $i\in \{ 2^l \}_{l\geq 1}$ are the strongest types, which . 
The corresponding mean progeny representation graph is shown in Figure~\ref{Example3}. 

\begin{figure}
\centering
\begin{tikzpicture}

\tikzset{vertex/.style = {shape=circle,draw,minimum size=1.3em}}
\tikzset{edge/.style = {->,> = latex'}}
% vertices
\node[vertex] (2) at  (0,0) {\small 2};
\node[vertex] (4) at  (1.7,0) {\small4};
\node[vertex] (8) at  (4.675,0) {\small8};
\node[vertex] (16) at  (10.2,0) {\scriptsize 16};

\node[vertex] (3) at  (0.85,-0.85) {3};

\node[vertex] (5) at  (2.1335,-1.054) {\small5};
\node[vertex] (6) at  (3.1875,-1.4875) {\small6};
\node[vertex] (7) at  (4.233,-1.054) {\small7};

\node[vertex] (9) at  (4.8875,-1.054) {\small9};
\node[vertex] (10) at  (5.4825,-1.955) {\scriptsize 10};
\node[vertex] (11) at  (6.3835,-2.55) {\scriptsize 11};
\node[vertex] (12) at  (7.4375,-2.7625) {\scriptsize 12};
\node[vertex] (13) at  (8.4915,-2.55) {\scriptsize 13};
\node[vertex] (14) at  (9.3925,-1.955) {\scriptsize 14};
\node[vertex] (15) at  (9.9875,-1.054) {\scriptsize 15};

%\node[vertex] (b) at  (4,3) {2};
%\node[vertex] (c) at  (8,0) {3};
%\node[vertex] (d) at  (4,-3) {$t$};
%\node[vertex] (a1) at (1.5,0) {};
%\node[vertex] (a2) at (3,0) {};
%edges

\draw[edge,above] (2) to node {\footnotesize $3p$ } (4);
\draw[edge,above] (4) to node {\footnotesize $3p(1-3p\varepsilon^2$) } (8);
\draw[edge,above] (8) to node {\footnotesize $3p(1-3p\varepsilon^4$) } (16);
%\draw[edge] (c) to (d);

%\draw[edge] (a)  to[bend left] (a1);
%\draw[edge] (a1) to[bend left] (a);
%
%\draw[edge] (a1) to[bend left] (a2);
%\draw[edge] (a2) to[bend left] (a1);
%
%\path (a2) to node {\dots} (c);
%\node [shape=circle,minimum size=1.5em] (a3) at (4.5,0) {};
%\draw[edge] (a2) to[bend left] (a3);
%\draw[edge] (a3) to[bend left] (a2);
%
%\node [shape=circle,minimum size=1.5em] (c1) at (6.5,0) {};
\draw[edge,below] (4) to[bend left] node[pos=.3,below] {\footnotesize $\varepsilon$ }(3);
\draw[edge,below] (3) to[bend left] node[pos=.7,below] {\footnotesize $\varepsilon$ }(2);

\draw[edge,below] (8) to[bend left=20] node[pos=.4,left] {\footnotesize $\varepsilon$ }(7);
\draw[edge,pos=.4,below] (7) to[bend left=20] node {\footnotesize $\varepsilon$ }(6);
\draw[edge,pos=.6,below] (6) to[bend left=20] node {\footnotesize $\varepsilon$ }(5);
\draw[edge, pos=.6 ,right] (5) to[bend left=20] node {\footnotesize $\varepsilon$ }(4);

\draw[edge,right] (16) to[bend left=12] node {\footnotesize $\varepsilon$ }(15);
\draw[edge,,pos=0.6,right] (15) to[bend left=12] node {\footnotesize $\varepsilon$ }(14);
\draw[edge,pos=0.2,below] (14) to[bend left=12] node {\footnotesize $\varepsilon$ }(13);
\draw[edge,pos=0.4,below] (13) to[bend left=12] node {\footnotesize $\varepsilon$ }(12);
\draw[edge,below] (12) to[bend left=12] node {\footnotesize $\varepsilon$ }(11);
\draw[edge,pos=0.6,below] (11) to[bend left=12] node {\footnotesize $\varepsilon$ }(10);
\draw[edge,pos=0.2,left] (10) to[bend left=12] node {\footnotesize $\varepsilon$ }(9);
\draw[edge,below] (9) to[bend left=12] node[pos=.5,right] {\footnotesize $\varepsilon$ }(8);

\node at (11.05,0) {$\dots$};

%\draw[edge,below] (4) to node[pos=.5,below] {\footnotesize $\varepsilon$ }(3);
%\draw[edge,below] (3) to node[pos=.5,below] {\footnotesize $\varepsilon$ }(2);
%
%\draw[edge,below] (8) to node[pos=.75,above] {\footnotesize $\varepsilon$ }(7);
%\draw[edge,below] (7) to node {\footnotesize $\varepsilon$ }(6);
%\draw[edge,below] (6) to node {\footnotesize $\varepsilon$ }(5);
%\draw[edge,below] (5) to node {\footnotesize $\varepsilon$ }(4);
%
%\draw[edge,below] (16) to node[pos=.75,above] {\footnotesize $\varepsilon$ }(15);
%\draw[edge,below] (15) to node {\footnotesize $\varepsilon$ }(14);
%\draw[edge,below] (14) to node {\footnotesize $\varepsilon$ }(13);
%\draw[edge,below] (13) to node {\footnotesize $\varepsilon$ }(12);
%\draw[edge,below] (12) to node[pos=.75,above] {\footnotesize $\varepsilon$ }(11);
%\draw[edge,below] (11) to node {\footnotesize $\varepsilon$ }(10);
%\draw[edge,below] (10) to node {\footnotesize $\varepsilon$ }(9);
%\draw[edge,below] (9) to node {\footnotesize $\varepsilon$ }(8);
\end{tikzpicture}
\caption{\label{Example3}The mean progeny representation graph corresponding to Example 3.}
\end{figure}

%\newpage
\begin{proposition}\label{TheoremU} 
For the branching process described in Example 3,
\begin{itemize}
\item[(i)] $\bm{\tilde{q}}=\bm{1}$,
\item[(ii)] if $p>1/3$ then $\bm{q}<\bm{1}$, and
\item[(iii)] if $p<2/3$ and $\bm{\alpha}^{(k)}=\bm{1}/k${,} then $\lim_{k \to \infty} \bm{\bar{q}}^{(k)}=\bm{1}$.
\end{itemize}
\end{proposition}

When $p>2/3$, one can show that $\bar{m}^{(k)} \to \infty$; while this implies $\bm{\bar{q}}^{(k)} < \bm{1}$ for all $k$ large enough, it alone does not rule out the case where $\bm{\bar{q}}^{(k)} \to \bm{1}$. We now give a numerical example to further explore the cases $1/3<p<2/3$ and $p>2/3$. 

In Figure \ref{ReplaceU} we plot the first entry of each sequence of extinction probability vectors for the branching processes described in Example 3 {with $\varepsilon=1/2$} and two different values of $p$. {In the upper panel} of Figure \ref{ReplaceU} we let $p=1/2$. In this case, $1/3<p<2/3$ and in agreement with Proposition \ref{TheoremU} {we have $\bar{q}^{(k)}_1 \to \tilde{q}_1 \neq q_1$ for $\bm{\alpha}^{(k)}= \bm{1}/k$.}
{In the lower panel we take} $p=7/9$. In this case, $p>2/3$, and Proposition \ref{TheoremU} does {not} provide any information about the convergence of $\bar{q}^{(k)}_1$, however, {simulations indicate that} $\bar{q}^{(k)}_1 \to q_1 <1$.

\begin{figure}
\centering
\includegraphics[angle=0,width=12.4cm]{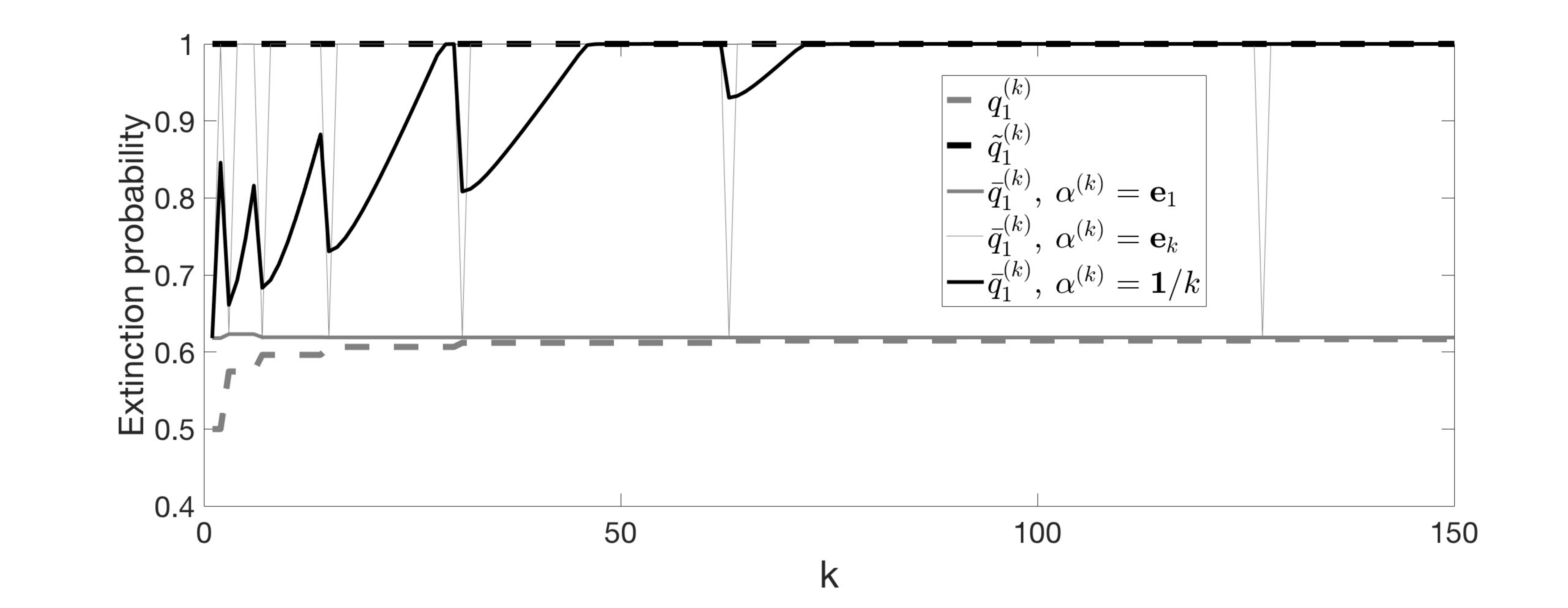}
\includegraphics[angle=0,width=12.4cm]{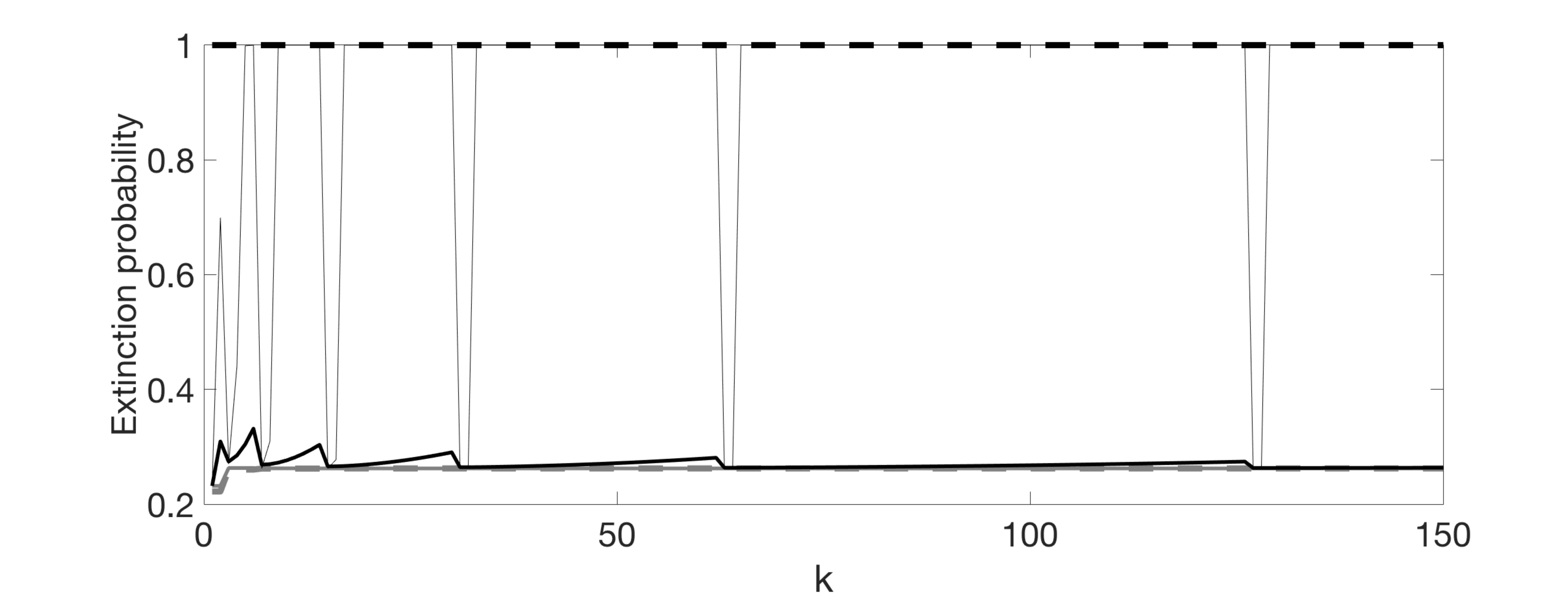}
\caption{\label{ReplaceU} Sequences of extinction probabilities $\tilde{q}_1^{(k)}$, $q_1^{(k)}$ and $\bar{q}^{(k)}_1$ for different replacement distributions and different parameters values, corresponding to Example 3. Details are given in the text.}
\end{figure}

%% The Appendices part is started with the command \appendix;
%% appendix sections are then done as normal sections
%% \appendix

%% \section{}
%% \label{}

%% If you have bibdatabase file and want bibtex to generate the
%% bibitems, please use
%%
%%  \bibliographystyle{elsarticle-num} 
%%  \bibliography{<your bibdatabase>}

%% else use the following coding to input the bibitems directly in the
%% TeX file.
\appendix
\section{Computational aspects}
\soph{The three sequences of extinction probabilities $\{\vc q^{(k)}\}$, $\{\tilde{\vc q}^{(k)}\}$, and $\{\bar{\vc q}^{(k)}\}$ defined in Section \ref{ftbp} are easy to implement in practice, as we show now. Since the convergences $\vc q^{(k)}\to \vc q$, $\tilde{\vc q}^{(k)}\to \tilde{\vc q}$, and $\bar{\vc q}^{(k)}\to {\vc q}$ (under the assumptions of Theorem \ref{general converge}) are pointwise, in order to evaluate the $i$th entry of the desired extinction probability vector, the chosen sequence of approximating vectors should be computed for $k\geq i$.
%Therefore, if one wishes to compute an approximation for the $i$th entry of $\tilde{\vc q}$ for example, the sequence $\{\tilde{\vc q}^{(k)}\}$ actually starts at $k= i$.

For $k\geq 1$, let $\vc s^{[k]}:=(s_1,\ldots,s_k)^\top\in [0,1]^k$ and $\vc G^{[k]}(\vc s):=(G_1(\vc s),\ldots,G_k(\vc s))^\top$, and let $\varepsilon$ be some predetermined tolerance error. For any $i\in\mathcal{S}$, the pseudo-code for the numerical computation of $q_i$ or $\tilde{q}_i$ depends on the function $\vc u(\vc s^{[k]})\in [0,1]^\infty$ as described in \eqref{u1}-\eqref{u3}, which determines which of the three sequences is used:

\medskip

\noindent {Set} $x^{(old)}_i:=2$, $k:=i$\\
{Compute} $\vc x^{(k)}$ as the minimal non-negative solution of $\vc s^{[k]}=\vc G^{[k]}( \vc s^{[k]},\vc u(\vc s^{[k]}))$\\
\textbf{While} $|x_i^{(k)}-x_i^{(old)}|>\varepsilon$ \textbf{do}\\
\phantom{While } $k:=k+1$\\
\phantom{While } Set $\vc x^{(old)}:=\vc x^{(k)}$\\
\phantom{While } Compute $\vc x^{(k)}$ as the minimal non-negative solution of $\vc s^{[k]}=\vc G^{[k]}(\vc s^{[k]},\vc u(\vc s^{[k]}))$\\
\textbf{endwhile}\\
Return $ x_i:=x_i^{(k)},$

\noindent where  
\begin{eqnarray}\label{u1}
\vc u(\vc s^{[k]})=\vc 0 & \rightarrow &\vc x^{(k)}=\vc q^{(k)} \;\textrm{and} \; x_i\approx q_i,\\
\vc u(\vc s^{[k]})=\vc1 & \rightarrow & \vc x^{(k)}=\tilde{\vc q}^{(k)}\;\textrm{and} \; x_i\approx\tilde{q}_i, \\\label{u3}
\vc u(\vc s^{[k]})=\sum_{j=1}^k \alpha_j^{(k)} s_j \vc 1& \rightarrow &\vc x^{(k)}=\bar{\vc q}^{(k)}\;\textrm{and} \;x_i\approx q_i.\end{eqnarray}
% \footnote{\red The $\bm{1}$ and $\bm{0}$ should be interchanged here right? \blue yes, this was to check if you were paying close attention ;)}
Note that the linear functional iteration algorithm or the quadratic Newton algorithm can be applied to compute the minimal non-negative solution of the finite system $\vc s^{[k]}=\vc G^{[k]}( \vc s^{[k]},\vc u(\vc s^{[k]}))$ for each value of $k$.}

\section{Proofs of the results related to Examples 2 and 3}
\begin{proof}[Proof of Lemma \ref{mean_prog_emb}]
We calculate $\bar{m}^{(k)}$ by taking the weighted sum of all first return paths to $k$ in the mean progeny representation graph,
%\begin{align*}
%\bar{m}^{(k)}  &=  \bar{M}^{(k)}_{k,k} + \bar{M}^{(k)}_{k,k-1}\bar{M}^{(k)}_{k-1,k} + \bar{M}^{(k)}_{k,k-1}\bar{M}_{k-1,k-2}\bar{M}^{(k)}_{k-2,k-1}\bar{M}^{(k)}_{k-1,k} \\
%&\quad +\bar{M}^{(k)}_{k,k-1}\bar{M}^{(k)}_{k-1,k-2}\bar{M}^{(k)}_{k-2,k-3}\bar{M}^{(k)}_{k-3,k-2}\bar{M}^{(k)}_{k-2,k-1}\bar{M}^{(k)}_{k-1,k} \\
%& \quad +\bar{M}^{(k)}_{k,k-1}\bar{M}^{(k)}_{k-1,k-2}\bar{M}^{(k)}_{k-2,k-1}\bar{M}^{(k)}_{k-1,k-2}\bar{M}^{(k)}_{k-2,k-1}\bar{M}^{(k)}_{k-1,k} \\
%&\quad +\dots \:.
%\end{align*}
$$
\bar{m}^{(k)}  =  \bar{M}^{(k)}_{k,k} + \bar{M}^{(k)}_{k,k-1}\bar{M}^{(k)}_{k-1,k} + \bar{M}^{(k)}_{k,k-1}\bar{M}_{k-1,k-2}^{(k)}\bar{M}^{(k)}_{k-2,k-1}\bar{M}^{(k)}_{k-1,k}+\ldots.$$
Observe that the number of these paths, $k \to (k-1)  \to \dots \to (k-1)\to k$, with length $2(l+1) \leq 2(k-1)$ is given by the Catalan number,
\begin{equation}\label{Catalan}
C_l=\frac{1}{l+1} { 2l \choose l },\quad l\geq 0,
\end{equation}
whereas, because no path can fall below type 1, the number of paths is less than $C_l$ when $2(l+1) > 2(k-1)$. In these expressions, $(l+1)$ can be interpreted as the total number of negative increments in the paths. Furthermore, the length of each first return path is even, the total number of positive and negative increments of each first return path is equal and each first return path alternates between odd and even states. Hence%\footnote{I do not see where the $ac$ is coming from. Also, why is it raised to the power $l+1$ and not $2(l+1)$?}
,
\[
\bar{M}^{(k)}_{k,k}+\sum^{k-2}_{l=0}C_l \left(ac \right)^{l+1} \leq \bar{m}^{(k)} \leq \bar{M}^{(k)}_{k,k}+\sum^\infty_{l=0} C_l \left(ac \right)^{l+1}.
\]
The infinite series converges when $ac\leq 1/4$ and diverges when $ac>1/4$. In addition, when $ac \leq 1/4$,
\[
\sum^\infty_{l=0} \frac{1}{l+1} { 2l \choose l} \left(ac \right)^l = \frac{1}{2} \left(1 - \sqrt{1-4ac} \right), 
\]
which gives the result. \end{proof}

\begin{proof}[Proof of Proposition \ref{Tridiagonal}]
Following an approach analogous to the proof of {Proposition 5.1 in \cite{haut12},} we can show that the convergence norm of the mean progeny matrix $M$ is $\nu(M)=2 \sqrt{ac}$. By {Proposition 4.1 in \cite{haut12}, we see that} $\bm{\tilde{q}} = \bm{1}$ {if and only if} $ac \leq 1/4$. 

{We first turn our attention to cases {\it{(i)}} and {\it{(ii)}}. O}bserve that the number of first return paths to $k$ of any fixed length is monotone increasing with $k$. This means that the sequences $\{ \bar{m}^{(2k+1)} \}$ and $\{ \bar{m}^{(2k)} \}$ are monotonically increasing with respect to $k$. Due to the repetitive structure of the progeny distributions and the relative weakness of type 1 with respect to other odd types, we also have
%\footnote{More explanations?}
\[
\bar{q}^{(2k+1)}_{2k+1} \geq \bar{q}^{(2k+3)}_{2k+3} \quad \mbox{and} \quad \bar{q}^{(2k)}_{2k} \geq \bar{q}^{(2k+2)}_{2k+2} ,
\]
for all $k \geq1$.
If $\lim_{k \to \infty} \bar{m}^{(2k+1)} >1$, then there exists $\varepsilon_1>0$ and an integer $k_1$ such that {for all $k\geq k_1$, }
\begin{equation}\label{EmbkBound1}
\bar{q}^{(2k+1)}_{2k+1}<1-\varepsilon_1,
\end{equation}
whereas, if $\lim_{k \to \infty} \bar{m}^{(2k+1)} \leq 1$ then, {for all $k \geq 1$,}
\begin{equation}\label{EmbkBound2}
\bar{q}^{(2k+1)}_{2k+1}=1.
\end{equation}
An equivalent result holds when we take the limit over the even values of $k$. 

Next, we have $\inf_{i} q_i  \geq \inf_i p_{i  }(\bm{0})>0$ 
%\footnote{yeah?}
which, by Lemma \ref{SeedDich}, implies that $| \bm{S}_k |$ satisfies the dichotomy property. Therefore, for any arbitrary integer $K\geq 1$,
\begin{equation*}
\limsup_{k \to \infty}\left( \bar{q}^{(k)}_i - q^{(k)}_i\right) ={c_i}\,\limsup_{k \to \infty}  \mbE_i \left( \left. \left( \bar{q}^{(k)}_k \right)^{| \bm{S}_k |} \right| | \bm{S}_k |> K\right)\,,
\end{equation*}
where $c_i=\tilde{q}_i-q_i=\lim_{k \to \infty} \mbP_i ( | \bm{S}_k | > K)$ by Lemmas \ref{absorb}, \ref{absorb2} and \ref{SeedDich}, with the same holding when $\limsup$ is replaced by $\liminf$. In combination with \eqref{EmbkBound1} and \eqref{EmbkBound2} we then obtain
 %\footnote{I am not following: Equations  \eqref{EmbkBound1} and \eqref{EmbkBound2} refer to $q$ and not to $\bar{q}$. How are you using them?},
\[
\lim_{k \to \infty} \left(\bar{q}^{(2k+1)}_i - q^{(2k+1)}_i \right)=
\begin{cases}
0, & \text{if  } \lim_{k \to \infty} \bar{m}^{(2k+1)} > 1, \\
 \tilde{q}_i-q_i, & \text{if  }  \lim_{k \to \infty} \bar{m}^{(2k+1)} \leq 1
\end{cases}
\]
and 
\[
\lim_{k \to \infty} \left(\bar{q}^{(2k)}_i - q^{(2k)}_i \right)=
\begin{cases}
0, & \text{if  } \lim_{k \to \infty} \bar{m}^{(2k)} > 1, \\
\tilde{q}_i-q_i, & \text{if  }  \lim_{k \to \infty} \bar{m}^{(2k)} \leq 1.
\end{cases}
\]
Use of the fact that $\lim_{k \to \infty}  \bm{q}^{(k)} \to \bm{q}$ and Lemma \ref{mean_prog_emb} then provides the result.

{Consider now case {\it{(iii)}}}.
We apply {Proposition 4.5 in \cite{haut12},}  which states that if there exists $\lambda \leq 1$ and a row vector $\bm{x} > \vc 0$ such that $\bm{x} \bm{1} < \infty$ and $\bm{x}M \leq \lambda \bm{x}$, then $\bm{q}=\bm{1}$. We let $\bm{x}=(x_i)_{i\geq 1}$ with $x_i =(\sqrt{d})^{(-1)^{i}}\,x^{1-i}$ for {some} $x>0$. In this case, $\bm{x}M \leq  \bm{x}$ is equivalent to $cx^2+a \leq x$, that is, $x$ {belongs to the interval }$[(1-\sqrt{1-4ac})/2c,(1+\sqrt{1-4ac})/2c]$. {Moreover, $x>1$} ensures $\bm{x} \bm{1} < \infty$. {It follows that whenever $1<(1+\sqrt{1-4ac})/2c$, there exists an $x$ satisfying both conditions, which implies that $\bm{\tilde{q}}=\bm{q}=\bm{1}$. }
\end{proof}

\begin{figure}
\centering
\begin{tikzpicture}

\tikzset{vertex/.style = {shape=circle,draw,minimum size=1.3em}}
\tikzset{edge/.style = {->,> = latex'}}

\node at (1.75,1.5) {(a)};
\node at (7.5,1.5) {(b)};
% vertices

\node[vertex] (2) at  (4.75,0) {\small 2};
\node[vertex] (4) at  (6.75,0) {\small4};
\node[vertex] (8) at  (10.25,0) {\small8};

\node[vertex] (3) at  (5.75,-1) {3};

\node[vertex] (5) at  (7.26,-1.24) {\small5};
\node[vertex] (6) at  (8.65,-1.75) {\small6};
\node[vertex] (7) at  (9.73,-1.24) {\small7};

\node[vertex] (4a) at  (0,0) {\small4};
\node[vertex] (8a) at  (3.5,0) {\small8};

\node[vertex] (5a) at  (0.51,-1.24) {\small5};
\node[vertex] (6a) at  (1.75,-1.75) {\small6};
\node[vertex] (7a) at  (2.98,-1.24) {\small7};

%\node[vertex] (b) at  (4,3) {2};
%\node[vertex] (c) at  (8,0) {3};
%\node[vertex] (d) at  (4,-3) {$t$};
%\node[vertex] (a1) at (1.5,0) {};
%\node[vertex] (a2) at (3,0) {};
%edges

\draw[edge,above] (2) to node {\footnotesize $3p$ } (4);
\draw[edge,above] (4) to node {\footnotesize $3p(1-3p\varepsilon^2$) } (8);
%\draw[edge] (c) to (d);

%\draw[edge] (a)  to[bend left] (a1);
%\draw[edge] (a1) to[bend left] (a);
%
%\draw[edge] (a1) to[bend left] (a2);
%\draw[edge] (a2) to[bend left] (a1);
%
%\path (a2) to node {\dots} (c);
%\node [shape=circle,minimum size=1.5em] (a3) at (4.5,0) {};
%\draw[edge] (a2) to[bend left] (a3);
%\draw[edge] (a3) to[bend left] (a2);
%
%\node [shape=circle,minimum size=1.5em] (c1) at (6.5,0) {};
\draw[edge,below] (4) to[bend left] node[pos=.3,below] {\footnotesize $\varepsilon$ }(3);
\draw[edge,below] (3) to[bend left] node[pos=.7,below] {\footnotesize $\varepsilon$ }(2);

\draw[edge,below] (8) to[bend left=20] node[pos=.4,left] {\footnotesize $\varepsilon$ }(7);
\draw[edge,pos=.4,below] (7) to[bend left=20] node {\footnotesize $\varepsilon$ }(6);
\draw[edge,pos=.6,below] (6) to[bend left=20] node {\footnotesize $\varepsilon$ }(5);
\draw[edge, pos=.6 ,right] (5) to[bend left=20] node {\footnotesize $\varepsilon$ }(4);

\draw[line width=1.0pt, edge,below] (8) to[bend right=20]  (7);
\draw[line width=1.0pt,edge,below] (8) to[bend right=10]  (6);
\draw[line width=1.0pt,edge,below] (8) to[bend right=0]  (5);
\draw[line width=1.0pt,edge,below] (8) to[bend left=10]  (4);
\draw[line width=1.0pt,edge,below] (8) to[out=-80, in=-60, looseness=1.5]  (3);
\draw[line width=1.0pt,edge,below] (8) to[out=-230, in=50, looseness=.5]  (2);

\draw[line width=1.0pt,edge,left] (8) to [out=20,in=-20,looseness=6] node {} (8);

\draw[edge,below] (8a) to[bend left=20] node[pos=.4,left] {\footnotesize $\varepsilon$ }(7a);
\draw[edge,pos=.4,below] (7a) to[bend left=20] node {\footnotesize $\varepsilon$ }(6a);
\draw[edge,pos=.6,below] (6a) to[bend left=20] node {\footnotesize $\varepsilon$ }(5a);
\draw[edge, pos=.6 ,right] (5a) to[bend left=20] node {\footnotesize $\varepsilon$ }(4a);

\draw[edge,above] (4a) to node {\footnotesize $3p(1-3p\varepsilon^2$) } (8a);

\draw[edge,above] (4a) to [out=70,in=110,looseness=6] node {$\tilde{m}^{(4)}$} (4a);
%\draw[left] (4a) to [out=160,in=200,looseness=6] node {$\tilde{m}^{(4)}$} (4a);

%\draw[edge,below] (4) to node[pos=.5,below] {\footnotesize $\varepsilon$ }(3);
%\draw[edge,below] (3) to node[pos=.5,below] {\footnotesize $\varepsilon$ }(2);
%
%\draw[edge,below] (8) to node[pos=.75,above] {\footnotesize $\varepsilon$ }(7);
%\draw[edge,below] (7) to node {\footnotesize $\varepsilon$ }(6);
%\draw[edge,below] (6) to node {\footnotesize $\varepsilon$ }(5);
%\draw[edge,below] (5) to node {\footnotesize $\varepsilon$ }(4);
%
%\draw[edge,below] (16) to node[pos=.75,above] {\footnotesize $\varepsilon$ }(15);
%\draw[edge,below] (15) to node {\footnotesize $\varepsilon$ }(14);
%\draw[edge,below] (14) to node {\footnotesize $\varepsilon$ }(13);
%\draw[edge,below] (13) to node {\footnotesize $\varepsilon$ }(12);
%\draw[edge,below] (12) to node[pos=.75,above] {\footnotesize $\varepsilon$ }(11);
%\draw[edge,below] (11) to node {\footnotesize $\varepsilon$ }(10);
%\draw[edge,below] (10) to node {\footnotesize $\varepsilon$ }(9);
%\draw[edge,below] (9) to node {\footnotesize $\varepsilon$ }(8);
\end{tikzpicture}
\caption{\label{Example3a}A visual representation of the mean progeny matrix in Example 3. Bold edges in (b) have weight $\frac{3p(1-3p \varepsilon^4)}{7}$.}
\end{figure}
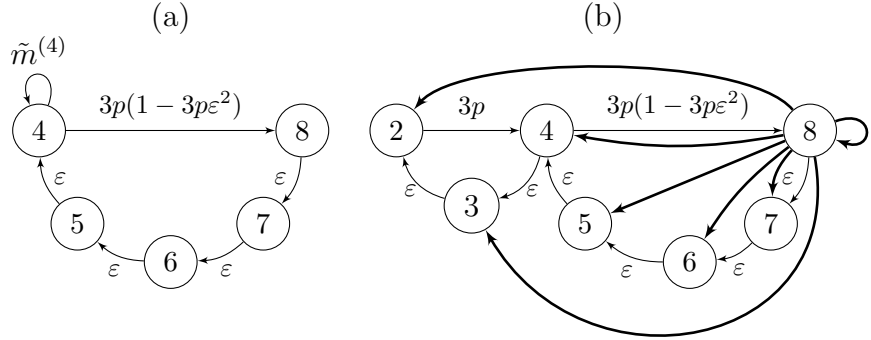

\begin{proof}[Proof of Proposition \ref{TheoremU}]
To prove \emph{(i)}, we consider the type-$2^{k}$ process embedded with respect to $\{ \bm{\tilde{Z}}^{(2^{k})}_n : \varphi_0= 2^{k}\}$, and calculate its mean number of offspring, denoted by $\tilde{m}^{(2^k)}$, for $k\geq 1$. We tackle this by computing the weighted sum of all first return paths to node $2^k$ in the mean progeny representation graph illustrated in Figure \ref{Example3}, which we alter by removing all nodes greater than $2^k$ to account for the corresponding types being sterile in $\{ \bm{\tilde{Z}}^{(2^{k})}_n \}$. Observe that when $k=2$, $(2^{k}=4)$ there is a single first return path $4 \to 3 \to 2 \to 4$, thus, $\tilde{m}^{(4)}=3p \varepsilon^2$. When $k>2$, we calculate $\tilde{m}^{(2^k)}$ recursively as follows. First observe that each first return path to $2^k$ begins with the sequence of edges $2^k \to (2^k-1) \to \dots \to (2^{k-1}+1) \to 2^{k-1}$, and ends with the edge $2^{k-1} \to 2^k$. Additionally, the remainder of each path (or the midsection) can be partitioned into the first return paths that were summed to obtain $\tilde{m}^{(2^{k-1})}$. That is, for the purpose of calculating $\tilde{m}^{(2^k)}$, in the mean progeny matrix representation graph, nodes of type $<2^{k-1}$ can then be replaced by a loop to type $2^{k-1}$ with weight $\tilde{m}^{(2^{k-1})}$. See Figure \ref{Example3a}(a) for an illustration when $k=3$. It can then be shown that,
\begin{align*}
\tilde{m}^{(2^{k})} &= \varepsilon^{2^{k-1}} \left[ 1 + \tilde{m}^{(2^{k-1})} + (\tilde{m}^{(2^{k-1})} )^2 + \dots \right] 3p(1-3p \varepsilon^{2^{k-2}}).
\end{align*}
We can then prove by induction that $\tilde{m}^{(2^k)}=\varepsilon^{2^{k-1}}3p$, which leads to $\tilde{m}^{(2^k)}<1$ for all $k\geq 1$ since, by assumption,  $3p \varepsilon^{(2^{k-1})}\leq 3p\varepsilon^2<1$. Combining this with the fact that for all $k \geq 2$, $\{ \bm{\tilde{Z}}^{(2^k)}_n \}$ is irreducible, we obtain $\bm{\tilde{q}}^{(2^k)}=\bm{1}$ for all {$k \geq 1$. Since $\{ 2^k\}_{k \geq 1}$} is an infinite subsequence of $\mbN$, the result then follows from the fact that $\lim_{k \to \infty} \bm{\tilde{q}}^{(k)} = \bm{\tilde{q}}$.

We now prove \emph{(ii)}. 
%First assume that $p\leq1/3$. 
%%In this case, $\sup_i (M\vc 1)_i<1$ (is that true?), which leads to $\vc q=1$ by Moyal, Theorem 3.2.
%Note that $\inf_i \bm p_{i \bm{0}} > 0$, which implies that $\{|\bm{Z}_n| \}$ satisfies the dichotomy property. \textit{The following argument still needs some work as it is not fully true:}
%Furthermore, the expected number of offspring from each type is less than 1, that is, $M \bm{1} \leq \bm{1}$ (\textit{this is what is not completely true}). Consequently, $\limsup_{k \to \infty} M^n \bm{1} < \bm{1}$. Combining this with the dichotomy of $\{|\bm{Z}_n|\}$, we then apply \cite[Theorem ...]{Tet04} to obtain $\bm{q}=\bm{1}$. 
If $p>1/3$, then there exists $\gamma>0$ such that $p=1/3+\gamma$, and there exist an integer $N$ and a constant $0<C<3\gamma$ such that $3p(1- 3p\varepsilon^{i/2})=(1+3\gamma)(1- 3p\varepsilon^{i/2})>1+C$ for all $i \geq 2^N$ (since $3p\varepsilon^{i/2}$ becomes arbitrarily close to 0 as $i$ increases).
%  We can choose $N \in \mbN$ such that $3p(1-3p\varepsilon^{i/2})=(1+3\gamma)(1-3p\varepsilon^{i/2})>1+\delta$ for all $i \geq 2^N$ for some $\delta>0$.
 By disregarding all types $j$ such that $j \leq 2^N$ or $j \notin \{2^k\}_{k \geq 2}$ it can be shown that $\{ |\bm{Z}_n| : \varphi_0=2^N \}$ is stochastically greater than the Galton-Watson process with progeny generating function $G(s)=(1/3+C/3)s^3+(2/3-C/3)$. Since $G'(1)>1$, we have $q_{2^N}<1$. The result follows from irreducibility.   
 
 To prove \emph{(iii)} we consider the type-$2^k$ process embedded in $\{ \bm{\bar{Z}}^{(2^k)}_n : \varphi_0=2^k\}$ and calculate its mean number  of offspring {$\bar{m}^{(2^k)}$}.  To account for the instantaneous replacement of all individuals of type $>2^k$ with a type {uniformly distributed on $\{2,\dots,2^k \}$}, the graph illustrated in Figure \ref{Example3} is altered by removing all nodes greater than $2^k$ and adding an edge of weight $3p(1-3p\varepsilon^{2^{k-1}})/(2^k-1)$ from node $2^k$ to all the remaining nodes. See Figure \ref{Example3a}(b) for an illustration when $k=3$. We then {calculate} the weighted sum of all first return paths to $2^k$. Note that these paths include those involved in the computation of $\tilde{m}^{(2^k)}$. More specifically, if we let $f_{i,k}$ be the weighted sum of all first passage paths from $2^i$ to $2^k$, then
\[
\bar{m}^{(2^k)}= \tilde{m}^{(2^k)}+ \frac{3p(1-3p\varepsilon^{2^{k-1}})}{2^k-1} \left( 1+ \sum^{k-1}_{j=1} f_{k-j,k} \left( \sum^{2^{k-j}-1}_{i=0} \varepsilon^i \right)\right).
\]
By applying a recursive argument analogous to the proof of \emph{(i)}, it can then be shown that $f_{i,k}=(3p)^{k-i}$. Assuming $p \neq1/3$ we have,
\begin{align*}
\bar{m}^{(2^k)} &= 3p \varepsilon^{2^{k-1}} + \frac{3p(1-3p\varepsilon^{2^{k-1}})}{2^k-1} \left(  1+ \frac{1}{1-\epsilon} \sum^{k-1}_{j=1} (3p)^j (1- \varepsilon^{2^{k-j}}) \right) \\
&\leq 3p \varepsilon^{2^{k-1}} + \frac{3p}{2^k-1} \left( \frac{1-(3p)^k}{(1-3p)(1-\varepsilon)} + \epsilon \right)\to 0\quad \mbox{as $k\to\infty$} 
\end{align*} 
when $p<2/3$. When $p=1/3$ it can be shown that an equivalent result holds. This demonstrates that $\lim_{k \to \infty} \bar{m}^{(2^k)}=0$ when $p<2/3$. The proof that $\bar{m}^{(k)} \to 0$ is an extension of the same method and is omitted. Therefore, there exists an integer $N$ such that for all $k>N$, $\bar{m}^{(k)}<1$ hence $\bar{\vc q}^{(k)}=\vc 1$, which shows that $\bar{\vc q}^{(k)}\to \vc 1$ as $k\to\infty$. 
%(Observe that we are actually also in the conditions of Corollary... which implies that $\bar{\vc q}^{(k)}\to \tilde{\vc q} $ as $k\to\infty$). 
\end{proof}

\section*{Acknowledgements}{The authors are grateful to the anonymous reviewers for their constructive comments, which helped us to improve the manuscript.}
The authors would like to acknowledge the support of the Australian Research Council (ARC) through the Centre of Excellence for the Mathematical and Statistical Frontiers (ACEMS). Sophie Hautphenne would further like to thank the ARC for support through Discovery Early Career Researcher Award DE150101044.

\end{document}